\documentclass[a4paper, 11pt, reqno]{amsart}

\usepackage{latexsym,amsmath,amssymb,amsfonts}
\usepackage{color}
\usepackage{graphicx}
\usepackage{float}
\usepackage{enumerate}
\usepackage{hyperref}
\usepackage{cleveref}
\usepackage{bbm}

\hypersetup{
  colorlinks = true,
  linkcolor = blue,
  citecolor = red
}

\theoremstyle{plain}
\begingroup
\newtheorem{theorem}{Theorem}[section]
\newtheorem{lemma}[theorem]{Lemma}
\newtheorem{proposition}[theorem]{Proposition}
\newtheorem{corollary}[theorem]{Corollary}
\endgroup
\theoremstyle{definition}
\begingroup
\newtheorem{definition}[theorem]{Definition}
\newtheorem{remark}[theorem]{Remark}

\endgroup
\theoremstyle{remark} 

\newcommand{\N}{\mathbb N}

\newcommand{\R}{\mathbb R}
\newcommand{\C}{\mathbb C}

\newcommand{\supp}{{\rm supp}}

\newcommand{\dd}{\mathrm{d}}

\newcommand{\dx}{\, \dd x}

\newcommand{\e}{\varepsilon}

\renewcommand{\o}{\Omega}

\newcommand{\f}{\mathcal{F}}

\newcommand{\average}{{\mathchoice {\kern1ex\vcenter{\hrule height.4pt
width 6pt
depth0pt} \kern-9.7pt} {\kern1ex\vcenter{\hrule height.4pt width 4.3pt
depth0pt}
\kern-7pt} {} {} }}

\newcommand{\blue}[1]{{\color{black} {#1}}}

\usepackage[skip=0pt]{subcaption}

\usepackage[space]{grffile}
\usepackage{bbm}
\usepackage{soul}
\usepackage[left=1in, right=1in]{geometry}

\newcommand{\dmo}{\mathcal{M}_\mathrm{d}(\Omega)}
\newcommand{\Pd}{\mathcal{P}_\mathrm{d}(\Omega)}
\newcommand{\dmr}{\mathcal{P}_\mathrm{d}(\R^+)}
\newcommand{\pr}{\mathcal{P}(\R^+)}
\renewcommand{\e}{\mathcal{E}}
\newcommand{\ca}{\mathbbm{1}}

\allowdisplaybreaks

\title[Asymptotic optimality of the triangular lattice]{Asymptotic optimality of the triangular lattice for a class of optimal location problems}

\author{David P.~Bourne, Riccardo Cristoferi}
\address{Maxwell Institute for Mathematical Sciences and Department of Mathematics, Heriot-Watt University, UK}
\email {d.bourne@hw.ac.uk} 
\address{Radboud University \\ Department of Mathematics \\ Nijmegen \\ Netherlands}
\email{riccardo.cristoferi@ru.nl}
\subjclass[2010]{}
\date{\today}

\begin{document}

\maketitle

\begin{abstract}
We prove an asymptotic crystallization result in two dimensions for a class of nonlocal particle systems.
To be precise, we consider the best approximation with respect to the 2-Wasserstein metric of a given absolutely continuous probability measure $f  \dd x$ by a discrete probability measure $\sum_i m_i \delta_{z_i}$, subject to a constraint on the particle sizes $m_i$.
The locations $z_i$ of the particles, their sizes $m_i$, and the number of particles are all unknowns of the problem. We study a one-parameter family of constraints. 
This is an example of an optimal location problem (or an optimal sampling or quantization problem) and it has applications in economics, signal compression, and numerical integration.
We establish the asymptotic minimum value of the (rescaled) approximation error as the number of particles goes to infinity. In particular, we show that for the constrained best approximation of the Lebesgue measure by a discrete measure, the discrete measure whose support is a triangular lattice is asymptotically optimal. In addition, we prove an analogous result for a problem where the constraint is replaced by a penalization. These results can also be viewed as the asymptotic optimality of the hexagonal tiling for an optimal partitioning problem.
They generalise the crystallization result of Bourne, Peletier and Theil (Communications in Mathematical Physics, 2014) from a single particle system to a class of particle systems, and prove a case of a conjecture by Bouchitt\'{e}, Jimenez and Mahadevan (Journal de Math\'ematiques Pures et Appliqu\'ees, 2011).
\blue{Finally, we prove a crystallization result which states that optimal configurations with energy close to that of a triangular lattice are geometrically close to a triangular lattice.}
\end{abstract}


\section{Introduction}

Consider the problem of approximating an absolutely continuous probability measure by a discrete probability measure.
To quantify the quality of the approximation, we measure the approximation error in the 2-Wasserstein metric.
Let $\Omega\subset\R^d$ be the closure of an open and bounded set, and let
\begin{equation}
\label{eq:f_ass}
f\in L^1(\Omega), \quad \quad f\geq c > 0, \quad \quad \int_\Omega f(x) \, \dd x=1,
\end{equation}
 be the density of the absolutely continuous probability measure.
We approximate $f  \dd x$ by a discrete measure from the set
\[
\Pd := \left\{ \mu = \sum_{i=1}^{N_\mu} m_i \delta_{z_i} : N_\mu \in \mathbb{N}, \, m_i > 0, \,
	\sum_{i=1}^{N_\mu} m_i = 1, \, z_i \in \Omega, \, z_i \ne z_j \textrm{ if } i \ne j \right\}.
\]
For $\mu \in \Pd$, we define $N_\mu := \# \mathrm{supp}(\mu)$, which is not fixed a priori.
For $p\geq1$, the $p$-Wasserstein distance (see \cite{S,VILL1}) between $f \dd x$ and $\mu \in \Pd$ is
\newpage
\begin{multline}
\label{eq:Wassdist_p}
W_p(f,\mu) \:= \\ \inf\left\{ \int_\Omega |x-T(x)|^p f(x)\,\dd x \,:\, 
	T:\Omega\to \{z_i\}_{i=1}^{N_\mu} \text{ is Borel},\, \int_{T^{-1}(\{z_i\}) }f(x) \,\dd x = m_i \,\, \forall \, i \right\}^{\frac 1p}.
\end{multline}
Observe that
\[
\inf\{ W_p(f,\mu): \mu \in \Pd \}=0
\]
since there exists a sequence of discrete measures $\mu_n$ converging weakly$^*$ to $f \dd x$, with $N_{\mu_n} \to \infty$ as $n \to \infty$. 
On the other hand, for each $N\in\N$, $\inf\{ W_p(f,\mu): \mu \in \Pd, \, N_\mu\leq N \} > 0$. Therefore the problem $\inf\{ W_p(f,\mu): \mu \in \Pd \}$ has no solution.
To obtain a minimizer we must constrain the number of atoms $N_\mu$, either explicitly (with a constraint) or implicitly (with a penalization).
Given an \emph{entropy} $H:\Pd\to[0,\infty]$ (defined below) we consider the \emph{constrained} optimal location problem
\begin{equation}\label{eq:minpbconstr}
\inf\left\{ W^p_p(f,\mu) : \mu\in\Pd,\,H(\mu)\leq L \right\} =: \mathcal{E}^{p,d}_H(L),
\end{equation}
where $L>0$, and the \emph{penalized} optimal location problem
\begin{equation}\label{eq:minpbpenal}
\inf\left\{ W^p_p(f,\mu) + \delta H(\mu) : \mu\in \Pd \right\} =: \mathcal{F}^{p,d}_H(\delta),
\end{equation}
where $\delta>0$.
If $H$ satisfies $H(\mu)\to \infty$ as $N_\mu \to \infty$,
then minimising sequences for problems \eqref{eq:minpbconstr} and \eqref{eq:minpbpenal} have a uniformly bounded number of atoms.
If in addition $H$ is lower semi-continuous with respect to the weak$^*$ convergence of measures, then problems \eqref{eq:minpbconstr} and \eqref{eq:minpbpenal} admit a solution.

When $L$ or $\delta$ are fixed, the geometry of the set $\Omega$ has a strong effect on optimal particle arrangements, and it is very difficult to characterise minimising configurations.
As $L$ increases, or $\delta$ decreases, the optimal number of particles $N_\mu$ increases, and it is believed that
optimal configurations locally form regular, periodic patterns; see the numerical evidence in Figures \ref{fig:numerics} and \ref{fig:numerics2}.
This phenomenon is known as \emph{crystallization} (see Section \ref{sec:other} for more on this).
The specific geometry of these patterns depends on the choice of $p$ in the Wasserstein distance, the choice of $H$, and the dimension $d$.
In this paper we will study the crystallization problem by taking the limits $L\to\infty$ and $\delta\to0$.

\begin{figure}
\centering
\begin{subfigure}[t]{0.31\linewidth}
	\centering
	\subcaption*{\footnotesize $\alpha=0.1$, $\delta=2.243 \times 10^{-3}$}
	\includegraphics[width=\linewidth]{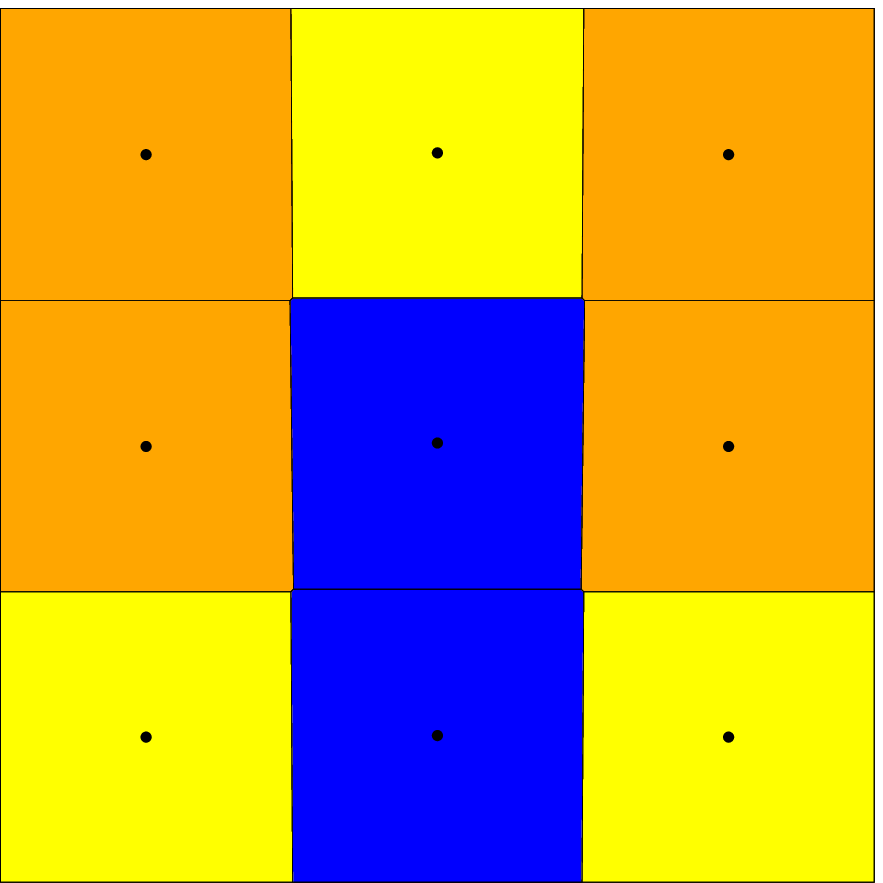}
\end{subfigure}
\begin{subfigure}[t]{0.31\linewidth}
	\centering
	\subcaption*{\footnotesize $\alpha=0.1$, $\delta=1.303 \times 10^{-4}$}
	\includegraphics[width=\linewidth]{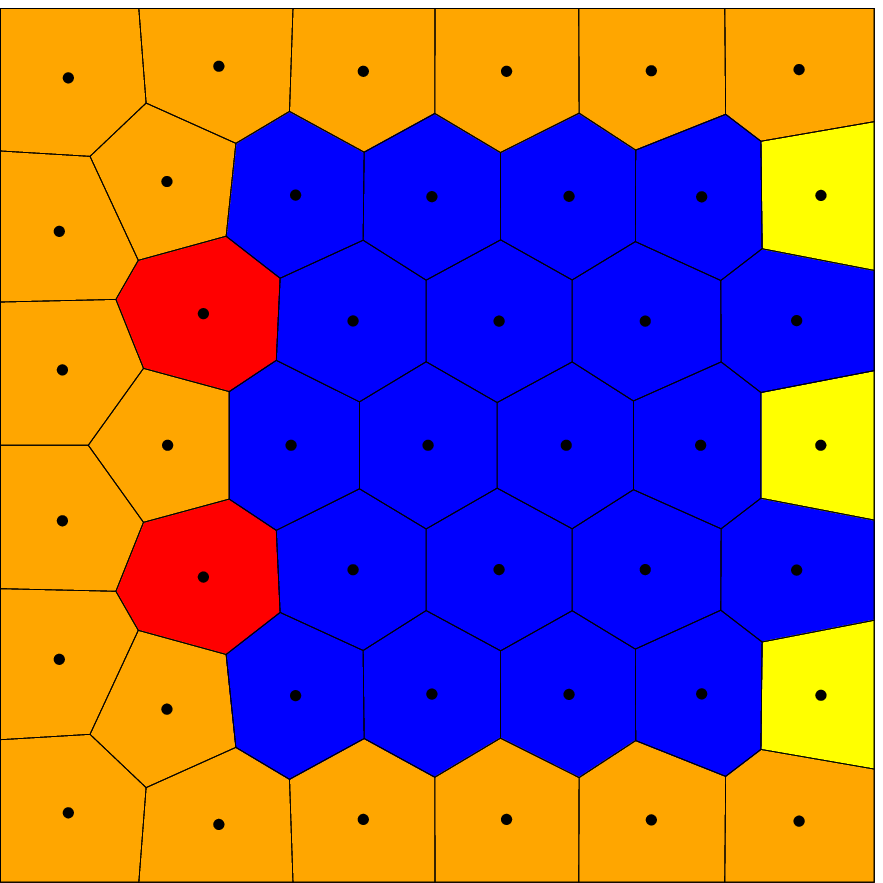}
\end{subfigure}
\begin{subfigure}[t]{0.31\linewidth}
	\centering
	\subcaption*{\footnotesize $\alpha=0.1$, $\delta=7.567 \times 10^{-6}$}
	\includegraphics[width=\linewidth]{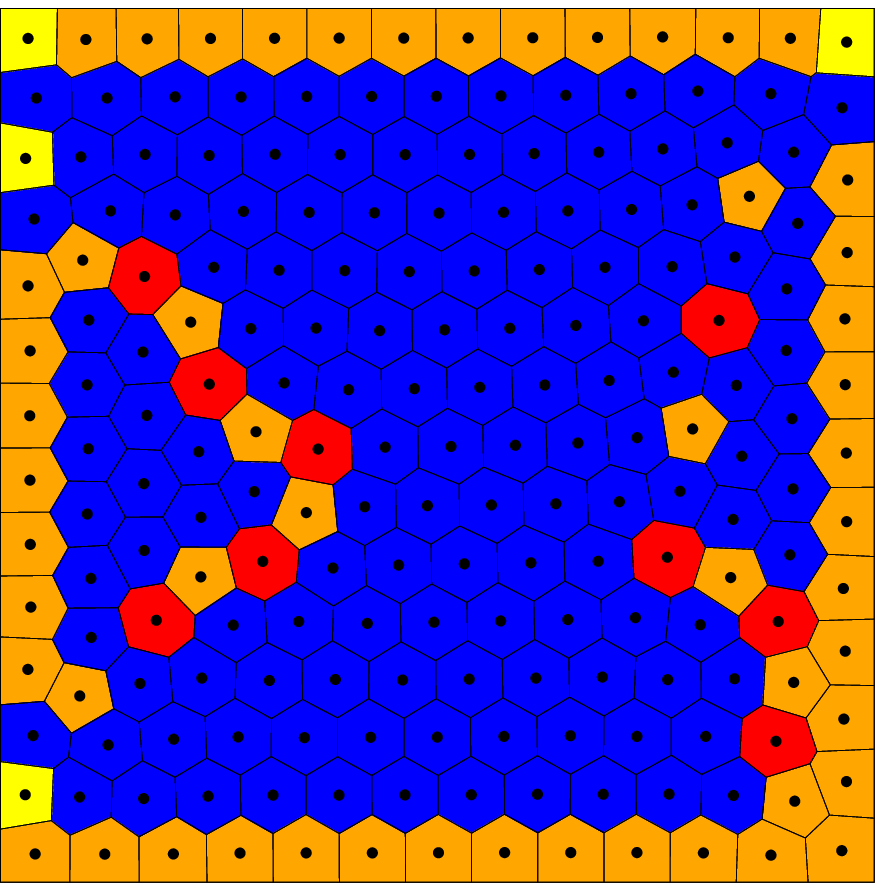}
	
\bigskip	
	
\end{subfigure}
\begin{subfigure}[t]{0.31\linewidth}
	\centering
	\subcaption*{\footnotesize $\alpha=0.583$, $\delta=1.472 \times 10^{-2}$}
	\includegraphics[width=\linewidth]{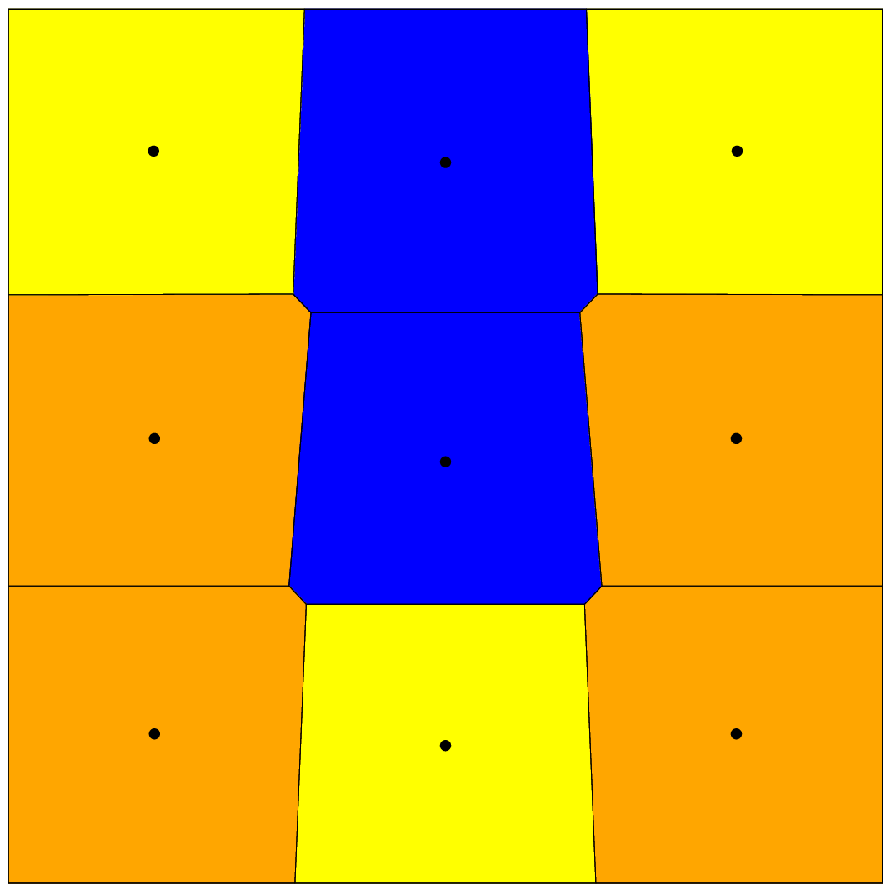}
\end{subfigure}
\begin{subfigure}[t]{0.31\linewidth}
	\centering
	\subcaption*{\footnotesize $\alpha=0.583$, $\delta=1.763 \times 10^{-3}$}
	\includegraphics[width=\linewidth]{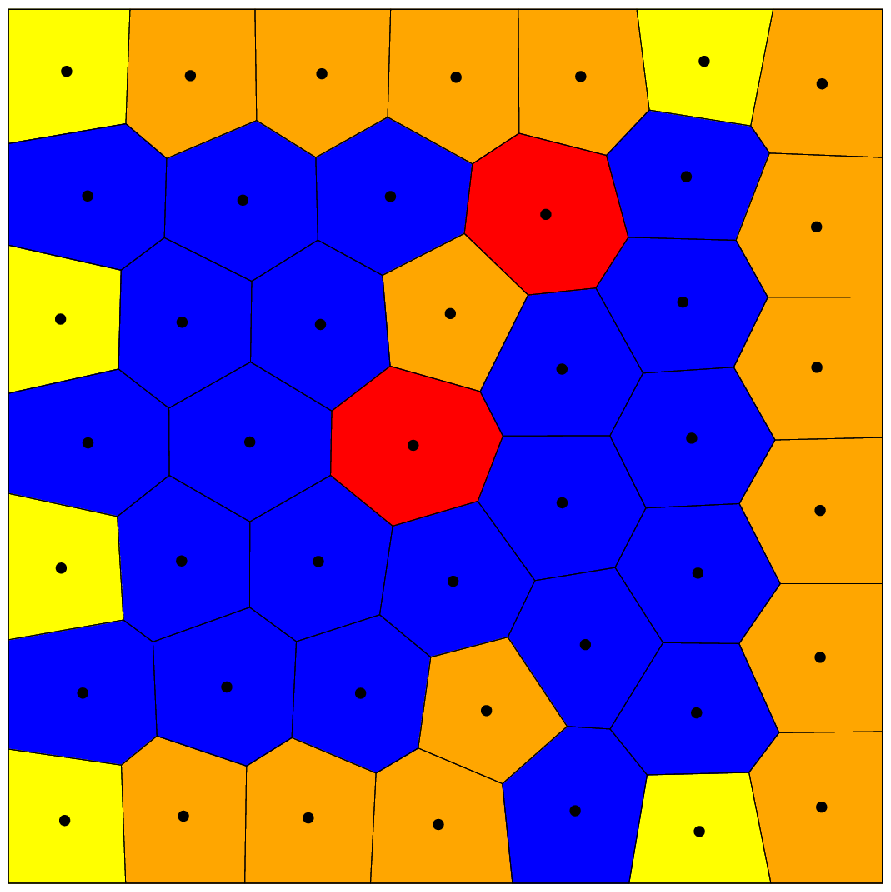}
	
\end{subfigure}
\begin{subfigure}[t]{0.31\linewidth}
	\centering
	\subcaption*{\footnotesize $\alpha=0.583$, $\delta=2.111 \times 10^{-4}$}
	\includegraphics[width=\linewidth]{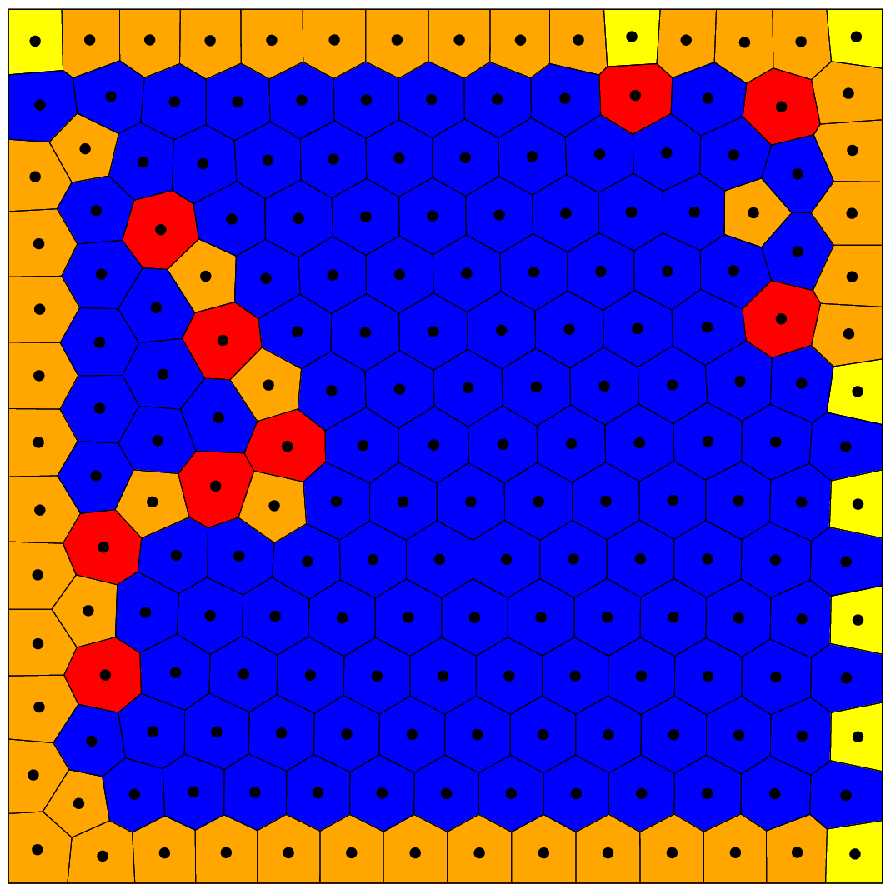}
\end{subfigure}

\bigskip

\begin{subfigure}[t]{0.31\linewidth}
	\centering
	\subcaption*{\footnotesize $\alpha=0.9$, $\delta=1.274 \times 10^{-1}$}
	\includegraphics[width=\linewidth]{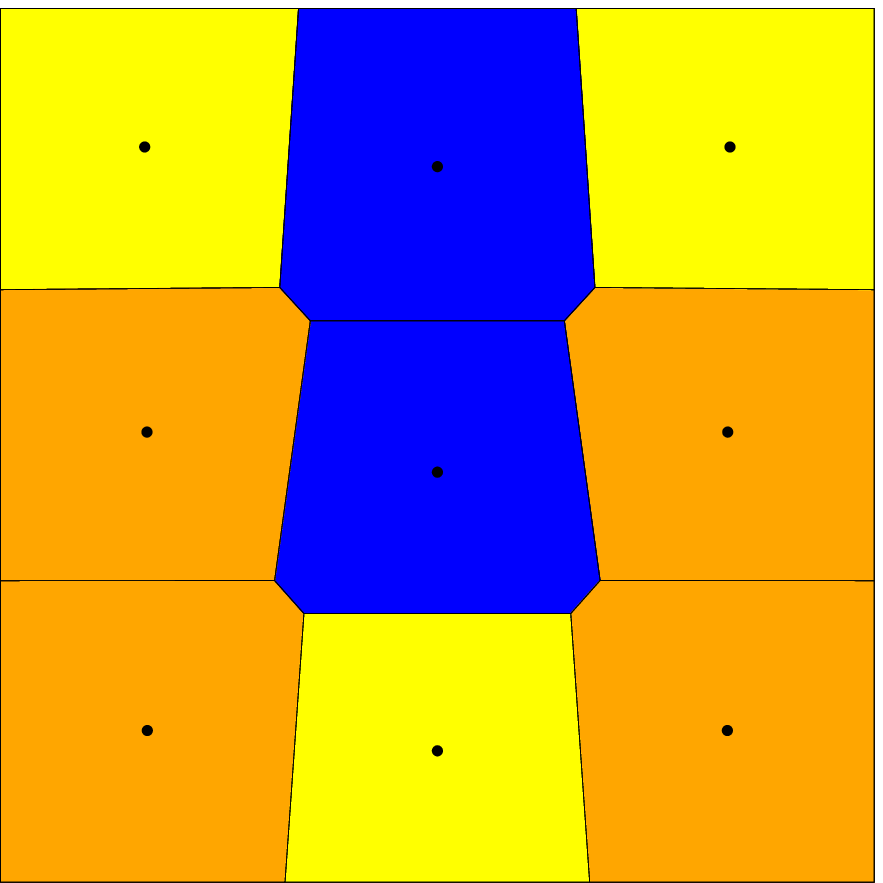}
\end{subfigure}
\begin{subfigure}[t]{0.31\linewidth}
	\centering
	\subcaption*{\footnotesize $\alpha=0.9$, $\delta=2.452 \times 10^{-2}$}
	\includegraphics[width=\linewidth]{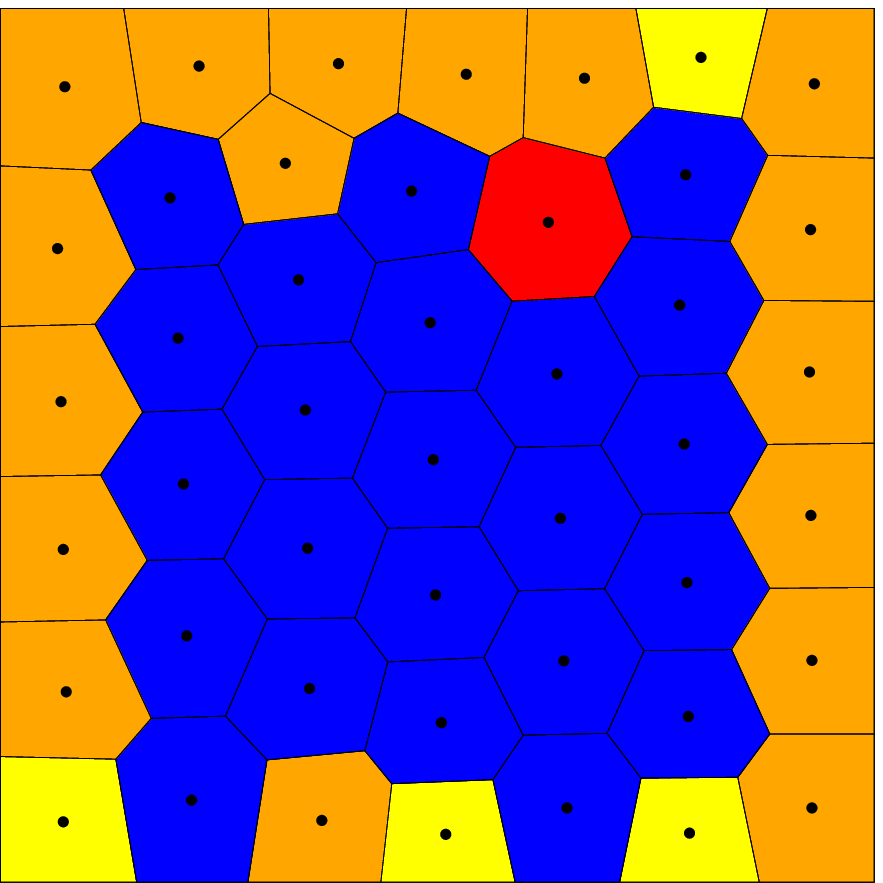}
\end{subfigure}
\begin{subfigure}[t]{0.31\linewidth}
	\centering
	\subcaption*{\footnotesize $\alpha=0.9$, $\delta=4.721 \times 10^{-3}$}
	\includegraphics[width=\linewidth]{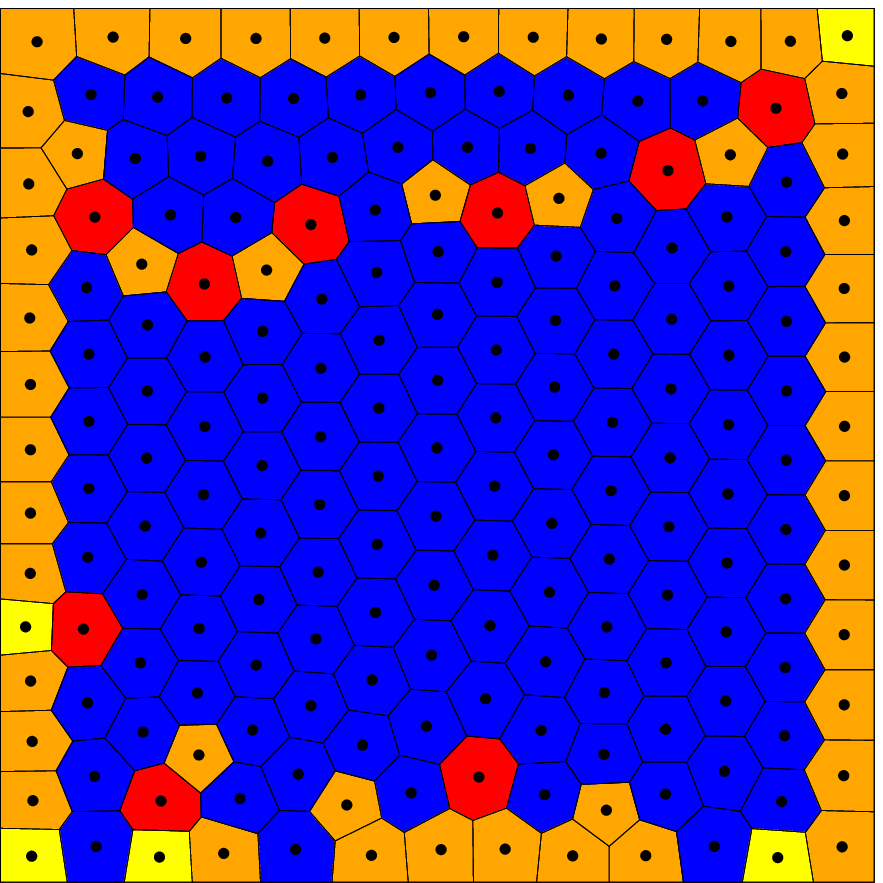}
\end{subfigure}	
\caption{\label{fig:numerics}Approximate local minimizers for the penalized problem \eqref{eq:minpbpenal} for the case $p=d=2$, $\Omega=[0,1]^2$, $f=\ca_\Omega$, $H_\alpha(\mu)=\sum_i m_i^\alpha$, for several values of $\alpha$ and $\delta$. The value of $\alpha$ is constant in each row, and the value of $\delta$ decreases from left to right in each row.  
	The black dots are the particles $z_i$,  where $\mu = \sum_{i=1}^{N_\mu} m_i \delta_{z_i}$ is an approximate local minimizer of \eqref{eq:minpbpenal}.
The polygons are the sets $T^{-1}(\{z_i\})$, where $T$ is the optimal transport map in \eqref{eq:Wassdist_p}. The particles $z_i$ are located at the centroids of the polygons. The masses $m_i$ are the areas of the polygons. 
The colours correspond to the number of sides: squares are yellow, pentagons are orange, hexagons are blue, and heptagons are red. For each value of $\alpha$, a hexagonal tiling (with defects) starts to emerge as $\delta$ is decreased.
This figure, Figure \ref{fig:numerics2} and Table \ref{Table1} were made by Steven Roper using the generalized Lloyd algorithm from \cite{BouRop}. To search for a global minimizer in the highly non-convex energy landscape, the algorithm was ran many times using different, randomly generated initial conditions. The values of $\delta$ were chosen by first choosing a target value of $N_\mu$ and then using the heuristic \eqref{eq:Nheuristic} to generate the corresponding $\delta$. Better results, without defects, can be achieved by taking the initial particle locations to be a perturbation of a triangular lattice; see Figure \ref{fig:numerics2}.}
\end{figure}

\begin{figure}
\centering
\begin{subfigure}[t]{0.45\linewidth}
\centering
\subcaption*{$\alpha=0.583$, $\delta=1.763 \times 10^{-3}$}
\includegraphics[width=\linewidth]{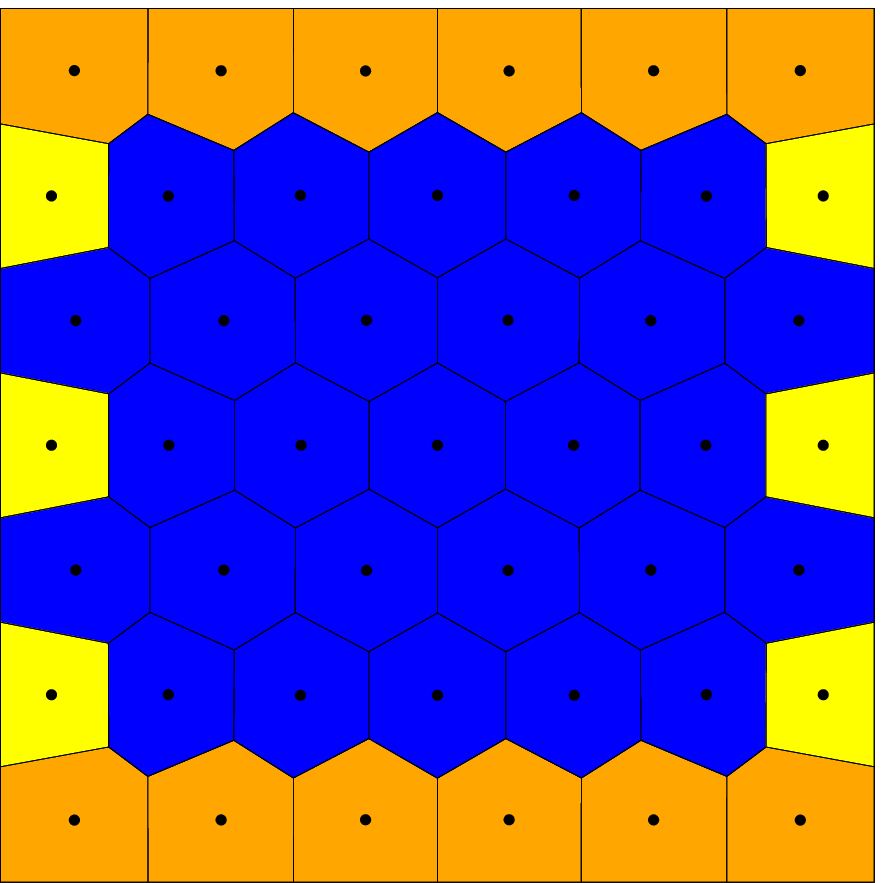}
\end{subfigure}
\begin{subfigure}[t]{0.45\linewidth}
\centering
\subcaption*{$\alpha=0.583$, $\delta=2.111 \times 10^{-4}$}
\includegraphics[width=\linewidth]{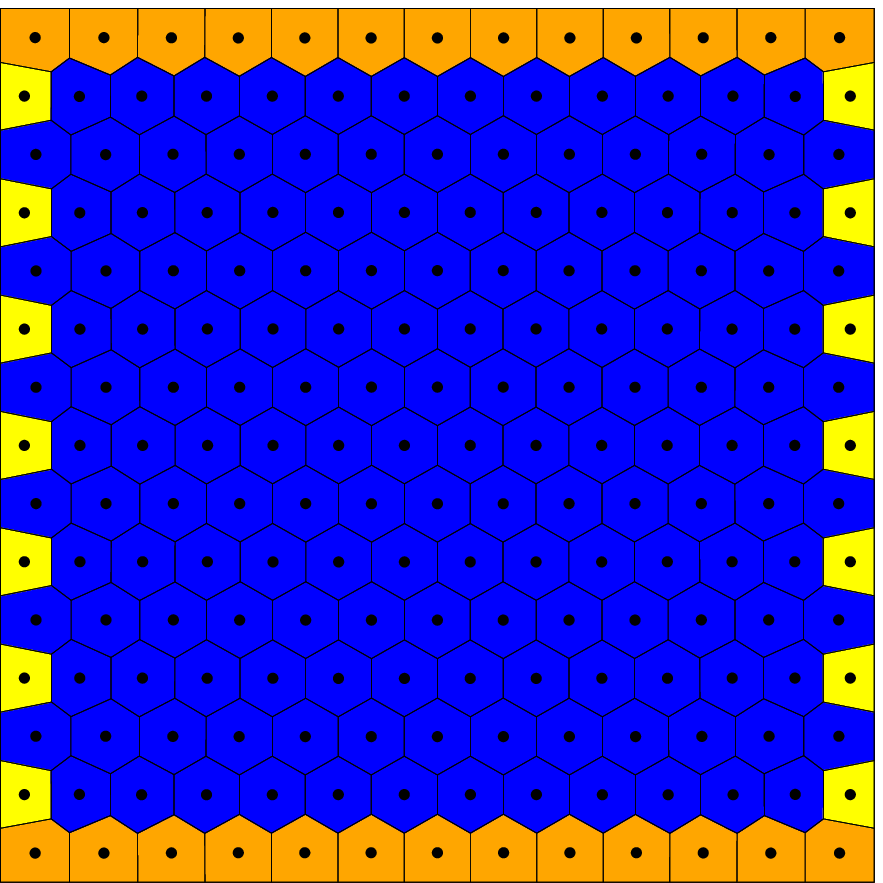}
\end{subfigure}
\caption{\label{fig:numerics2}Approximate global minimizers for the penalized problem \eqref{eq:minpbpenal} for the case $p=d=2$, $\Omega=[0,1]^2$, $f=\ca_\Omega$, $H_\alpha(\mu)=\sum_i m_i^\alpha$, $\alpha=0.583$ for the values of $\delta$ used in Figure \ref{fig:numerics2} (middle row, middle and right columns). See the caption to Figure \ref{fig:numerics} for a description of the polygons and the colour scheme. These configurations have lower energy ($W_2^2(f,\mu)+\delta H_\alpha(\mu)$) than the corresponding configurations shown in Figure \ref{fig:numerics}, and they do not have defects. This figure was generated by Steven Roper using the generalized Lloyd algorithm from \cite{BouRop} and by taking the initial conditions to be perturbations of a triangular lattice. In Figure \ref{fig:numerics} (middle row, right column) there are $N_\mu=200$ particles whereas in this figure (right) there are $N_\mu=202$ particles; algorithm \cite{BouRop} attempts to find the optimum number of particles.}
\end{figure}

For the entropy
\begin{equation}
\label{eq:alpha=0}
H(\mu)=\#\mathrm{supp}(\mu)
\end{equation}
Zador's Theorem for the asymptotic quantization error states that
\begin{equation}\label{eq:asympt}
\lim_{L\to\infty} \left[ L^{\frac{p}{d}} \mathcal{E}^{p,d}_H(L) \right]
	= C_{p,d} \left( \int_\Omega f(x)^{\frac{d}{d+p}} \,\dd x \right)^{\frac{d+p}{d}}
\end{equation}
for some positive constant $C_{p,d}$ that is independent of the density $f$.
See for example \cite{BucWis,GL,Gr04,Za82} and see \cite{Gru_Riem,Iac,Klo} for the more general case where $\Omega$ is a Riemannian manifold.
The constant $C_{p,d}$ is known in two dimensions:
\begin{equation}
\label{eq:Cp20}
C_{p,2} = \int_{P_6} |x|^p \,\dd x,
\end{equation}
where $P_6$ is a regular hexagon of unit area centred at the origin. This follows from Fejes T\'oth's Theorem on Sums of Moments (see \cite{FEJ,Gru_Fej}), which has also been proved in various levels of generality by several other authors including \cite{BollobasStern72,Fej_Stability,BolMor,New}.

The geometric interpretation of \eqref{eq:asympt} and \eqref{eq:Cp20} is the following: In two dimensions it is asymptotically optimal to arrange the atoms of the discrete measure at the centres of regular hexagons, i.e., on a regular triangular lattice, where the areas of the hexagons depend on the density $f$. Locally, where $f$ is approximately constant, these hexagons form a regular honeycomb. 
\blue{By the \emph{regular triangular lattice} we mean the set $\mathbb{Z}(1,0)\oplus \mathbb{Z}(1/2,\sqrt{3}/2)$ up to dilation and isometry.}
See Remark \ref{rem:rescaling} below for more on this geometric interpretation.

Formula \eqref{eq:asympt} was extended to more general entropies by Bouchitt\'{e}, Jimenez and Mahadevan in \cite{BouJimMaha}.
Their class of entropies includes the case
\begin{equation}\label{eq:entropies}
H_\alpha(\mu)=\sum_{i=1}^{N_\mu} m_i^\alpha,
\end{equation}
\blue{where $\alpha\in(-\infty,1)$.}
This reduces to the entropy \eqref{eq:alpha=0} when $\alpha=0$.
Bouchitt\'{e}, Jimenez and Mahadevan \cite[Proposition 3.11(i)]{BouJimMaha} proved that
\begin{equation}
\label{eq:Prop3.11(i)}
\lim_{L\to\infty} \left[ L^{\frac{p}{d(1-\alpha)}} \mathcal{E}^{p,d}_{H_\alpha}(L) \right]
	= C_{p,d}(\alpha) \left( \int_\Omega f(x)^{\frac{d(1-\alpha)+\alpha p}{d(1-\alpha)+p}} \,\dd x
		\right)^{1+\frac{p}{d(1-\alpha)}}
\end{equation}
for some positive constant $C_{p,d}(\alpha)$.
Moreover, they conjectured \cite[Section 3.6 (ii)]{BouJimMaha} that
\[
C_{p,d}(\alpha) \textrm{ is independent of } \alpha. 
\]
If this conjecture is true, then by \eqref{eq:Cp20}
\[
C_{p,2}(\alpha) = C_{p,2}(0) = \int_{P_6} |x|^p \,\dd x.
\]
In particular, the conjecture for the case $p=2$, $d=2$ is
\begin{equation}
\label{eq:conjecture}
C_{2,2}(\alpha)=\int_{P_6} |x|^2 \,\dd x=\frac{5}{18\sqrt{3}} =: c_6
\end{equation}
\blue{for all $\alpha \in (-\infty,1)$. It is known that $C_{2,2}(\alpha)=c_6$ for all $\alpha \in (-\infty,0]$ (see \cite[Section 3.6]{BouJimMaha})
and so it remains to establish the conjecture for the case $\alpha \in (0,1)$.}
The conjecture would mean that in two dimensions a discrete measure supported on a regular triangular lattice gives asymptotically the best constrained approximation of the Lebesgue measure (again, see Remark \ref{rem:rescaling} below for this geometric interpretation).

\subsection{Main results}
In this paper we prove conjecture \eqref{eq:conjecture} 
 for all
$\alpha\in(-\infty,\overline{\alpha}]$, where $\overline{\alpha}=0.583$; see Theorem \ref{thm:constr}. The conjecture for $\alpha \in (\overline{\alpha},1)$ remains open, although we suggest a direction for proving it in Theorem \ref{thm:alpha01}, where we prove it under an additional assumption. 
In Theorem \ref{thm:pen} we prove an analogous asymptotic quantization formula for the penalized optimal location problem \eqref{eq:minpbpenal} \blue{for all $\alpha\in(-\infty,\overline{\alpha}]$.} This generalises the crystallization result of \cite{BouPelTheil}, where Theorem \ref{thm:pen} was proved for the special case $\alpha=0.5$, $f=1$.
\blue{Moreover, for the case $f= 1$, we prove that minimal configurations are `asymptotically approximately' a triangular lattice; see Theorem \ref{thm:stability}. To be more precise, we prove that, as $\delta\to0$, rescaled minimal configurations for the penalized quantization problem are quantitatively close to a triangular lattice. This result will be proved for the case $\alpha=\overline{\alpha}$. The proof can be easily modified for any $\alpha\leq\overline{\alpha}$.
}

Define the constrained optimal quantization error by
\begin{equation}
\label{eq:m_c}
\mathrm{m_c}(\alpha,L):= \mathcal{E}^{2,2}_{H_\alpha}(L) = \inf\left\{ W^2_2(f,\mu) : \mu\in\Pd,\,\, \sum_{i=1}^{N_\mu} m_i^\alpha\leq L \right\},
\end{equation}
and the penalized optimal quantization error by
\begin{equation}
\label{eq:m_p}
\mathrm{m_p}(\alpha,\delta):= \mathcal{F}^{2,2}_{H_\alpha}(\delta) = \inf\left\{ W^2_2\left(\,f,\mu\,\right) + \delta \sum_{i=1}^{N_\mu} m_i^\alpha : \mu\in\Pd \right\}.
\end{equation}
Since the Wasserstein distance on the compact set $\Omega$ metrizes the tight convergence of probability measures, and the map $\mu \mapsto \sum_i m_i^\alpha$ is lower semi-continuous with respect to this convergence \cite[Lemma 7.11]{S}, both infima above are attained.
Our main results are the following.

\begin{theorem}[Asymptotic crystallization for the penalized optimal location problem]
\label{thm:pen}
Let \blue{$\alpha\in(-\infty,\overline{\alpha}]$}, where $\overline{\alpha}:=0.583$.
Let $\Omega\subset\R^2$ be the closure of an open and bounded set.
Assume that $f:\Omega \to [0,\infty)$ is lower semi-continuous with $f\geq c>0$ and $\int_\Omega f \, \dd x = 1$.
Then
\begin{equation}\label{eq:limitpen}
\lim_{\delta\to0}\left[ \left(\frac{c_6}{\delta(1-\alpha)}\right)^{\frac{1}{2-\alpha}}
	\mathrm{m_p}(\alpha,\delta) \right]
		=\frac{2-\alpha}{1-\alpha} \, c_6\int_\o f(x)^{\frac{1}{2-\alpha}} \,\dd x.
\end{equation}
\end{theorem}

Taking the special case $f=1$, $|\Omega|=1$, $\alpha=0.5$ in Theorem \ref{thm:pen} gives \cite[Theorem 2]{BouPelTheil}.
We illustrate Theorem \ref{thm:pen} in Table \ref{Table1} and Figures \ref{fig:numerics} and \ref{fig:numerics2}.

\begin{theorem}[Asymptotic crystallization for the constrained optimal location problem]
\label{thm:constr}
Let \blue{$\alpha\in(-\infty,\overline{\alpha}]$}, where $\overline{\alpha}:=0.583$.
Let $\Omega\subset\R^2$ be the closure of an open and bounded set.
Assume that $f:\Omega \to [0,\infty)$ is lower semi-continuous with $f\geq c>0$ and $\int_\Omega f \, \dd x = 1$.
Then
\begin{equation}\label{eq:limitcon}
\lim_{L\to\infty}\left[ L^{\frac{1}{1-\alpha}} \, \mathrm{m_c}(\alpha,L) \right]
	= c_6\left( \int_\o f(x)^{\frac{1}{2-\alpha}} \,\dd x \right)^{\frac{2-\alpha}{1-\alpha}}.
\end{equation}
\end{theorem}

By comparing equation \eqref{eq:Prop3.11(i)} to equation \eqref{eq:limitcon} with $p=d=2$, we read off that $C_{2,2}(\alpha)=c_6$ for all \blue{$\alpha\in(-\infty,\overline{\alpha}]$}, which proves conjecture \eqref{eq:conjecture} for this range of $\alpha$.
We believe that Theorem \ref{thm:pen} and Theorem \ref{thm:constr} hold for all \blue{$\alpha\in(-\infty,1)$}, not just for \blue{$\alpha\in(-\infty,\overline{\alpha}]$}, but we are only able to prove them for the whole range of $\alpha$ if we make an ansatz about minimal configurations; see Theorem \ref{thm:alpha01}.

\begin{table}
\begin{tabular}{ccc}
	\hline
	$\, \; \alpha \; \,$ & $\, \; \delta\; \, $ 
	& $\, \; \frac{ \left( \frac{c_6}{\delta (1-\alpha)}\right)^{\frac{1}{2-\alpha}\phantom{\big|}} \! \! \! \left( W_2^2(\ca_{\Omega},\mu)+\delta H_\alpha(\mu) \right)}{ c_6 (2-\alpha)/(1-\alpha)} \, \;$ 
	\vspace{0.1cm}
	\\
	\hline
	 0.1  & 0.0022433361900  
	 & 1.025664680453751 
	 \\
	 0.1 & 0.0001302910000 
	 & 1.012634927464421 
	 \\
	 0.1 & $7.5672376050508 \times 10^{-6}$ 
	 & 1.006166789745220 
	 \\
	 \hline
	 0.583 & 0.0147231527000 
	 & 1.015173622346968 
	 \\
	 0.583 & 0.0017628730000 
	 &  1.007372522192174
	 \\
	 0.583 & $2.1107719443098 \times 10^{-4}$ 
	 & 1.003575800389361 
	 \\
	 \hline
	 0.9 & 0.1273904500000 
	 & 1.004506412114665
	 \\
	 0.9 & 0.0245228180000 
	 & 1.002133746895388
	 \\
	 0.9 & 0.004720672550584 
	 & 1.001028468127525
	 \\
	\hline
	\\
\end{tabular}
\caption{\label{Table1}Illustration of Theorem \ref{thm:pen} for the case $\Omega=[0,1]^2$, $f=\ca_\Omega$. In the last column we give the ratio of a numerical approximation of
$\big(\tfrac{c_6}{\delta(1-\alpha)}\big)^{\frac{1}{2-\alpha}} \mathrm{m_p}(\alpha,\delta)$ (which appears on the left-hand side of \eqref{eq:limitpen}) to $\frac{2-\alpha}{1-\alpha} \, c_6$ (which appears on the right-hand side of \eqref{eq:limitpen}; note that $\int_\o f(x)^{\frac{1}{2-\alpha}} \,\dd x=1$ here). For three representative values of $\alpha$, we see that the ratio tends to $1$ as $\delta$ tends to $0$, which supports Theorems \ref{thm:pen} and \ref{thm:alpha01}.
The numerical approximation of the minimum energy and the choice of $\delta$ is described in Figures \ref{fig:numerics} and \ref{fig:numerics2}.}
\end{table}

\begin{remark}[Energy scaling and the geometric interpretation of Theorems \ref{thm:pen} \& \ref{thm:constr}]\label{rem:rescaling}
To motivate the rescaling on the left-hand side of \eqref{eq:limitpen} we reason as follows.
Let
\[
\Omega = \bigcup_{i=1}^N H_i
\]
be the union of $N$ disjoint regular hexagons of equal area $|\Omega|/N$. Let $z_i$ be the centroid of $H_i$ and let
$f = \tfrac{1}{|\Omega|}\ca_\Omega$ be the uniform probability distribution on $\Omega$.
Here $\ca_\Omega$ denotes the characteristic function of the set $\Omega$. By definition of $c_6$ (equation \eqref{eq:conjecture}) and a change of variables,
\[
\int_{H_i} |x-z_i|^2 \, \dd x = c_6 |H_i|^2 = c_6 \left(\frac{|\Omega|}{N}\right)^2
\]
for all $i$.
Therefore the penalized quantization error of approximating $f \dd x$ by $\mu=\sum_{i=1}^N \tfrac{1}{N} \delta_{z_i}$ is
\begin{equation}\label{eq:opten1}
W_2^2(f,\mu)+\delta H_\alpha(\mu) = \sum_{i=1}^N \int_{H_i} |x-z_i|^2 \frac{1}{|\Omega|} \, \dd x + \delta N \left( \frac{1}{N} \right)^{\alpha}
=   c_6\frac{|\o|}{N} + \delta N^{1-\alpha}.
\end{equation}
 The right-hand side of \eqref{eq:opten1} is minimized when
\begin{equation}
\label{eq:Nheuristic}
N = \left(\frac{c_6 |\Omega|}{\delta(1-\alpha)}\right)^{\frac{1}{2-\alpha}}.
\end{equation}
Substituting this value of $N$ into \eqref{eq:opten1} (assuming for a moment that it is an integer) gives
\[
\left(\frac{c_6}{\delta(1-\alpha)}\right)^{\frac{1}{2-\alpha}} \left( W_2^2(f,\mu)+\delta H_\alpha(\mu) \right)
= \frac{2-\alpha}{1-\alpha} \, c_6 \, |\o|^{\frac{1-\alpha}{2-\alpha}} = \frac{2-\alpha}{1-\alpha} \, c_6\int_\o f(x)^{\frac{1}{2-\alpha}} \,\dd x,
\]
which motivates the rescaling used in \eqref{eq:limitpen}. This heuristic computation suggests an upper bound for the left-hand side of \eqref{eq:limitpen}, for the case where $f$ is the uniform distribution. Theorem \ref{thm:pen} says that this upper bound is in fact asymptotically optimal. In this sense we can say that the honeycomb structure gives asymptotically the best approximation of the uniform distribution.

The rescaling used in \eqref{eq:limitcon} can be derived in a similar way.
Indeed, fix $L>0$ and consider the constraint
\[
H_\alpha(\mu)=\sum_{i=1}^{N_\mu} m_i^\alpha\leq L.
\]
If all the masses are the same, $m_i = 1/N_\mu$ for all $i$, then the biggest number $N_\mu$ for which this constraint is satisfied is
\[
N_\mu=L^{\frac{1}{1-\alpha}}.
\]
Assuming that $N_\mu$ is an integer, take as above
\[
\Omega = \bigcup_{i=1}^{N_\mu} H_i, \quad\quad
f=\frac{1}{|\Omega|} \ca_\Omega, \quad\quad
\mu=\sum_{i=1}^{N_\mu} \frac{1}{N_\mu} \delta_{z_i}.
\]
Then
\[
L^{\frac{1}{1-\alpha}} \, W_2^2(f,\mu) =
c_6 \, |\o| = c_6\left(\, \int_\o f(x)^{\frac{1}{2-\alpha}} \,\dd x\,\right)^{\frac{2-\alpha}{1-\alpha}},
\]
which motivates the rescaling used in \eqref{eq:limitcon}. Combining this formal calculation with Theorem \ref{thm:constr} again suggests the asymptotic optimality of the honeycomb.
\end{remark}

\blue{
Theorem \ref{thm:pen} gives the asymptotic minimum value of the penalized quantization error but says nothing about the configuration of the particles; it says that the triangular lattice is asymptotically optimal, but it does not say that asymptotically optimal configurations are close to a triangular lattice. We prove this in the following theorem.

\begin{theorem}[Asymptotically optimal configurations are close to a regular triangular lattice]\label{thm:stability}
Let $\Omega\subset\R^2$ be a convex polygon with at most six sides, $|\Omega|=1$, $f=\ca_\Omega$, and $\alpha = \overline{\alpha}$.
There exist constants $\varepsilon_0, c, \beta_1, \beta_2 >0$
with the following property.
Let $\delta>0$ and $\mu_\delta =\sum_{i=1}^{N_\delta} \widetilde{m}_i\delta_{\widetilde{z}_i}\in\Pd$ be a solution of the penalized quantization problem defining $\mathrm{m}_{\mathrm{p}}(\alpha,\delta)$. Define the defect of $\mu_\delta$ by
\[
\mathrm{d}(\mu_\delta):=  \left(\frac{c_6}{\delta(1-\alpha)}\right)^{\frac{1}{2-\alpha}}\mathrm{m_p}(\alpha,\delta) - \frac{2-\alpha}{1-\alpha} \, c_6.
\]
Note that $\lim_{\delta \to 0} \mathrm{d}(\mu_\delta) = 0$  by Theorem \ref{thm:pen}.
Define
\[
V_{\delta,\alpha}= \left( \frac{c_6}{\delta(1-\alpha)} \right)^{\frac{1}{2-\alpha}}.
\]
Define rescaled particle positions $z_i = V_{\delta,\alpha}^{1/2} \widetilde{z}_i$, $i \in \{1,\ldots,N_\delta \}$.
Let $\{ V_i \}_{i=1}^{N_\delta}$ be the Voronoi tessellation of $\Omega$ generated by $\{z_i\}_{i=1}^{N_\delta}$, i.e., 
\[
V_i = \{ z \in V_{\delta,\alpha}^{1/2} \Omega : |z-z_i| \le |z-z_j| \; \forall  \, j \in \{1,\ldots,N_\delta\}\}.
\]
\begin{itemize}
\item[(a)] The optimal number of particles $N_\delta$ is asymptotically equal to 
$V_{\delta,\alpha}$:
\[
\lim_{\delta \to 0} \frac{V_{\delta,\alpha}}{N_\delta}=1.
\]
\item[(b)]
If $\delta>0$ is sufficiently small, and if $\varepsilon\in(0,\varepsilon_0)$ and $\mu_\delta$ satisfy
\begin{equation}
\label{eq:dmulesseps}
\beta_1 \mathrm{d}(\mu_\delta) + \beta_2 V_{\delta,\alpha}^{-1/2} \leq \varepsilon,
\end{equation}
then, with the possible exception of at most $N_\delta c \varepsilon^{1/3}$ indices $i\in\{1,\dots,N_\delta \}$, the following hold:
\begin{itemize}
\item[(i)] $V_i$ is a hexagon;
\item[(ii)] the distance between $z_i$ and each vertex of $V_i$ is between $(1 \pm \varepsilon^{1/3}) \sqrt{\frac{V_{\delta,\alpha}}{N_\delta}} \sqrt{\frac{\phantom{|}2\phantom{|}}{3 \sqrt{3}}}$;
\item[(iii)] the distance between $z_i$ and each edge of $V_i$ is between $(1 \pm \varepsilon^{1/3}) \sqrt{\frac{V_{\delta,\alpha}}{N_\delta}} \sqrt{\frac{\phantom{|}1\phantom{|}}{2 \sqrt{3}}}$. 
\end{itemize}
\end{itemize}
\end{theorem}
}

\blue{
Even though Theorem \ref{thm:stability} is stated only for the case $\alpha=\overline{\alpha}$, the same proof holds for any $\alpha\leq\overline{\alpha}$, up to proving the convexity inequality \eqref{eq:ciintro} for that specific value of $\alpha$ (by using the same strategy we used for the case $\alpha=\overline{\alpha}$). A similar result can be proved for the constrained quantization problem.}

\blue{
\begin{remark}[Geometric interpretation of Theorem \ref{thm:stability}]
Note that the term $ \beta_2 V_{\delta,\alpha}^{-1/2}$ in \eqref{eq:dmulesseps} converges to $0$ as $\delta \to 0$.
Theorem \ref{thm:stability} essentially states that if the defect $\mathrm{d}(\mu_\delta)$ is small, then the support of $\mu_\delta$ is close to a regular triangular lattice, and it quantifies how close. Note that the Voronoi tessellation generated by the regular triangular lattice is a regular hexagonal tessellation. The theorem states that the Voronoi tessellation of $ V_{\delta,\alpha}^{1/2} \Omega$ generated by the rescaled particles $z_i$ is close to a regular hexagonal tessellation in the sense that, except for at most $N_\delta c \varepsilon^{1/3}$ Voronoi cells, the Voronoi cells are hexagons, and it quantifies how far the hexagons are from being regular. For a \emph{regular} hexagon of area $\frac{V_{\delta,\alpha}}{N_\delta}$, the distance between the centre of the hexagon and each vertex is 
$\sqrt{\frac{V_{\delta,\alpha}}{N_\delta}}  
\sqrt{\frac{\phantom{|}2\phantom{|}}{3 \sqrt{3}}}$, and the distance between the centre of the hexagon and each edge is $\sqrt{\frac{V_{\delta,\alpha}}{N_\delta}} \sqrt{\frac{\phantom{|}1\phantom{|}}{2 \sqrt{3}}}$. Since $\lim_{\delta \to 0} V_{\delta,\alpha} /N_\delta \to 1$, `most' of the rescaled Voronoi cells $V_i$ are `close' to a regular hexagon of area 1.
\end{remark}
}

\blue{
\begin{remark}[Locality and weaker assumptions on $f$]
Theorems \ref{thm:pen}, \ref{thm:constr} say that the quantization problems are essentially independent of $f$, in the sense that the optimal constants $\frac{2-\alpha}{1-\alpha} c_6$ and $c_6$ are independent of $f$ and are determined by the corresponding quantization problems with $f=1$; see Remarks \ref{remark:3_5} and \ref{rem:Cp}. The locality of the quantization problems is independent of the crystallization and is easier to prove. The locality for the constrained problem was proved by \cite{BouJimMaha} and the locality for the penalised problem follows easily from this, as we shall see in Section \ref{sec:asympt_pen}. Locality results for the classical quantization problem were proved among others by \cite{BucWis,Gr04,Klo,Za82}. 
We believe that the assumption of lower semi-continuity on $f$ in Theorems \ref{thm:pen}, \ref{thm:constr} could be relaxed by using the approach in \cite{MosTil}, where a locality result is proved for the related \emph{irrigation problem}, which concerns the best approximation of an absolutely continuous probability measure by a
one-dimensional Hausdorff measure supported on a curve.
\end{remark} 
}

\blue{
\begin{remark}[$\alpha \ge 1$]
For $\alpha\geq1$, the constrained and penalized quantization problems $\mathrm{m_c}(1,L)$ and $\mathrm{m_p}(1,\delta)$ do not have a minimizer. The infimum is zero since both the Wasserstein distance and the entropy can be sent to zero by sending the number of particles to infinity. In \cite{BouJimMaha} the authors considered the constraint
\[
\sum_{i=1}^{N_\mu}m_i^\alpha \ge \frac{1}{L}
\]
for $\alpha > 1$.
For $f \in L^\infty(\Omega)$, $\alpha \in (1,2) \cup (2,\infty)$ they proved that there exists a constant $C_{2,2}(\alpha)$ such that
\[
\lim_{L \to \infty} L^{\frac{1}{\alpha - 1}} \inf \left\{ W^2_2(f,\mu) : \mu\in\Pd,\, \sum_{i=1}^{N_\mu}m_i^\alpha \ge \frac{1}{L} \right\} =
C_{2,2}(\alpha) \left(\int_\Omega f(x)^{\frac{1}{2-\alpha}} \, \mathrm{d}x \right)^{\frac{2-\alpha}{1-\alpha}}.
\]
See \cite[Proposition 3.11(iii), Remark 3.13(iii)]{BouJimMaha}. For $\alpha >2$, $C_{2,2}(\alpha)=0$. For $\alpha \in (1,2)$, $C_{2,2}(\alpha)$ is not known, but it satisfies the bounds
\[
\int_B |x|^2 \, \mathrm{d}x \le C_{2,2}(\alpha) \le c_6
\]
where $B$ is the ball of unit area centred at the origin \cite[Lemma 3.10]{BouJimMaha}. 
\end{remark}
}

\blue{
\begin{remark}[Motivation for the choice of entropy $H_\alpha$]
The are several reasons why we chose to study the entropy $H_\alpha(\mu)=\sum_i m_i^\alpha$, both mathematical and from a modelling point of view.
\begin{itemize}
\item[(i)] The functional $\mu \mapsto W_2^2(f,\mu) + \delta \sum_i h(m_i)$ is lower semi-continuous if $h(0)=0$, $h(t) \ge0$, $h$ is lower semi-continuous, subadditive and $\lim_{t\to0+}h(t)/t=+\infty$; see \cite[Lemma 7.11]{S}. This includes our entropy $h(m)=m^\alpha$. There is evidence, however, that crystallization does not hold for all entropies in this class, or at least that optimal configurations consist of particles of different sizes; see \cite[Section 3.4]{BouJimMaha}. In this paper we have found a subclass for which crystallization holds. It is an open problem to find the largest class of such entropies. 
\item[(ii)] Functionals of the form $\mu \mapsto W_2^2(f,\mu) + \delta \sum_i h(m_i)$ arise in models of economic planning; see \cite{ButSan}. For example, consider the problem of the optimal location of warehouses in a county $\Omega$ with population density $f$. The measure $\mu=\sum_i m_i \delta_{z_i}$ represents the locations $z_i$ and sizes $m_i$ of the warehouses. The Wasserstein term in the functional above penalizes the average distance between the population and the warehouses, and the entropy term penalizes the building or running costs of the warehouses. The subadditive nature of the entropy $H_\alpha$ corresponds to an economy of scale, where it is cheaper to build one warehouses of size $m$ than two of size $m/2$.
\item[(iii)] The special case $\alpha=0.5$ arises in a simplified model of a two-phase fluid, namely a diblock copolymer melt, in two dimensions; see \cite{BouPelRop}. Here the entropy $\sqrt{m}$ corresponds to the interfacial length between a droplet of one phase of area $m$ and the surrounding, dominant phase. 
\item[(iv)] Finally, from a mathematical perspective, we were inspired to study the entropy $H_\alpha$ by the conjecture of Bouchitt\'e et al.~\cite[Section 3.6 (ii)]{BouJimMaha}.
\end{itemize}
\end{remark}
}


\subsection{Sketch of the proofs of Theorems \ref{thm:pen} \& \ref{thm:constr}}

We briefly present the main ideas of the paper.
We will see that Theorem \ref{thm:constr} is an easy consequence of Theorem \ref{thm:pen} (see Section \ref{sec:proofconstrained}), and so here we just focus on the \blue{ideas} behind the proof of Theorem \ref{thm:pen}. \blue{The strategy for proving Theorem \ref{thm:stability} is discussed in Section \ref{sec:stability}.} 

First we identify the scaling of the penalized quantization error $\mathrm{m_p}(\alpha,\delta)$ as $\delta \to 0$ using the
$\Gamma$-convergence result of  \cite{BouJimMaha}. This gives
\begin{equation}
\label{eq:scaling}
\lim_{\delta\to0}\left[\left(\frac{c_6}{\delta(1-\alpha)}\right)^{\frac{1}{2-\alpha}}
	\mathrm{m_p}(\alpha,\delta)\right]=C_{\mathrm{p}}(\alpha) \int_\o f(x)^{\frac{1}{2-\alpha}} \,\dd x
\end{equation}
where
\begin{equation}
\label{eq:FormOfC}
C_{\mathrm{p}}(\alpha)=\
\lim_{\delta\to0}\left[\left(\frac{c_6}{\delta(1-\alpha)}\right)^{\frac{1}{2-\alpha}}
	\min\left\{  \delta\sum_{i=1}^{N_\mu} m_i^\alpha + W^2_2(\ca_Q,\mu) \,:\, \mu\in\mathcal{P}_{\mathrm{d}}(Q) \right\} \right]
\end{equation}
and $Q=[-1/2,1/2]^2$;
see Corollary \ref{cor:monoCp} and Remark \ref{rem:Cp}. The main challenge in this paper is to show that the optimal constant is $C_{\mathrm{p}}(\alpha)= c_6 (2-\alpha)/(1-\alpha)$ for all \blue{$\alpha\in(-\infty,\overline{\alpha}]$}. Thanks to equations \eqref{eq:scaling} and \eqref{eq:FormOfC},  to prove Theorem \ref{thm:pen}
it is sufficient to prove it for the case where $\Omega= Q$ and $f=\ca_\Omega$.

Next we prove a monotonicity result (Lemma \ref{lem:monoCp}), which is analogous to a monotonicity result proved by \cite{BouJimMaha} for the constrained quantization problem, which asserts that
 if Theorem \ref{thm:pen} holds for some \blue{$\widetilde{\alpha}\in(-\infty,1)$}, then it holds for all \blue{$\alpha\in(-\infty,\widetilde{\alpha}]$}. Therefore we only need to prove Theorem \ref{thm:pen} for the single value $\alpha=\overline{\alpha}=0.583$.
Therefore for the rest of the paper we can take $\Omega=Q$, $f=\ca_\Omega$, $\alpha=\overline{\alpha}$ without loss of generality.

From the definition of the Wasserstein distance, equation \eqref{eq:Wassdist_p}, if $\mu=\sum_{i=1}^{N_\mu} m_i \delta_{z_i}$, then
\begin{equation}
\label{eq:Sketch1}
W_2^2(\ca_Q,\mu) = \sum_{i=1}^{N_\mu} \int_{T^{-1}(\{z_i\})} |x-z_i|^2 \,\dd x,
\end{equation}
where $T$ is the optimal transport map.
\blue{Since $Q$ is a polygonal set}, it is well known (see Lemma \ref{lem:cellspolygons}) that the sets $T^{-1}(\{z_i\})$ are convex polygons, called \emph{Laguerre cells}.

A classical result by Fejes T\'{o}th (see Lemma \ref{lem:Fejes}) states that
the second moment of a polygon about any point in the plane is greater than or equal to the second moment of a regular polygon (with the same area and same number of edges) about its centre of mass:
\begin{equation}
\label{eq:Sketch2}
\int_{P(m,n)} |x-z|^2 \,\dd x \ge \int_{R(m,n)} |x|^2 \, \dd x = m^2 \int_{R(1,n)} |y|^2 \, \dd y =: m^2 c_n
\end{equation}
where $P(m,n)$ is a polygon with area $m$ and $n$ edges, $R(m,n)$ is a regular polygon centred at the origin with area $m$ and $n$ edges, and $z \in \mathbb{R}^2$. Combining \eqref{eq:Sketch1} and \eqref{eq:Sketch2} gives
\begin{equation}\label{eq:ineqW}
W_2^2(\ca_\Omega,\mu) \geq \sum_{i=1}^{N_\mu} m_i^2 c_{n_i},
\end{equation}
for all $\mu \in \Pd$, where $n_i$ denotes the number of edges of the polygon $T^{-1}(\{z_i\})$ and $m_i$ denotes its area.
Our proofs are limited to the $p$-Wasserstein metric with $p=2$ since, for $p \ne 2$, the transport regions $T^{-1}(\{ z_i \})$ are not convex polygons. Moreover, our proofs are limited to two dimensions since there is no equivalent statement of Fejes T\'oth's Moment Lemma in higher dimensions (due to the lack of regular polytopes in higher dimensions).

Next we recall the proof of Theorem \ref{thm:constr} due to Gruber \cite{Gru_Fej} for the case $\alpha=0$, which we will adapt to prove
Theorem \ref{thm:pen} (and consequently Theorem \ref{thm:constr}) for general \blue{$\alpha\in(-\infty,\overline{\alpha}]$}.
It can be shown that the function
\[
(m,n)\mapsto g(m,n):=m^2 c_n
\]
is convex. (Note that $n \mapsto c_n$ can be extended from a function on $\mathbb{N} \cap [3,\infty)$ to a function on $[3,\infty)$; see Lemma \ref{lem:Fejes}.) 
If $\mu$ is a minimizer of $W_2^2(\ca_\Omega,\cdot)$ subject to the constraint $N_\mu \le L$, then clearly 
$N_\mu=L$ (assuming that $L$ is an integer) since we get the best constrained approximation of $\ca_\Omega$ by taking as many Dirac masses as possible. By convexity,
\begin{equation}\label{eq:ci}
m^2 c_n = g(m,n) \geq g\left(\tfrac 1L,6\right) + \nabla g\left(\tfrac 1L,6\right)\cdot \left(m-\tfrac 1L,n-6 \right) = \frac{c_6}{L^2} + \frac{2c_6}{L} \left(m - \tfrac 1L \right) + \frac{\kappa}{L^2} (n-6)
\end{equation}
where $\kappa:=g_n(1,6)=\partial_n c_n |_{n=6}<0$.
Combining equations \eqref{eq:ineqW} and \eqref{eq:ci} gives
\begin{equation}
\label{eq:14}
W_2^2(\ca_\Omega,\mu) \geq \sum_{i=1}^{L} \left( \frac{c_6}{L^2} + \frac{2c_6}{L} \left(m_i - \tfrac 1L \right) + \frac{\kappa}{L^2} (n_i-6) \right)
= \frac{c_6}{L} + \frac{\kappa}{L^2}\sum_{i=1}^{L}(n_i-6)
\end{equation}
since $\sum_i m_i = |\Omega| = 1$.
Euler's formula for planar graphs implies that the average number of edges in any partition of the unit square $\Omega$ by convex polygons is less than or equal to $6$: $\frac 1L \sum_{i} n_i \le 6$; see Lemma \ref{lemma:PCP}. Therefore, by equation \eqref{eq:14} and since $\kappa < 0$,
\[
L \, \mathrm{m_c}(0,L) \ge c_6.
\]
This is the lower bound in Theorem \ref{thm:constr} for the case $\alpha=0$.
A matching upper bound can be obtained in the limit $L \to \infty$ by taking $\mu=\sum_{i=1}^L \tfrac 1L \delta_{z_i}$ where $z_i$ lie on a regular triangular lattice.

In  \cite{BouPelRop} Gruber's strategy was generalized to prove Theorem \ref{thm:pen} for the case $\alpha=0.5$ and $f=1$. 
Thanks to the results of \cite{BouJimMaha} and our results in Section \ref{sec:asympt_pen}, it follow that Theorem \ref{thm:pen} holds
for all \blue{$\alpha\in(-\infty,0.5]$} and all lower semi-continuous $f$ satisfying \eqref{eq:f_ass}.  
In this paper we extend these ideas further to prove Theorem \ref{thm:pen} for the case $\alpha=0.583$, and hence all \blue{$\alpha\in(-\infty,0.583]$}.
First of all, we rescale the square $Q$ as follows (see Remark \ref{rem:rescaling}):
\[
Q_{\delta,\alpha}:=
V_{\delta,\alpha}^{\frac{1}{2}}Q, \qquad V_{\delta,\alpha}= \left( \frac{c_6}{\delta(1-\alpha)} \right)^{\frac{1}{2-\alpha}}.
\]
The rescaling factor is chosen in such a way that 
a discrete measure supported at the centres of regular hexagons of \emph{unit area} is asymptotically optimal. 
Up to a multiplicative factor, the rescaled energy is
\[
\f_{\delta,\alpha}(\mu):= \frac{c_6}{1-\alpha} \sum_{i=1}^{N_\mu} m_i^\alpha
	+ W^2_2(\ca_{Q_{\delta,\alpha}},\mu),
\]
where $\ca_{Q_{\delta,\alpha}}$ denotes the characteristic function of the square $Q_{\delta,\alpha}$.
Here $\mu$ is a Borel measure on \blue{$Q_{\delta,\alpha}$} of the form $\mu=\sum_{i=1}^{N_\mu} m_i \delta_{z_i}$ with $\sum_{i=1}^{N_\mu} m_i = V_{\delta,\alpha}$.
By \eqref{eq:ineqW} we have
\[
\f_{\delta,\alpha}(\mu) \geq \sum_{i=1}^{N_\mu} \left[ \frac{c_6}{1-\alpha} m_i^\alpha + m_i^2 c_{n_i} \right]
= \sum_{i=1}^{N_\mu} g_\alpha(m_i,n_i),
\]
where
\[
g_\alpha(m,n):=\frac{c_6}{1-\alpha} m^\alpha + m^2 c_n.
\]
Unfortunately, for $\alpha \in (0,1)$, $g_\alpha$ is not convex. Our first main technical result is to show that
 for $\alpha=\overline{\alpha}$ there exists $m_0>0$ such that 
the following `convexity inequality' holds for
all $m\geq m_0$, $n \in \mathbb{N} \cap [3,\infty)$:
\begin{equation}\label{eq:ciintro}
g_{\overline{\alpha}}(m,n)\geq g_{\overline{\alpha}}(1,6) + \nabla g_{\overline{\alpha}}(m,n)\cdot(m-1,n-6).
\end{equation}
See Lemma \ref{lem:inequality}, Corollary \ref{cor:reduction} and Corollary \ref{Cor:CI}.
Our second main technical result (Lemma \ref{lem:lowerboundmass}) is to show that if $\mu=\sum_{i=1}^{N_\mu} m_i \delta_{z_i}$ minimizes $\f_{\delta,\alpha}$, then
\begin{equation}
\label{eq:Sketch3}
m_i > 2.0620 \times 10^{-4} > m_0.
\end{equation}
Therefore minimizers satisfy the convexity inequality \eqref{eq:ciintro}, and the proof of Theorem \ref{thm:pen} now follows using Gruber's strategy.

To be precise, we are only able to prove the inequality \eqref{eq:Sketch3} for particles $z_i$ that are not too close to the boundary (\blue{Lemma \ref{lem:lowerboundmass}(i)}). Nevertheless, we are able to prove a worse lower bound on the mass $m_i$ of particles near the boundary (\blue{Lemma \ref{lem:lowerboundmass}(ii)}), which is still sufficient to show that the number of particles near the boundary is asymptotically negligible. This fixes what appears to be a gap in the proof in \cite{BouPelTheil}, where it was tacitly assumed that all of the particles were sufficiently far from the boundary of the rescaled domain (at least distance $3.2143$; see the proof of \cite[Lemma 7]{BouPelTheil}).

The idea of the proof of \eqref{eq:Sketch3} is to compare the energy of a minimizer $\mu$ with that of a competitor $\widetilde{\mu}$ that is obtained by gluing the smallest particle of $\mu$ with one of its neighbours. The proofs of 
\eqref{eq:ciintro} and \eqref{eq:Sketch3} require some delicate positivity estimates. As in the proof of \cite{BouPelTheil}, we also use computer evaluation at several points in the proof to check the sign of some explicit numerical constants (that are much larger than machine precision).


\subsection{Literature on crystallization, optimal partitions and quantization}\label{sec:other}
Our work belongs to the very active research programme of establishing crystallization results for nonlocal interacting particle systems. This problem is known as the \emph{crystallization conjecture} \cite{BlaLew}.
Despite experimental evidence that many particle systems, such as atoms in metals, have periodic ground states, until recently there
were few rigorous mathematical results.
Results in one dimension include \cite{BlaLeB,GarRad} and results in two dimensions include 
\cite{AuYeungFrieseckeSchmidt12,Betermin16,BetDeLPet,BeterminKnupfer,BeterminPetrache2019,BouPelTheil,DeLFri,ELi09,HeiRad, PetracheSerfaty2020,Rad81,The}.
\blue{Let us recall that a central open problem in mathematical physics is to establish the optimality of the Abrikosov (triangular) lattice for the Ginzburg-Landau energy \cite{SerfatySandier12}}.
In three dimensions there are few rigorous results.
Even establishing the optimal configuration of just \emph{five} charges on a sphere was only achieved in 2013 via a computer-assisted proof
\cite{Schwartz13}. 
The Kepler conjecture about optimal sphere packing was also computer-assisted \cite{Hales05,HalesEtAl10},
while the optimal sphere \emph{covering} remains to this day unknown.
In even higher dimensions (in particular 8 and 24), there start to be more rigorous results again, e.g., \cite{CohnEtAl2017,CohnKumarMillerRadchenkoViazovska,Viazovska}.
For a thorough survey of recent crystallization results for nonlocal particle systems see \cite{BlaLew} and  \cite{Ser_Review1}.

Our result also falls into the field of \emph{optimal partitions} (see Remark \ref{remark:equiv}). The optimality of hexagonal tilings, or \emph{Honeycomb conjectures}, have been proved for example by \cite{BucurFragala,BucurFragala2019,BucurFragalaVelichkovVerzini,Hales99}.
Kelvin's problem of finding the optimal foam in 3D (the `three-dimensional Honeycomb conjecture') remains to this day unsolved; for over 100 years it was believed that truncated octahedra gave the optimal tessellation, until the remarkable discovery of a better tessellation by Weire and Phelan \cite{WeairePhelan94}.

Finally, our result also belongs to the field of \emph{optimal quantization} or \emph{optimal sampling} \cite{GL,Gr04}, \cite[Section 33]{Gruber07}, which concerns the best approximation of a probability measure by a discrete probability measure. The most commonly used notion of best approximation is the Wasserstein distance.  
This problem has been studied by a wide range of communities including applied analysis \cite{BouJimMaha,ButSan,Iacobelli2018}, computational geometry \cite{DuVen}, discrete geometry \cite{ConwaySloane,Gr04}, and probability \cite{GL}. Applications include 
optimal location of resources \cite{BollobasStern72}, signal and image compression \cite{DuGunzburgerJuWang2006,GrayNeuhoff}, numerical integration \cite[Section 2.3]{PagesPhamPrintems2004}, 
mesh generation \cite{DuWan,MerigotMeyronThibert2018},
finance \cite{PagesPhamPrintems2004}, materials science \cite[Section 3.2]{BourneKokRoperSpanjer2020}, and particle methods for PDEs (sampling the initial distribution) \cite[Example 7.1]{BourneEganPelloniWilkinson}.

It is well known that if $\mu=\sum_{i=1}^N m_i z_i$ is a minimizer of $W_2(f  \dd x,\cdot)$, then the particles $z_i$ generate a  \emph{centroidal Voronoi tessellation} (CVT) \cite{DuVen,LiuEtAl2009}, which means the particles $z_i$ lie at the centroids of their Voronoi cells. 
Numerical methods for computing CVTs include 
Lloyd's algorithm \cite{DuVen} and quasi-Newton methods \cite{LiuEtAl2009}.
More generally, minimizers of the penalized energy $\mu \mapsto \delta \sum_i m_i^\alpha + W_2^2(f \dd x,\mu)$ generate \emph{centroidal Laguerre tessellations} (see Remark \ref{rem:CLTs}). Numerical methods for solving the constrained and penalized quantization problems include \cite{BouRop} (which was used to produce Figures \ref{fig:numerics} and \ref{fig:numerics2}) and \cite{LevyCentroidalPower}.

There is a large literature on optimal CVTs of $N$ points (global minima of $\mu \mapsto W_2(\ca_\Omega ,\mu)$ subject to $\# \mathrm{supp}(\mu)=N$).
According to \emph{Gersho's conjecture} (see \cite{Ger}), minimizers correspond to regular tessellations consisting of the repetition of a single polytope whose shape depends only on the spatial dimension. In two dimensions the polytope is a hexagon  \cite{BollobasStern72,Fej_Stability,FEJ,Gru_Fej,BolMor,New} and moreover the result holds for any $p$-Wasserstein metric, $p \in [1,\infty)$. Gersho's conjecture is open in three dimensions, although it is believed that the optimal CVT is a tessellation by truncated octahedra, which is generated by the body-centred cubic (BCC) lattice. Some numerical evidence for this is given in \cite{DuWang2005}, and
in \cite{BarSlo} it was proved that the BCC lattice is optimal among lattices (but we do not know whether the optimal configuration is in fact a lattice). Geometric properties of optimal CVTs in 3D were recently proved in \cite{ChoLu}, who also suggested a strategy for a computed-assisted proof of Gersho's conjecture.


\subsection{Organization of the paper}
In Section \ref{sec:prel} we recall some basic notions from optimal transport theory 
and convex geometry.
In Section \ref{sec:asympt_con} we recall from \cite{BouJimMaha} the result \eqref{eq:Prop3.11(i)} for the case $d=p=2$, namely the 
scaling of the minimum value of the energy 
 for the constrained problem \eqref{eq:m_c}. In Section  \ref{sec:asympt_pen} we derive the scaling of the minimum value of the energy for the penali\blue{z}ed problem \eqref{eq:m_p}. These results give the optimal \emph{scaling} of the minimum values of the constrained and penalized energies, but they do not give the optimal constants. In Section \ref{sec:proofpenalized} we identify the optimal constant for the penalized problem (which proves Theorem \ref{thm:pen}) and in Section \ref{sec:proofconstrained} we identify the optimal constant for the constrained problem (which proves Theorem \ref{thm:constr}). In Section \ref{sec:alpha1} we prove the asymptotic crystallization result for all \blue{$\alpha\in(-\infty,1)$} under an additional assumption.
\blue{Finally, Section \ref{sec:stability} is devoted to the proof of Theorem \ref{thm:stability}.}


\section{Preliminaries}\label{sec:prel}

\subsection{Main assumptions}
We assume that $\Omega\subset\R^2$ is the closure of an open and bounded set, and $f\in L^1(\Omega)$ is a lower semi-continuous function satisfying $f\geq c>0$ and
\[
\int_\Omega f(x) \,\dd x =1.
\]


\subsection{Notation}

Define \blue{$\R^+:=(0,\infty)$.}
For a Lebesgue-measurable set $A\subset\R^2$, we denote by $|A|$ its area and by $\ca_A$ its characteristic function.
We let $\mathcal{M}(X)$ denote the set of non-negative finite Borel measures on $X\subset\R^d$ and $\mathcal{P}(X)\subset\mathcal{M}(X)$ denote the set of probability measures on $X$.
Moreover, we let $\mathcal{M}_\mathrm{d}(X)\subset\mathcal{M}(X)$ be the following set of discrete measures:
\[
\mathcal{M}_\mathrm{d}(X):=\left\{ \sum_{i=1}^{N} m_i \delta_{z_i} : 
	N \in \mathbb{N}, \, m_i > 0, \, z_i \in X, \, z_i \ne z_j \textrm{ if } i \ne j \right\}.
\]
Recall that $\mathcal{P}_\mathrm{d}(X)$ denotes the set of discrete probability measures, $\mathcal{P}_\mathrm{d}(X)=\mathcal{M}_\mathrm{d}(X)\cap \mathcal{P}(X)$,
and that, for $\mu \in \Pd$, $N_\mu := \# \mathrm{supp}(\mu)$.
For brevity, in an abuse of notation, we denote the preimage of a singleton set $\{ z\} \subset X$ under a map $T:X \to X$ by $T^{-1}(z)$ instead of $T^{-1}(\{z\})$.


\subsection{Facts from optimal transport theory and convex geometry}
We start by recalling the characterization of solutions of the semi-discrete transport problem  \eqref{eq:Wassdist_p} for the case $p=2$. The following result goes back to \cite{Aurenhammer98} and  is now well-known in the optimal transport community; see for example \cite{LevySchwindt2018,Merigot2011,MerigotThibert2021}.

\begin{lemma}[Characterization of the optimal transport map]\label{lem:cellspolygons}
Let $U\subset\R^2$ be a convex polygon, $\mu = \sum_{i=1}^{N_\mu} m_i \delta_{z_i} \in\mathcal{P}_\mathrm{d}(U)$,
 $g \in L^1(U;\mathbb{R}^+)$, $\int_U g \, \dd x =1$, and $W_2(g,\mu)$ be the Wasserstein metric
\[
W_2(g,\mu) \:=\inf\left\{ \int_U |x-T(x)|^2 g(x)\,\dd x \, : \,
T : U\to \{z_i\}_{i=1}^{N_\mu} \textrm{ is Borel}, \, \int_{T^{-1}(z_i) }g(x) \,\dd x = m_i \, \forall \, i \right\}^{\frac 12}.
\]
Then the infimum is attained and the minimizer
$T:U\to\{z_i\}_{i=1}^{N_\mu}$ is unique (up to a set of measure zero).
Moreover, by possibly modifying $T$ on a set of measure zero,
there exists $(w_1,\dots, w_{N_\mu})\in\R^{N_\mu}$ such that
\[
\overline{T^{-1}(z_i)}=\left\{ z\in U : |z-z_i|^2 - w_i \leq |z-z_j|^2 - w_j 
	\text{ for all } j=1,\dots, N_\mu \right\}.
\]
\end{lemma}

\begin{remark}[Laguerre cells]
The previous lemma implies that the partition $\{\overline{T^{-1}(z_i)} \}_{i=1}^{N_\mu}$ is the \emph{Laguerre tessellation} or \emph{power diagram} generated by the weighted points $\{ (z_i, w_i) \}_{i=1}^{N_\mu}$; see \cite{AurenhammerKleinLee13,MerigotThibert2021}.
The sets $\overline{T^{-1}(z_i)}$
are convex polygons, known as \emph{Laguerre cells} or \emph{power cells}.
\end{remark}

We now recall a classical result by L.~Fejes T\'{o}th (see \cite[p.~198]{FEJ}), which says that the minimal second moment of an $n$-gon is greater than or equal to the minimal second moment of a \emph{regular} $n$-gon of the same area:

\begin{lemma}[Fejes T\'oth's Moment Lemma]\label{lem:Fejes}
For $n\in\N$, $n\geq 3$, define
\[
c_n:=\inf\left\{ \min_{\xi\in\R^2} \int_P |x-\xi|^2 \dd x : P \text{ is an $n$-gon},\, |P|=1  \right\}.
\]
Then the infimum is attained by a regular $n$-gon.
Consequently a direct calculation gives
\begin{equation}\label{eq:cn}
c_n=\frac{1}{2n}\left( \frac{1}{3}\tan\frac{\pi}{n}+\cot\frac{\pi}{n} \right).
\end{equation}
\end{lemma}

\begin{remark}\label{rem:cnerscaled}
Note that a change of variables gives 
\[
\inf\left\{ \min_{\xi\in\R^2} \int_P |x-\xi|^2 \dd x : P \text{ is an $n$-gon},\, |P|=m  \right\}
	=c_n m^2
\]
for all $m>0$.
\end{remark}

We extend the definition of $c_n$ to all $n\in[3,\infty)$
using equation \eqref{eq:cn}. 
\blue{Its main properties are stated in the next result, whose proof is a direct computation (see \cite{Gru_Fej}).}

\begin{lemma}[Properties of $c_n$]\label{lem:cn}
The function $n \mapsto c_n$, $n \in [3,\infty)$, is convex and decreasing. Moreover
\[
\lim_{n\to\infty}c_n=c_\infty:=\frac{1}{2\pi}.
\]
\end{lemma}

Finally, we recall one more result from convex geometry, which follows from Euler's polytope formula. It is proved for example in \cite[Lemma 4]{BouPelTheil} or \cite[Lemma 3.3]{BolMor}.

\blue{
\begin{lemma}[Partitions by convex polygons]\label{lemma:PCP}
Let $U \subset \mathbb{R}^2$ be a convex polygon with at most 6 sides. 
In any partition of $U$ 
by convex polygons, the average number of edges per polygon is less than or equal to $6$. 	
\end{lemma}
}


\section{Scaling of the asymptotic quantization error}\label{sec:asympt}

\subsection{The constrained optimal location problem}
\label{sec:asympt_con}

We report here a result about the asymptotic quantization error from \cite{BouJimMaha}.

\begin{definition}[Young measures]
Given $\varepsilon>0$ and a measure $\mu\in\Pd$ of the form
\[
\mu=\sum_{i=1}^{N_\mu} m_i \blue{\delta_{z_i}},
\]
define the measures $\rho(\mu)\in\dmr$ and
$\lambda_\varepsilon(\mu)\in\mathcal{M}_\mathrm{d}(\Omega\times\R^+)$ by
\[
\rho(\mu):=\sum_{i=1}^{N_\mu} m_i \delta_{m_i},
\quad\quad
\lambda_\varepsilon(\mu):=\sum_{i=1}^{N_\mu} m_i\delta_{\left( z_i,\frac{m_i}{\varepsilon^2} \right)}.
\]
Observe that the first marginal of $\lambda_\varepsilon(\mu)$ is $\mu$ and that the second marginal of $\lambda_1(\mu)$ is $\rho(\mu)$.
\end{definition}

In order to define the cell formula for the asymptotic quantization error, we need to introduce the following metric on the space of probability measures. Given $\rho_1, \rho_2\in \pr$, we define
\[
\mathrm{d}_{\mathrm{BL}} (\rho_1, \rho_2):= \sup \left\{ \int_{\mathbb{R}^+} \varphi \, d(\rho_1-\rho_2) :  \varphi\in \mathrm{Lip}(\R^+), \,  |\varphi|_{\infty} + \mathrm{Lip}(\varphi) \leq 1 \right\},
\]
where $\mathrm{Lip}(\R^+)$ is the space of Lipschitz continuous functions on $\R^+$ and $ \mathrm{Lip}(\varphi)$ denotes the Lipschitz constant of $\varphi$.
It is well known that $\mathrm{d}_{\mathrm{BL}}$ metrizes tight convergence (see \cite[Theorem 11.3.3]{DM}).

The energy density of the asymptotic quantization error is introduced as follows.

\begin{definition}[Cell formula]
Given $t>0$ and $\rho\in\pr$, define
\[
G_t(\rho):=\inf_{k>0}\frac{S_t(\rho, Q_k)}{k^2}
\]
where $Q_k:=[-k/2,k/2]^2\subset\R^2$ and
\[
S_t(\rho, Q_k):=\inf\left\{ W^2_2(\ca_{Q_k}, \mu) +\frac{k^2}{t^2}\mathrm{d}_{\mathrm{BL}}(\rho,\rho(\mu))
	: \mu\in\mathcal{P}_\mathrm{d}(Q_k) \right\}.
\]
Define $G:\pr\to\R$ by
\[
G(\rho):=\sup_{t>0} G_t(\rho) = \lim_{t \to 0} G_t(\rho).
\]
\end{definition}

Given $\lambda \in \mathcal{M}(\Omega \times \mathbb{R}^+)$, let $\pi_1 \# \lambda$ denote its first marginal, where $\pi_1 : \Omega \times \mathbb{R}^+ \to \Omega$ is the projection $\pi_1(x,t)=x$.
One of the main results of Bouchitt\'{e}, Jimenez and Mahadevan \cite[Theorem 3.1]{BouJimMaha} is the following:

\begin{theorem}[Gamma-limit of the quantization error]\label{thm:Bouchitte}
For $\varepsilon>0$, let
\begin{equation}\label{eq:funcBou}
\mathcal{E}_{\varepsilon}(\lambda):=
\begin{cases}
\frac{1}{\varepsilon^2} W^2_2(f, \mu) & \text{if } \lambda=\lambda_\varepsilon(\mu), \, \mu\in\Pd,
\\
+ \infty & \text{otherwise.}
\end{cases}
\end{equation}
Then $\mathcal{E}_{\varepsilon}\stackrel{\Gamma}{\rightarrow}\mathcal{E}_{0}$ with respect to tight convergence on $\mathcal{M}(\Omega\times\R^+)$ where
\[
\mathcal{E}_{0}(\lambda):=
\begin{cases}
\displaystyle\int_\Omega G(\lambda^x) \dx & \text{if } 
\lambda=f  \dd x\otimes\lambda^x, 
\\
+ \infty & \text{otherwise,}
\end{cases}
\]
where $f \dd x\otimes \lambda^x$ denotes the disintegration of $\lambda$ with respect to $f \dd x$; see \cite[Theorem 2.28]{AFP}.
\end{theorem}

Bouchitt\'{e}, Jimenez and Mahadevan used Theorem \ref{thm:Bouchitte} to prove the following result about the scaling of the asymptotic quantization error for the constrained optimal location problem; see \cite[Lemma 3.10, Proposition 3.11(i)]{BouJimMaha}\blue{.}

\begin{corollary}[Asymptotic quantization error for the constrained problem]\label{cor:monotonicityCc}
For all \blue{$\alpha\in(-\infty,1)$},
\[
\lim_{L\to\infty}\left[L^{\frac{1}{1-\alpha}}\,\mathrm{m_c}(\alpha,L)\right]
	= C_{\mathrm{c}}(\alpha)\left( \int_\o f(x)^{\frac{1}{2-\alpha}} \, \dd x \right)^{\frac{2-\alpha}{1-\alpha}}
\]
where
\begin{equation}
\label{eq:DefCcalpha}
C_{\mathrm{c}}(\alpha) := \min \left\{ G(\rho) : \rho\in\mathcal{P}(\R^+),\,
	\int_0^\infty t^{\alpha-1} \,\mathrm{d}\rho(t)\leq 1 \right\}.
\end{equation}
Moreover, the function $\alpha\mapsto C_{\mathrm{c}}(\alpha)$ is non-increasing.
\end{corollary}

Note that the result proved in \cite{BouJimMaha} holds more generally: in any dimension, for any $p$-Wasserstein metric, and for more general entropies.

\begin{remark} 
	\label{remark:3_5}
	Let $Q \subset \mathbb{R}^2$ be a unit square. Taking $\Omega=Q$ and $f=\ca_Q$ in Corollary \ref{cor:monotonicityCc} yields
\begin{equation}\label{eq:charCc}
C_{\mathrm{c}}(\alpha)=\lim_{L\to\infty}  L^{\frac{1}{1-\alpha}} \, \inf\left\{ W^2_2(\ca_Q,\mu) :
	\mu\in\mathcal{P}_{\mathrm{d}}(Q),\,\, \sum_{i=1}^{N_\mu} m_i^\alpha\leq L \right\}.
\end{equation}
\end{remark}

\begin{remark}[Optimal constant]\label{rem:Gdelta1}
	The constant $C_{\mathrm{c}}(\alpha)$ in Corollary \ref{cor:monotonicityCc} was known explicitly for the case \blue{$\alpha \in (-\infty,0]$}, where \blue{$C_\mathrm{c}(\alpha) = C_\mathrm{c}(0)=G(\delta_1)=c_6$ for all $\alpha \le 0$}. We briefly recall the proof: By Fejes T\'oth's Theorem on Sums of Moments \cite{Gru_Fej},
	\[
	c_6= C_{\mathrm{c}}(0) 
	\stackrel{\eqref{eq:DefCcalpha}}{\le} G(\delta_1)
	 \le  c_6
	\]
	where the final inequality follows from \cite[Prop.~3.2(iv)]{BouJimMaha}. \blue{Therefore $C_{\mathrm{c}}(0) =c_6$. In addition, $ C_{\mathrm{c}}(\alpha) \ge  C_{\mathrm{c}}(0)  = c_6$ for all $\alpha \le 0$ by the monotonicity of the map $\alpha \mapsto  C_{\mathrm{c}}(\alpha)$ (see Corollary \ref{cor:monotonicityCc}). On the other hand, $ C_{\mathrm{c}}(\alpha) \le c_6$ by 
	 \cite[Prop.~3.2(iv)]{BouJimMaha}. We conclude that $ C_{\mathrm{c}}(\alpha) = c_6$ for all $\alpha \le 0$, as required.} 
	One of our contributions is to prove that $C_{\mathrm{c}}(\alpha)=c_6$ for all \blue{$\alpha\in(-\infty,0.583]$}; see Section \ref{sec:proofconstrained}.
\end{remark}


\subsection{The penalized optimal location problem}
\label{sec:asympt_pen}

Here we prove analogous results to those presented in the previous section.

\begin{definition}[penalized energy]
	\label{def:PenEn}
	Let $\delta>0$ and \blue{$\alpha\in(-\infty,1)$}.
	Define $\e_{\delta,\alpha} : \mathcal{P}_\mathrm{d}(\Omega) \to [0,\infty)$ by
	\[
	\e_{\delta,\alpha}(\mu):=\delta\sum_{i=1}^{N_\mu} m_i^\alpha + W^2_2(f,\mu)
	\]
	where $\mu = \sum_{i=1}^{N_\mu} m_i \delta_{z_i}$.
\end{definition}

\begin{proposition}[Gamma-limit of the penalized energy]\label{prop:Glimpen}
Let $\delta>0$, \blue{$\alpha\in(-\infty,1)$} and 
\[
\varepsilon_{\delta,\alpha}:=\left(\frac{\delta(1-\alpha)}{c_6}\right)^{\frac{1}{2(2-\alpha)}}.
\]
Define the rescaled penalized energy $\widetilde{\e}_{\delta,\alpha}:\mathcal{M}(\o\times\R^+)\to[0,\infty]$ by
	\[
	\widetilde{\e}_{\delta,\alpha}(\lambda):=
	\begin{cases}
	\varepsilon_{\delta,\alpha}^{-2} \, \e_{\delta,\alpha}(\mu)
	& \text{if } \lambda=\lambda_{\varepsilon_{\delta,\alpha}}(\mu), \, \mu\in\Pd,
	\\
	+ \infty & \text{otherwise.}
	\end{cases}
	\]
Then $\widetilde{\e}_{\delta,\alpha}\stackrel{\Gamma}{\rightarrow}\mathcal{G}_\alpha$ as $\delta\to0$ with respect to tight convergence on $\mathcal{M}(\o\times\R^+)$ where
\[
\mathcal{G}_\alpha(\lambda):=
\begin{cases}
\displaystyle\int_\o\left[ G(\lambda^x) + f(x)\frac{c_6}{1-\alpha}\int_0^{\infty} t^{\alpha-1}
	\,\dd \lambda^x(t)  \right] \, \dd x & \text{if } 
	\lambda=f \dd x\otimes\lambda^x,
	\\
+\infty & \text{otherwise}.
\end{cases}
\]
\end{proposition}

To prove Proposition \ref{prop:Glimpen} we need the following technical result from \cite[Lemma 6.3]{BouJimMaha}, which says 
that we can modify $\mu$ to remove asymptotically small Dirac masses (as $\delta \to 0$)
without increasing the energy \blue{$\widetilde{\e}_{\delta,\alpha}(\mu)$} too much. 

\begin{lemma}
	\label{lem:remove_small_masses}
	Let $\lambda=f \dd x \otimes\lambda^x \in \mathcal{M}(\o\times\R^+)$ satisfy $		\mathcal{E}_0(\lambda)<\infty$.
	Then, for every $\gamma>1$, there exists a decreasing sequence $(t_k)_{k\in\N} 	\subset (0,\infty)$, $t_k \to 0$, and a doubly-indexed sequence $(  \lambda_	\varepsilon^k )_{\varepsilon>0, k\in\N}\subset \mathcal{M}(\o\times\R^+)$ 		
	satisfying the following:
	\begin{itemize}
	\item[(i)] $\lambda_\varepsilon^k$ is supported in $\Omega\times [t_k,\infty)$;
	\item[(ii)] \blue{$\limsup_{k \to\infty}  \limsup_{\varepsilon\to0} \|\lambda^k_	\varepsilon - \lambda \| 
	= 0$, where $\| \cdot \|$ denotes the total variation norm on the space of signed measures on $\Omega \times \R^+$};
	\item[(iii)] for all \blue{$\alpha\in(-\infty,1)$}, $k \in \mathbb{N}$,
			\[ \limsup_{\varepsilon \to 0} \int_{\Omega \times (0,\infty)} t^{\alpha-1} \, \dd \lambda_\varepsilon^k(x,t)  \le 
			\int_{\Omega \times (0,\infty)} t^{\alpha-1} \, \dd \lambda (x,t);
			\] 
	\item[(iv)] there exist\blue{s} $\mu_\varepsilon^k \in \mathcal{P}_\mathrm{d}(\Omega)$ 	such that $\lambda_\varepsilon^k=\lambda_\varepsilon(\mu_\varepsilon^k)$ and  
	\[
	\limsup_{k\to\infty} \, \limsup_{\varepsilon\to0} \, \varepsilon^{-2} \, W_2^2(f,	\mu^k_\varepsilon) \leq \gamma  \int_\Omega G(\lambda^x) \dx .
	\]
	\end{itemize}
\end{lemma}

\begin{proof}[Proof of Proposition \ref{prop:Glimpen}]
For $\mu\in\Pd$, $\lambda=\lambda_{\varepsilon_{\delta,\alpha}}(\mu)$,  we can write
\begin{align*}
\widetilde{\e}_{\delta,\alpha}(\lambda) 
	&=\frac{c_6}{1-\alpha}\varepsilon_{\delta,\alpha}^{2(1-\alpha)}\sum_{i=1}^{N_\mu}m_i^\alpha
		+ \frac{1}{\varepsilon_{\delta,\alpha}^2}W^2_2(f,\mu) 
		\\
&=\frac{c_6}{1-\alpha}\int_{\Omega\times(0,\infty)} t^{\alpha-1}\,
	\dd \lambda_{\varepsilon_{\delta,\alpha}}(\mu)(x,t)
	+ \frac{1}{\varepsilon_{\delta,\alpha}^2}W^2_2(f,\mu).
\end{align*}
Since the function $t\mapsto t^{\alpha-1}$ is unbounded, and thus the first term of $\widetilde{\e}_{\delta,\alpha}(\lambda)$ is not continuous in $\lambda$, the $\Gamma$-convergence result does not follow directly from Theorem \ref{thm:Bouchitte} and the stability of $\Gamma$-limits under continuous perturbations.
We therefore reason as follows.

\emph{Step 1: liminf inequality.}
Fix $(\delta_n)_{n\in\N}$ with $\delta_n\to0$ as $n\to\infty$.
Let $\lambda\in\mathcal{M}(\o\times\R^+)$ and $(\lambda_n)_{n\in\N}\subset\mathcal{M}(\o\times\R^+)$ satisfy $\lambda_n\to\lambda$ tightly. Without loss of generality we can assume that
\begin{equation}
	\label{eq:BoundedEnergy}
\liminf_{n\to\infty} \widetilde{\e}_{\delta_n,\alpha}(\lambda_n)<\infty.
\end{equation}
Therefore there exists $(\mu_n)_{n\in\N}\subset\dmo$ such that
$\lambda_n=\lambda_{\varepsilon_{\delta_n,\alpha}}(\mu_n)$.
Observe that $\pi_1 \# \lambda_n = \mu_n$. 
By \eqref{eq:BoundedEnergy}, and since $W_2$ metrizes weak convergence of measures \cite[Theorem 5.9]{S}, then $\mu_n \to f \dd x$ as $n \to \infty$.
Therefore $\pi_1 \# \lambda = \lim_{n \to \infty} \pi_1 \# \lambda_n = f \dd x$. By the Disintegration Theorem \cite[Theorem 2.28]{AFP} there exists $\lambda^x \in \mathcal{M}(\mathbb{R}^+)$ satisfying 
$\lambda=f \dd x\otimes\lambda^x$.

For $M>0$ define the continuous bounded function $g_M:(0,\infty) \to \mathbb{R}$ by 
\[
g_M(t):=\min\{ t^{\alpha-1}, M\}.
\]
Then, by using the liminf inequality of Theorem \ref{thm:Bouchitte}, we get
\begin{align*}
\liminf_{n\to\infty} \widetilde{\e}_{\delta_n,\alpha}(\lambda_n)
	&\geq 
	\liminf_{n\to\infty}\left(
	\frac{c_6}{1-\alpha}\int_{\Omega\times(0,\infty)} g_M(t)\,
	\dd \lambda_{\varepsilon_{\delta_n,\alpha}}(\mu_n)(x,t)
	+ \frac{1}{\varepsilon_{\delta_n,\alpha}^2}W^2_2(f,\mu_n)\right)
	 \\
&\geq \frac{c_6}{1-\alpha} \int_\o  \left( \int_0^{\infty} g_M(t)\,\dd \lambda^x(t) \right) f(x) \, \dd x
	+ \int_\o  G(\lambda^x) \,\dd x.
\end{align*}
Since the function $g_M$ is non-negative and pointwise non-decreasing in $M$, we obtain the liminf inequality by passing to the limit $M \to \infty$ using the Monotone Convergence Theorem.

\emph{Step 2: limsup inequality.}
Let $\lambda=f \dd x \otimes\lambda^x \in \mathcal{M}(\o\times\R^+)$
satisfy
$\mathcal{G}_\alpha(\lambda) < \infty$,
which implies that $\mathcal{E}_0(\lambda)<\infty$.
Let $\gamma >1$. By Lemma \ref{lem:remove_small_masses}(iii),(iv), 
there exists a decreasing sequence $(t_k)_{k\in\N} \subset (0,\infty)$, a sequence $\delta_n \to 0$, and 
 a doubly-indexed sequence $(  \lambda_{\varepsilon_{\delta_n,\alpha}}^k )_{n, k\in\N}\subset \mathcal{M}(\o\times\R^+)$ such that
\[
		\limsup_{k\to\infty} \, \limsup_{n \to \infty} \,
		\widetilde{\e}_{\delta_n,\alpha} \left( \lambda_{\varepsilon_{\delta_n,\alpha}}^k \right) \leq 
		\frac{c_6}{1-\alpha} 	\int_{\Omega \times (0,\infty)} t^{\alpha-1} \, \dd \lambda (x,t)
		+ \gamma  \int_\Omega G(\lambda^x) \dx \le \gamma  \mathcal{G}_\alpha (\lambda).
\]
By a diagonalization argument and Lemma \ref{lem:remove_small_masses}(ii), we can find a subsequence
$\delta_n$ (not relabelled) such that  $\lambda_{\varepsilon_{\delta_{n},\alpha}}^{k_n} \to \lambda$ tightly as $n \to \infty$ and 
\[
\limsup_{n \to \infty}  \,
\widetilde{\e}_{\delta_n,\alpha} \left( \lambda_{\varepsilon_{\delta_{n},\alpha}}^{k_n} \right)
\le \gamma  \mathcal{G}_\alpha (\lambda).
\]
Since $\gamma>1$ is arbitrary, the limsup inequality follows.
\end{proof}

\begin{corollary}[Asymptotic quantization error for the penalized problem]\label{cor:monoCp}
For all \blue{$\alpha\in(-\infty,1)$},
\begin{equation}\label{eq:asymptCp}
\lim_{\delta\to0}\left[\left(\frac{c_6}{\delta(1-\alpha)}\right)^{\frac{1}{2-\alpha}}
	\mathrm{m_p}(\alpha,\delta)\right] = C_{\mathrm{p}}(\alpha) \int_\o f(x)^{\frac{1}{2-\alpha}} \,\dd x
\end{equation}
where
\begin{equation}\label{eq:Cp}
C_{\mathrm{p}}(\alpha):=\min\left\{ G(\rho)+\frac{c_6}{1-\alpha}\int_0^{\infty} t^{\alpha-1} \,\dd\rho(t)
			: \rho\in\pr \right\}.
\end{equation}
\end{corollary}

\begin{proof}
\emph{Step 1.} The functional $\widetilde{\e}_{\delta,\alpha}$ has at least one minimizer (by \cite[Theorem 2.1]{ButSan}), sequences $(\lambda_\delta)$ with bounded energy have tightly convergent subsequences (by \cite[Theorem 3.1(i)]{BouJimMaha}), and $\widetilde{\e}_{\delta,\alpha}$ $\Gamma$-converges to $\mathcal{G}_\alpha$ (by Proposition \ref{prop:Glimpen}). Therefore a standard result in the theory of $\Gamma$-convergence implies that the minimum value of  $\widetilde{\e}_{\delta,\alpha}$ converges to the minimum value of $\mathcal{G}_\alpha$:
\[
\lim_{\delta\to0}\left[\left(\frac{c_6}{\delta(1-\alpha)}\right)^{\frac{1}{2-\alpha}}
	\mathrm{m_p}(\alpha,\delta)\right]
	=\min_{\mathcal{M}(\o\times\R^+)} \mathcal{G}_\alpha.
\]
We are thus left with proving that
\begin{equation}\label{eq:claim_minimum}
\min_{\mathcal{M}(\o\times\R^+)} \mathcal{G}_\alpha
	= C_{\mathrm{p}}(\alpha) \int_\o f(x)^{\frac{1}{2-\alpha}} \,\dd x,
\end{equation}
where $C_{\mathrm{p}}(\alpha)$ is defined in \eqref{eq:Cp}.

\emph{Step 2.} For each $x \in \Omega$, define $\mathcal{G}_\alpha^x:\mathcal{P}(\mathbb{R^+}) \to \mathbb{R}$ by
\[
\mathcal{G}^x_\alpha(\rho) := G(\rho)+f(x)\frac{c_6}{1-\alpha}\int_0^{\infty} t^{\alpha-1} \, \dd\rho(t).
\]
By definition, if $\lambda=f \dd x\otimes\lambda^x$,
\begin{equation}
\label{eq:GGx}
\mathcal{G}_\alpha(\lambda)=\int_\Omega \mathcal{G}^x_\alpha(\lambda^x) \, \dd x.
\end{equation} 
For each $x \in \Omega$, $\mathcal{G}^x_\alpha$ is lower semi-continuous since $G$ is lower semi-continuous \cite[Prop.~3.2(i)]{BouJimMaha} and since $\rho \mapsto \int_0^\infty 
 t^{\alpha-1} \, \dd\rho(t)$ is lower semi-continuous \cite[Lemma 1.6]{S}. By \cite[Prop.~3.2(iv)]{BouJimMaha},
 \[
 \gamma_{2,2} \int_0^\infty t \, \dd \rho(t) \le G(\rho)
 \]
 where $\gamma_{2,2} = \int_{B_1(0)} |x|^2 \, \dd x$. Therefore, for each $x \in \Omega$, minimising sequences for $\mathcal{G}^x_\alpha$ are tight. Consequently $\mathcal{G}^x_\alpha$ has at least one minimizer.
 
We claim that there exits a Borel measurable function $x \mapsto \rho^x \in \mathcal{P}(\mathbb{R}^+)$, $x \in \Omega$, such that
\begin{equation}
\label{eq:Selection}
\mathcal{G}^x_\alpha(\rho^x) = \min_{\mathcal{P}(\mathbb{R}^+)} \mathcal{G}^x_\alpha.
\end{equation}
This will follow from Aumann's Selection Theorem (see \cite[Theorem 6.10]{FL}) once we prove that the graph of the multifunction $\Gamma: \Omega \to 2^{\mathcal{P}(\mathbb{R}^+)} \setminus \emptyset$, defined by $\Gamma(x):=\mathrm{argmin} \, \mathcal{G}^x_\alpha$, belongs to $\mathcal{B}(\Omega)\otimes \mathcal{B}(\mathcal{P}(\R^+))$, the product $\sigma$-algebra of the Borel sets of $\Omega$ and the Borel sets of $\mathcal{P(\R^+)}$.
To prove this, we define the function $\Psi:\mathbb{R}^+ \times \mathcal{P}(\mathbb{R}^+) \to \mathbb{R}$ by 
\[
\Psi(s,\rho):= G(\rho)+s \frac{c_6}{1-\alpha}\int_0^{\infty} t^{\alpha-1} \, \dd\rho(t).
\]
In the following, the target space $\mathbb{R}$ will always be equipped with the Borel $\sigma$-algebra.
For each $\rho \in \mathcal{P}(\mathbb{R}^+)$, the function $s \mapsto \Psi(s,\rho)$ is continuous. For each $s \in \mathbb{R}^+$, the function $\rho \mapsto \Psi(s,\rho)$ is lower semi-continuous and hence $\mathcal{B}(\mathcal{P(\R^+)})$-measurable. Therefore $\Psi$ is a Carath\'eodory function and hence $\mathcal{B}(\mathbb{R}^+) \otimes \mathcal{B}(\mathcal{P(\R^+)})$-measurable (see, e.g., \cite[Lemma 4.51]{AliprantisBorder}). Define the composite function 
	$\Phi:\Omega \times \mathcal{P}(\mathbb{R}^+) \to \mathbb{R}$ by 
	\[
	\Phi(x,\rho):=\Psi(f(x),\rho) = \mathcal{G}^x_\alpha(\rho). 
	\]
	This is  $\mathcal{B}(\Omega) \otimes \mathcal{B}(\mathcal{P(\R^+)})$-measurable since $f$ and $\Psi$ are Borel measurable. 
	
We claim that the map $x \mapsto \min_{\rho\in\pr}\Phi(x,\rho)$ is  $\mathcal{B}(\Omega)$-measurable. 
Then $\overline{\Phi}:\Omega \times \mathcal{P}(\mathbb{R}^+) \to\R$ defined by $\overline{\Phi}(x,\nu):=\min_{\rho\in\pr}\Phi(x,\rho)$ is $\mathcal{B}(\Omega)\otimes \mathcal{B}(\mathcal{P(\R^+)})$-measurable (since $\overline{\Phi}$ is constant in its second argument). The required  $\mathcal{B}(\Omega)\otimes \mathcal{B}(\mathcal{P(\R^+)})$-measurability of graph of the multifunction $\Gamma$ then follows by noticing that
\[
\mathrm{graph}(\Gamma) = (\Phi - \overline{\Phi})^{-1}(\{0\}).
\]
To show that $x \mapsto \min_{\rho\in\pr}\Phi(x,\rho)$ is  $\mathcal{B}(\Omega)$-measurable, we write it as the composite function 
$x \mapsto f(x) \mapsto \min_{\rho\in\pr}\Psi(f(x),\rho)$. This is $\mathcal{B}(\Omega)$-measurable since $x \mapsto f(x)$ is 
$\mathcal{B}(\Omega)$-measurable and $s \mapsto  \min_{\rho\in\pr}\Psi(s,\rho)$ is the pointwise infimum of a family of continuous functions, hence upper semi-continuous and $\mathcal{B}(\mathbb{R}^+)$-measurable. This completes the proof that there exists a 
Borel measurable function $x \mapsto \rho^x \in \mathcal{P}(\mathbb{R}^+)$ satisfying \eqref{eq:Selection}.

\emph{Step 3.}
Define $\lambda :=f\dd x \otimes \rho^x \in \mathcal{M}(\o\times\R^+)$, where $\rho^x$ is the minimizer of $\mathcal{G}^x_\alpha$ constructed in Step 2
(note that $\lambda$ is well\blue{-}defined by \cite[Definition 2.27]{AFP} since $x \mapsto \rho^x$ is Borel measurable and hence Lebesgue measurable). By equations \eqref{eq:GGx}, \eqref{eq:Selection}, $\lambda$ is a minimizer of $\mathcal{G}_\alpha$ and
\begin{equation}\label{eq:integral_minimum}
\min_{\mathcal{M}(\o\times\R^+)} \mathcal{G}_\alpha =
	\int_\Omega  \min\left\{ \mathcal{G}^x_\alpha(\rho) : \rho\in\pr  \right\} \, \dd x.
\end{equation}
We now rewrite
\[
\min\left\{ \mathcal{G}^x_\alpha(\rho) : \rho\in\pr  \right\}
\]
as follows.
For $a>0$, define the dilation $L^a:\R^+\to\R^+$ by $L^a(t):=at$.
Let $\rho\in\pr$ and consider the push-forward $\rho_a:=L^a\#\rho \in\pr$. 
It was proved in \cite[Prop.~3.2(ii)]{BouJimMaha} that 
\begin{equation}
\label{eq:Grhoa}
G(\rho_a)=a G(\rho).
\end{equation} 
Note that 
\begin{equation}
\label{eq:rescalegterm}
\int_0^{\infty} t^{\alpha-1} \,\dd \rho_a(t) = a^{\alpha-1} \int_0^{\infty} t^{\alpha-1} \,\dd\rho(t).
\end{equation}
Fix $x\in\o$ and let $a:=f(x)^{-\frac{1}{2-\alpha}}$. 
By \eqref{eq:Grhoa} and \eqref{eq:rescalegterm} we can write
\begin{equation}\label{eq:rescaling}
\mathcal{G}^x_\alpha(\rho) = G(\rho)+f(x)\frac{c_6}{1-\alpha}\int_0^{\infty} t^{\alpha-1} \, \dd\rho(t)
	=f(x)^{\frac{1}{2-\alpha}}
		\left[ G(\rho_a)+\frac{c_6}{1-\alpha}\int_0^{\infty} t^{\alpha-1} \, \dd\rho_a(t)  \right].
\end{equation}
Therefore, by using \eqref{eq:rescaling} and the definition of $C_{\mathrm{p}}(\alpha)$ (see \eqref{eq:Cp}), we have that
\begin{align}
\nonumber
 \min\left\{ \mathcal{G}^x_\alpha(\rho) : \rho\in\pr  \right\} 
&  =
\min\left\{ G(\rho)+f(x)\frac{c_6}{1-\alpha}\int_0^{\infty} t^{\alpha-1} \, \dd\rho(t) 
	: \rho\in\pr  \right\} 
	\\
	\label{eq:min_form}
	& = f(x)^{\frac{1}{2-\alpha}}  C_{\mathrm{p}}(\alpha)
\end{align}
for all $x\in\o$.
 By combining \eqref{eq:integral_minimum} and \eqref{eq:min_form} we prove \eqref{eq:claim_minimum} and conclude the proof.
\end{proof}

\begin{remark}\label{rem:Cp}
	Let $Q \subset \mathbb{R}^2$ be a unit square. Taking $\Omega=Q$ and $f=\ca_Q$ in Corollary \ref{cor:monoCp} yields
\begin{equation}\label{eq:charCp}
C_{\mathrm{p}}(\alpha)=\
\lim_{\delta\to0}\left(\frac{c_6}{\delta(1-\alpha)}\right)^{\frac{1}{2-\alpha}}
	\min\left\{  \delta\sum_{i=1}^{N_\mu} m_i^\alpha + W^2_2(\ca_Q,\mu) : \mu\in\mathcal{P}_{\mathrm{d}}(Q) \right\}.
\end{equation}
\end{remark}

By Corollary \ref{cor:monoCp}, in order to prove Theorem \ref{thm:pen} it is sufficient to prove that
\[
C_{\mathrm{p}}(\alpha) = \frac{2-\alpha}{1-\alpha} \, c_6
\]
for all \blue{$\alpha\in(-\infty,\overline{\alpha}]$}.
The next result, which is analogous to the monotonicity of the map $\alpha\mapsto C_{\mathrm{c}}(\alpha)$, means that
 in order to prove Theorem \ref{thm:pen} for all \blue{$\alpha\in(-\infty,\overline{\alpha})$}, it is sufficient to prove it for the single value $\alpha=\overline{\alpha}$.

\begin{lemma}[Monotonicity of the constant $C_{\mathrm{p}}$]\label{lem:monoCp}
Assume that for some \blue{$\tilde{\alpha}\in (-\infty,1)$},
\[
C_{\mathrm{p}}(\tilde{\alpha}) = \frac{2-\tilde{\alpha}}{1-\tilde{\alpha}} \, c_6.
\]
Then
\[
C_{\mathrm{p}}(\alpha) = \frac{2-\alpha}{1-\alpha} \, c_6
\]
for every $\alpha<\tilde{\alpha}$.
\end{lemma}

\begin{proof}
Recall from Remark \ref{rem:Gdelta1} that $c_6 = G(\delta_1)$.
By \eqref{eq:Cp}, for all \blue{$\alpha\in(-\infty,1)$},
\begin{equation}
\label{eq:3.11a}
C_{\mathrm{p}}(\alpha) \le G(\delta_1) + \frac{c_6}{1-\alpha}
\int_0^\infty t^{\alpha-1} \, \dd \delta_1(t)
= c_6 +  \frac{c_6}{1-\alpha}
 = \frac{2-\alpha}{1-\alpha} \, c_6.
\end{equation}
Write $C_{\mathrm{p}}(\alpha)=\min \{ F(\alpha,\rho) : \rho \in \mathcal{P}(\mathbb{R}^+)\}$, where
\[
F(\alpha,\rho) = G(\rho)+\frac{c_6}{1-\alpha}\int_0^{\infty} t^{\alpha-1} \, \dd\rho(t).
\]
For all $ \rho \in \mathcal{P}(\mathbb{R}^+)$, \blue{$\alpha\in(-\infty,1)$}, we have
\[
F(\alpha,\rho) - \blue{F(\alpha,\delta_1)}
= G(\rho) - G(\delta_1) + c_6 \int_0^\infty \frac{t^{\alpha-1}-1}{1-\alpha} \, \dd \rho(t).
\]
Let $\phi(t,\alpha)=(t^{\alpha-1}-1)/(1-\alpha)$ denote the integrand on the right-hand side. Then
\[
(1-\alpha)^2 \partial_\alpha \phi (t,\alpha)
= t^{\alpha-1}( 1+(1-\alpha)\ln t )-1 =: \psi(t,\alpha). 
\]
Since $\partial_\alpha \psi(t,\alpha)=(1-\alpha)(\ln t)^2 t^{\alpha-1}\geq 0$ and $\psi(t,1)=0$, we obtain that $\psi(t,\alpha)\leq 0$ for all $t\in(0,\infty)$, \blue{$\alpha\in(-\infty,1)$}. Therefore $\partial_\alpha \phi \le 0$ and $\phi$ is non-increasing in $\alpha$. Consequently the map
$\alpha \mapsto F(\alpha,\rho) -  \blue{F(\alpha,\delta_1)}$ is non-increasing. 

Let \blue{$\tilde{\alpha}\in(\infty,1)$} be such that
\[
C_{\mathrm{p}}(\tilde{\alpha}) = \frac{2-\tilde{\alpha}}{1-\tilde{\alpha}} \, c_6
= \blue{F(\tilde{\alpha},\delta_1)}.
\]
For all $\rho \in \mathcal{P}(\mathbb{R}^+)$ and all \blue{$\alpha\in(-\infty,\tilde{\alpha}]$},
\[
 F(\alpha,\rho) -  \blue{F(\alpha,\delta_1)} \ge
  F(\tilde{\alpha},\rho) -  \blue{F(\tilde{\alpha},\delta_1)}
  \ge C_{\mathrm{p}}(\tilde{\alpha})  - \blue{F(\tilde{\alpha},\delta_1)} = 0.
\]
Taking the infimum over $\rho$ gives 
\begin{equation}
\label{eq:3.11b}
C_{\mathrm{p}}(\alpha) \ge 
 \blue{F(\alpha,\delta_1)}
= \frac{2-\alpha}{1-\alpha} \, c_6.
\end{equation}
Combining \eqref{eq:3.11a} and \eqref{eq:3.11b} completes the proof.
\end{proof}


\section{The penalized optimal location problem: Proof of Theorem \ref{thm:pen}}\label{sec:proofpenalized}

This section is devoted to the proof of Theorem \ref{thm:pen}. In particular, we prove that
\[
C_{\mathrm{p}}(\overline{\alpha}) = \frac{2 - \overline{\alpha}}{1-\overline{\alpha}} \, c_6. 
\]
The upper bound is easy to prove:
\begin{lemma}[Upper bound on $C_{\mathrm{p}}(\alpha)$]
	\label{lem:UBCp}
For all \blue{$\alpha\in(-\infty,1)$},
\[
C_{\mathrm{p}}(\alpha) \le \frac{2 - \alpha}{1-\alpha} \, c_6. 
\]
\end{lemma}
\begin{proof} Recall from Remark \ref{rem:Gdelta1} that $G(\delta_1)=c_6$. Therefore
\[
C_\mathrm{p}(\alpha) \stackrel{\eqref{eq:Cp}}{\le} G(\delta_1) + \frac{c_6}{1-\alpha} \int_{0}^{\infty} t^{\alpha-1} \, \dd \delta_1(t) 
= c_6 + \frac{c_6}{1-\alpha} 
= \frac{2 - \alpha}{1-\alpha} \, c_6.
\]
\end{proof}

\begin{remark}[Direct proof of the upper bound]
	Lemma \ref{lem:UBCp} can also be proved without using the result from \cite{BouJimMaha} that $G(\delta_1)=c_6$. Instead we can start from equation \eqref{eq:charCp} and directly build a sequence of asymptotically optimal competitors $\mu_\delta$ supported on a subset of a triangular lattice. This is done by covering the square $Q$ with regular hexagons of a suitable size and making the heuristic calculation from Remark \ref{rem:rescaling} rigorous; cf.~\cite[Lemma 8]{BouPelTheil}.
\end{remark}

The matching lower bound 
\begin{equation}
\label{eq:CLB}
C_{\mathrm{p}}(\overline{\alpha}) \ge \frac{2 - \overline{\alpha}}{1-\overline{\alpha}} \, c_6
\end{equation}
requires much more work.
Owing to Corollary \ref{cor:monoCp} and Remark \ref{rem:Cp} we can assume without loss of generality that
\[
\o= Q=[-1/2,1/2]^2, \qquad f = \ca_Q.
\]
We will do this throughout the rest of the paper.

\subsection{Rescaling of the energy and the energy of a partition}
To prove \eqref{eq:CLB} it is convenient to rescale the domain $Q$.
As $\delta \to 0$, the optimal masses $m_i$ in \eqref{eq:charCp} go to $0$. Following \cite{BouPelTheil}, instead of keeping the domain $Q$ fixed as $\delta \to 0$, we blow up $Q$ in such a way that the optimal masses $m_i$ tend to $1$. The following definition is motivated by the heuristic calculation given in Remark \ref{rem:rescaling}.
  
\begin{definition}[Rescaled domain and energy]
For \blue{$\alpha\in(-\infty,1)$} and $\delta>0$, define
\[
V_{\delta,\alpha}:=\left(\frac{c_6}{\delta(1-\alpha)} \right)^{\frac{1}{2-\alpha}}
\]
and define the rescaled square domain $Q_{\delta,\alpha}$ by
\[
Q_{\delta,\alpha}:= V_{\delta,\alpha}^{\frac{1}{2}} Q.
\]
Moreover, define the set of admissible discrete measures $\mathcal{A}_{\delta,\alpha}$ by
\[
\mathcal{A}_{\delta,\alpha}:=
	\left\{ \mu = \sum_{i=1}^{N_\mu} m_i \delta_{z_i} : N_\mu \in \mathbb{N}, \, m_i > 0, \, \sum_{i=1}^{N_\mu}  m_i = V_{\delta,\alpha}, \, z_i \in Q_{\delta,\alpha}, \, z_i \ne z_j \textrm{ if } i \ne j \right\}
\]
and define the rescaled energy $\f_{\delta,\alpha}:\mathcal{A}_{\delta,\alpha} \to \mathbb{R}$ by
\begin{equation}
\label{eq:Fda}
\f_{\delta,\alpha}(\mu):= \frac{c_6}{1-\alpha} \sum_{i=1}^{N_\mu} m_i^\alpha + W_2^2(\ca_{Q_{\delta,\alpha}},\mu).
\end{equation}
\end{definition}

\begin{remark}[Restating $C_\mathrm{p}(\alpha)$ in terms of $\f_{\delta,\alpha}$]\label{rem:rescaled}
Let $\mu = \sum_i m_i \delta_{z_i}\in \mathcal{P}_{\mathrm{d}}(Q)$. For \blue{$\alpha\in(-\infty,1)$} and $\delta>0$, define 
$\widetilde{z}_i:=V_{\delta,\alpha}^{1/2} z_i$, $\widetilde{m}_i:=V_{\delta,\alpha} m_i$, and 
$\widetilde{\mu}_{\delta,\alpha}:=\sum_i \widetilde{m}_i \delta_{\widetilde{z}_i} \in\mathcal{A}_{\delta,\alpha}$.
Then
\begin{equation}\label{eq:relationrescaled}
V^{-2}_{\delta,\alpha}\f_{\delta,\alpha}(\widetilde{\mu}_{\delta,\alpha})
		= \delta\sum_i m_i^\alpha + W^2_2(\ca_{Q},\mu).
\end{equation}
Therefore, 
by \eqref{eq:charCp} and \eqref{eq:relationrescaled},  
\begin{equation}
\label{eq:suff}
C_\mathrm{p}(\alpha)  = \lim_{\delta\to 0} V^{-1}_{\delta,\alpha} \min \left\{
\f_{\delta,\alpha}(\mu) : \mu\in\mathcal{A}_{\delta,\alpha} \right\}.
\end{equation}
\end{remark}

We now state two first-order necessary conditions for minimizers of $\mathcal{F}_{\delta,\alpha}$. For a proof see, for instance, \cite[Theorem 4.5]{BouPelRop}.

\begin{lemma}[Properties of minimizers]
	\label{lem:propoptcell}
Let $\mu = \sum_{i=1}^{N_\mu} m_i \delta_{z_i}\in \mathcal{A}_{\delta,\alpha}$ be a minimizer of $\mathcal{F}_{\delta,\alpha}$.
Let $T$ be the optimal transport map defining 
$W_2(\ca_{Q_{\delta,\alpha}},\mu)$ and let $(w_1,\ldots,w_{N_\mu})$ be the weights of the corresponding Laguerre tessellation (see Lemma \ref{lem:cellspolygons}).
	\begin{itemize}
		\item[(i)] For all $i \in \{1,\dots,N_\mu\}$, we have
		\[
		w_i=-\frac{\alpha}{1-\alpha} c_6 m_i^{\alpha-1}.
		\]
		\item[(ii)]   
		The point $z_i$ is the centroid of the Laguerre cell $\overline{T^{-1}(z_i)}$, namely
		\[
		z_i = \frac{1}{m_i}\int_{T^{-1}(z_i)} x\,\dd x.
		\]
		In particular, $z_i\in T^{-1}(z_i)$.
	\end{itemize}
\end{lemma}

\begin{remark}[Centroidal Laguerre tessellations]
\label{rem:CLTs}	
Lemma \ref{lem:propoptcell}(ii) implies that minimizers of $\mathcal{F}_{\delta,\alpha}$ generate \emph{centroidal Laguerre tessellations},  which means that the particles $z_i$ lie at the centroids of their Laguerre cells $T^{-1}(z_i)$ \cite{BouRop,LevyCentroidalPower}. 
\end{remark}

In the following it will also be convenient to reason from a geometrical point of view.
Each $\mu\in\mathcal{A}_{\delta,\alpha}$ induces a partition of $Q_{\delta,\alpha}$ by the Laguerre cells $\overline{T^{-1}(z_i)}$, where $T$ is the optimal transport map defining $W_2(\ca_{Q_{\delta,\alpha}},\mu)$. We define a wider class of partitions as follows:

\begin{definition}[Admissible partitions]
	\label{def:Sdeltaalpha}
Let $\mathcal{S}_{\delta,\alpha}$ denote the family of partitions of $Q_{\delta,\alpha}$ of the form $\mathcal{C}=(C_1,\dots, C_k)$ where $k \in \mathbb{N}$,  
$C_i \subset Q_{\delta,\alpha}$ is measurable, and $\sum_{i=1}^k \ca_{C_i} = 1$ a.e.~in $Q_{\delta,\alpha}$.
\end{definition}

The advantage of working with partitions instead of measures is that it allows us to localise the nonlocal energy $\f_{\delta,\alpha}$.

\begin{definition}[Optimal partitions]
	\label{def:optpart}
Define the partition energy $F:\mathcal{S}_{\delta,\alpha} \to \mathbb{R}$ by
\[
F(\mathcal{C}):= \sum_{i=1}^k \left( \frac{c_6}{1-\alpha} |C_i|^\alpha
	+ \int_{C_i} |x-\xi_{C_i}|^2 \, \dd x \right),
\]
where $\xi_{C_i}:=\frac{1}{|C_i|} \int_{C_i} x\, \dd x$ is the centroid of $C_i$, $i\in\{1,\dots,k\}$.
We say that $\mathcal{C}\in\mathcal{S}_{\delta,\alpha}$ is an \emph{optimal partition} if it minimizes $F$.
\end{definition}

To each $\mu\in\mathcal{A}_{\delta,\alpha}$ it is possible to associate an element of $\mathcal{S}_{\delta,\alpha}$ as follows:

\begin{definition}[Partition associated to a discrete measure]
	\label{def:partmeas}
Let $\mu\in\mathcal{A}_{\delta,\alpha}$ be of the form $\mu=\sum_{i=1}^{N_\mu} m_i \delta_{z_i}$.
Define $\mathcal{C}^\mu = (C^\mu_1,\dots,C^\mu_{N_\mu})\in\mathcal{S}_{\delta,\alpha}$ by
\[
C^\mu_i := \overline{T^{-1}(z_i)}
\]
for all $i\in\{1,\dots,N_\mu\}$, where $T$ is the optimal transport map defining $W_2(\ca_{Q_{\delta,\alpha}},\mu)$.
\end{definition}

\begin{remark}[Equivalence of the partition formulation]
	\label{remark:equiv}
It was proved in \cite[p.~125]{BouPelTheil} that
\[
\min \{ \f_{\delta,\alpha}(\mu) : \mu\in\mathcal{A}_{\delta,\alpha} \}
 = \min\{ F(\mathcal{C}) : \mathcal{C}\in \mathcal{S}_{\delta,\alpha}  \}.
\]
\end{remark}

Let $\mu=\sum_{i=1}^{N_\mu} m_i \delta_{z_i}$
be a minimizer of $\f_{\delta,\alpha}$.
For $i\in\{1,\dots,N_\mu\}$, let $n_i$ denote the number of edges of $C_i^\mu$.
Then we can bound the energy from below as follows:
\begin{align*}
\f_{\delta,\alpha}(\mu) = F(\mathcal{C}^\mu)
	 =\sum_{i=1}^{N_\mu} \left(  \frac{c_6}{1-\alpha}|C^\mu_i|^\alpha + \int_{C^\mu_i} |x-z_i|^2 \, \dd x \right)  
\geq \sum_{i=1}^{N_\mu} \left( \frac{c_6}{1-\alpha}
|C^\mu_i|^\alpha + c_{n_i} |C^\mu_i|^2 \right) 
\end{align*}
by Lemma \ref{lem:Fejes} and Remark \ref{rem:cnerscaled}. 
For \blue{$\alpha\in(-\infty,1)$}, define $g_\alpha:[0,\infty) \times [3,\infty) \to \mathbb{R}$ by
\[
g_\alpha(m,n):=\frac{c_6}{1-\alpha} m^\alpha+c_n m^2.
\]
In this notation the lower bound above becomes
\begin{equation}
\label{eq:Putin}
\f_{\delta,\alpha}(\mu) = F(\mathcal{C}^\mu) 
\geq \sum_{i=1}^{N_\mu} g_\alpha(|C_i^\mu|, n_i).
\end{equation}
In the following section we study the function $g_\alpha$.


\subsection{The convexity inequality}\label{sec:convexity}

We start by proving a technical result that plays the role of a convexity inequality for $g_\alpha$.
We want to show that, for large enough values of $m$,
\[
g_\alpha(m,n)\geq g_\alpha(1,6) + \nabla g_\alpha(1,6)\cdot(m-1,n-6).
\]
Writing this out explicitly gives
\begin{equation}\label{eq:ineq}
\frac{c_6}{1-\alpha}m^\alpha + c_n m^2 \geq c_6\left(\frac{2-\alpha}{1-\alpha}\right)m+\kappa(n-6),
\end{equation}
where
\[
\kappa:=\partial_n c_n |_{n=6}=\frac{2\pi}{243}-\frac{5\sqrt{3}}{324}<0.
\]
For $\alpha\in(0,1)$, define the function $h_\alpha:[0,\infty)\times[3,\infty)\to\mathbb{R}$ to be the difference between $g_\alpha$ and its tangent plane approximation at $(1,6)$:
\begin{align*}
h_\alpha(m,n) := &
g_\alpha(m,n) - \left( g_\alpha(1,6) + \nabla g_\alpha(1,6)\cdot(m-1,n-6) \right)
\\
 = &
\frac{c_6}{1-\alpha}m^\alpha + c_n m^2 -c_6\left(\frac{2-\alpha}{1-\alpha}\right)m-\kappa(n-6).
\end{align*}
\blue{Note that in this section we restrict our attention to $\alpha >0$ without loss of generality since by the monotonicity result (Lemma \ref{lem:monoCp}) in the end we only need to consider $\alpha= \overline{\alpha}=0.583$.}
The typical behaviour of the function $m\mapsto h_\alpha(m,n)$ is depicted in Figure \ref{fig:Falpha}. Our aim is to prove that $h_\alpha(m,n)$ is non-negative for all integers $n \ge 3$ for large enough values of $m$, as suggested by the figure.

\begin{figure}
\includegraphics[scale=0.4]{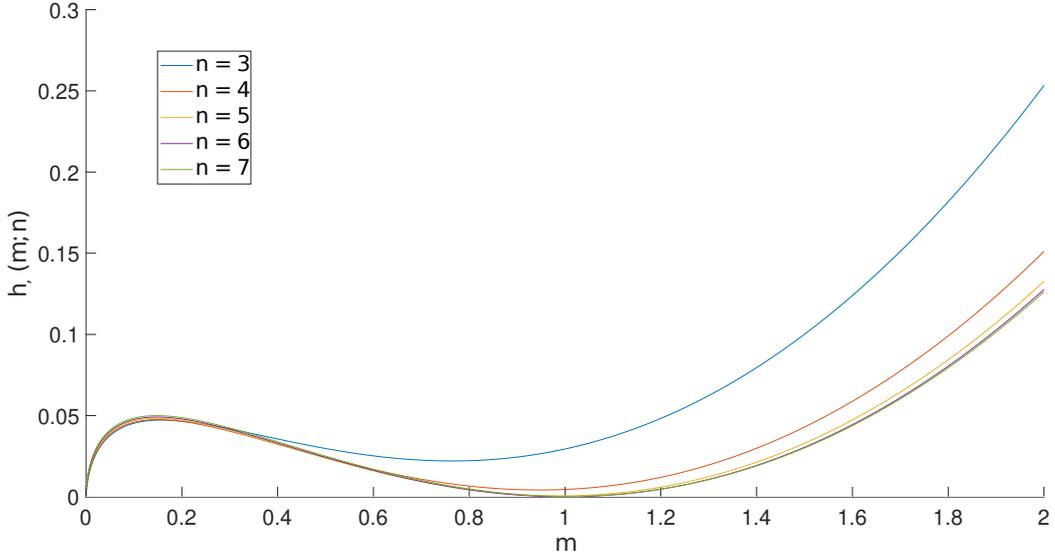}
\caption{Behaviour of the map $m\mapsto h_\alpha(m,n)$ for different values of $n$ and for $\alpha=0.583$. Note that $h_\alpha(0,n)<0$ for $n \in \{3,4,5\}$, even if this is not evident from the figure.}
\label{fig:Falpha}
\end{figure}

\begin{lemma}[Positivity of $h_{\overline{\alpha}}$]\label{lem:inequality}
Let $n\ge 3$ be a integer.
If $h_{\overline{\alpha}}(m_1,n)\geq0$ for some $m_1\geq0$, then $h_{\overline{\alpha}}(m,n)\geq0$ for all $m\geq m_1$.
\end{lemma}

Before proving this, we prove the following easy but important corollary, which allows us to reduce the proof of the convexity inequality $h_{\overline{\alpha}}(\cdot,n)\ge 0$ for all integers $n \ge 3$ to the finite number of cases $n \in \{3,4,5\}$:
\begin{corollary}[Reduction to $n \in \{3,4,5\}$]
	\label{cor:reduction}
Let $n \ge 6$ be an integer. Then
	$h_{\overline{\alpha}}(m,n)\geq0$ for all $m \ge 0$.
\end{corollary}
\begin{proof}
Observe that $h_{\overline{\alpha}}(0,n)=-\kappa(n-6) \ge 0$ for all $n \ge 6$. Therefore the result follows immediately from Lemma \ref{lem:inequality}.
\end{proof}

On the other hand, $h_{\overline{\alpha}}(0,n)<0$ for $n \in \{3,4,5\}$. This is why in the next section we will need to prove a lower bound on the masses $m_i$ of minimizers of $\f_{\delta,\overline{\alpha}}$ to ensure the validity of the convexity inequality \eqref{eq:ineq}.

\begin{proof}[Proof of Lemma \ref{lem:inequality}]
\emph{Step 1.} First we study the shape of the function  $m\mapsto h_\alpha(m,n)$. 
\blue{In particular, we show that it has exactly one local minimum point.}
Its derivative is
\begin{equation}
\label{eq:h_m}
\partial_m h_\alpha(m,n) = \frac{\alpha}{1-\alpha} c_6 m^{\alpha-1} + 2 c_n m
	- c_6\frac{2-\alpha}{1-\alpha}.
\end{equation}
\blue{It is easy to see that $\partial_m h_\alpha(m,n)$ is strictly convex in $m$ \blue{(since $\alpha > 0$)}. Therefore $m \mapsto \partial^2_{mm} h_\alpha(m,n)$ is increasing and so  $m\mapsto \partial_m h_\alpha(m,n)$ has at most one critical point. On the other hand, $\lim_{m \to 0}\partial_m h_\alpha(m,n)=+\infty$, $\lim_{m \to \infty}\partial_m
 h_\alpha(m,n)=+\infty$.} Therefore  
 $m\mapsto \partial_m h_\alpha(m,n)$ has exactly one critical point and 
  $m\mapsto h_\alpha(m,n)$ has at most two critical points. 
  
  \blue{
  Next we prove that  $m\mapsto h_\alpha(m,n)$ has exactly two critical points.
  It is sufficient to prove that 
  \[
  \partial_m h_\alpha(1/2,n) < 0 
  \]
  for all $n \ge 3$ (since $\lim_{m \to 0}\partial_m h_\alpha(m,n) >0$ and $\lim_{m \to \infty}\partial_m
 h_\alpha(m,n) > 0$). We have
 \begin{equation}
 \label{eq:qui}
  \partial_m h_\alpha(1/2,n) = c_6 \frac{\alpha 2^{1-\alpha}-2+\alpha}{1-\alpha} + c_n.
 \end{equation}
 It is straightforward to check that
 \begin{equation}
 \label{eq:quo}
 \frac{d}{d \alpha} \frac{\alpha 2^{1-\alpha}-2+\alpha}{1-\alpha}
 = \frac{\psi(\alpha)}{(1-\alpha)^2}
 \end{equation}
 with
 \[
 \psi(\alpha) = 2^{1-\alpha}(1-\alpha \ln 2 + \alpha^2 \ln 2) - 1.
 \]
 Differentiating again gives
 \[
 \psi'(\alpha) = q(\alpha) 2^{1-\alpha} \ln 2, \qquad q(\alpha):=- \alpha^2 \ln 2 + (2+\ln 2)\alpha -2.
 \]
 The concave quadratic polynomial $q$ has roots $\alpha=1$ and $\alpha = 2/\ln 2 > 2$. Therefore, for all $\alpha \in (0,1)$, $q(\alpha)<0$, $\psi'(\alpha)<0$, and
 \begin{equation}
 \label{eq:qua}
 \psi(\alpha) \ge \psi(1) = 0.
 \end{equation}
 From equations \eqref{eq:qui}--\eqref{eq:qua} we conclude that
$\alpha \mapsto \partial_m h_\alpha(1/2,n)$
 is increasing. Therefore
\begin{align*}
\partial_m h_\alpha(1/2,n)  \le \lim_{\alpha \to 1} \partial_m h_\alpha(1/2,n)
 = \lim_{\alpha \to 1} c_6 \frac{\alpha 2^{1-\alpha}-2+\alpha}{1-\alpha} + c_n
 = (\ln 2 - 2)c_6 + c_n.
\end{align*}
Since $n \mapsto c_n$ is decreasing (see Lemma \ref{lem:cn}),
\[
\partial_m h_\alpha(1/2,n) \le  (\ln 2 - 2)c_6 + c_3 <  -0.017 < 0
\]
as required.

We have shown that $m\mapsto h_\alpha(m,n)$ has exactly two critical points.
 The smallest critical point is a local maximum point and the largest critical point is a local minimum point (since  $m\mapsto \partial^2_{mm} h_\alpha(m,n)$ is increasing). Let $\widetilde{m}(\alpha,n)$ denote the local minimum point.} To prove the lemma it is sufficient to prove that 
\begin{equation}
\label{eq:non-neg}
\varphi(\alpha,n) := h_{\alpha}(\widetilde{m}(\alpha,n), n) \ge 0
\end{equation}
for $\alpha=\overline{\alpha}$ and for all $n \in \mathbb{N} \cap [3,\infty)$.

\emph{Step 2.} Next we prove \eqref{eq:non-neg} for the case $n=6$.
A direct computation shows that $1$ is \blue{a} local minimum point of $m\mapsto h_\alpha(m,6)$ for all $\alpha\in(0,1)$. Therefore
$\widetilde{m}(\alpha,6)=1$ and 
$\varphi(\alpha,6)= h_{\alpha}(1,6) =0$. 

\emph{Step 3.} Next we prove \eqref{eq:non-neg} for the case  $n \in \{3,4,5,7\}$. 
Let
\[
m_1(n):=
\begin{cases}
0.764 & n=3,
\\
0.946 & n=4,
\\
0.98705 & n=5,
\\
1.00516 & n=7,
\end{cases}
\quad\quad
m_2(n):=
\begin{cases}
0.765 & n=3,
\\
0.947 & n=4,
\\
0.9871 & n=5,
\\
1.00518 & n=7.
\end{cases}
\]
Then numerically evaluating $\partial_m h_{\overline{\alpha}}$ gives
\[
\partial_m h_{\overline{\alpha}}(m_1(n),n)\leq
\begin{cases}
-8\times 10^{-6} & n=3,\\
-2\times 10^{-5} & n=4,\\
-1\times 10^{-6} & n=5,\\
-4\times 10^{7} & n=7,
\end{cases}
\quad\quad
\partial_m h_{\overline{\alpha}}(m_2(n),n)\geq
\begin{cases}
3\times 10^{-5} & n=3,\\
8\times 10^{-6} & n=4,\\
4\times 10^{-7} & n=5,\\
1\times 10^{-7} & n=7.
\end{cases}
\]
Let $n \in \{3,4,5,7\}$.
By the Intermediate Value Theorem, the map $m \mapsto \partial h_{\overline{\alpha}}(m,n)$ has a root between $m_1(n)$ and $m_2(n)$. Moreover, since $\partial_m h_{\overline{\alpha}}(m_1(n),n) < 0$, we have bracketed the largest root $\widetilde{m}(\overline{\alpha},n)$:  $m_1(n)<\widetilde{m}(\overline{\alpha},n)<m_2(n)$.
Therefore
\begin{align*}
\varphi(\overline{\alpha},n)\geq \frac{c_6}{1-\overline{\alpha}}m_1(n)^{\overline{\alpha}} + c_n m_1(n)^2
- c_6\left(\frac{2-\overline{\alpha}}{1-\overline{\alpha}}\right)m_2(n) - \kappa(n-6)
\geq
\begin{cases}
2\times 10^{-2} & n=3,\\
3\times 10^{-3} & n=4,\\
5\times 10^{-4} & n=5,\\
2\times 10^{-4} & n=7,
\end{cases}
\end{align*}
which proves \eqref{eq:non-neg} for the case  $n \in \{3,4,5,7\}$.

\emph{Step 4.} Finally, we prove 
 \eqref{eq:non-neg} for the case $n \in \mathbb{N} \cap [8,\infty)$.
To do this
we prove that $\partial_n \varphi(\overline{\alpha},n)>0$ for $n\geq 7$.
Then the result follows from the case $n=7$ proved in Step 3.
By definition of $\widetilde{m}(\alpha,n)$, for all $\alpha\in(0,1)$, we have
$\partial_m h_\alpha(\widetilde{m}(\alpha,n),n)=0$. Therefore
\[
\partial_n \varphi(\alpha,n) = 
\frac{\partial h_\alpha}{\partial n} (\tilde{m}(\alpha,n),n)
=
\widetilde{m}^2(\alpha,n) \, \partial_n c_n - \kappa.
\]
Since $n\mapsto c_n$ is convex (Lemma \ref{lem:cn}), then $\partial_n c_n$ is increasing and we get the lower bound
\begin{equation}\label{eq:estpartialvarphi1}
\partial_n \varphi(\alpha,n) \geq  \widetilde{m}^2(\alpha,n) \, \partial_n c_n|_{n=7} - \kappa
\end{equation}
for all $n \ge 7$.
We will prove below that
\begin{equation}\label{eq:ubm}
\widetilde{m}(\overline{\alpha},n)\leq\frac{3}{2}.
\end{equation}
Observe that
\begin{equation}
\label{eq:Dcn}
\partial_n c_n = - \frac 1n c_n
+ \frac{1}{2n} \left( - \frac{\pi}{3n^2} \sec^2 \left( \frac{\pi}{n} \right)+ \frac{\pi}{n^2}  \csc^2 \left(\frac{\pi}{n} \right) \right).
\end{equation}
From \eqref{eq:estpartialvarphi1}, \eqref{eq:ubm}, \eqref{eq:Dcn} 
we obtain that 
\begin{align*}
\partial_n \varphi(\overline{\alpha},n) \geq  \widetilde{m}^2(\overline{\alpha},n) \, \partial_n c_n |_{n=7} -\kappa
\geq \left( \frac{3}{2} \right)^2 \, (\partial_n c_n)|_{n=7} - \kappa
 > 1.5 \times 10^{-5} >0,
\end{align*}
as required.

To prove \eqref{eq:ubm} we reason as follows. Using \eqref{eq:h_m} and the fact that $n\mapsto c_n$ is decreasing (Lemma \ref{lem:cn}), we deduce that 
$n \mapsto \partial_m h_\alpha(m,n)$ is decreasing for all $m$, and hence 
$n\mapsto \widetilde{m}(\alpha,n)$ is increasing.
Therefore $\widetilde{m}(\alpha,n)\leq \widetilde{m}(\alpha,\infty)$, where $\widetilde{m}(\alpha,\infty)$ is defined to be the largest root of 
\[
\partial_m h_\alpha(m,\infty) := \frac{\alpha}{1-\alpha} c_6 m^{\alpha-1} + 2 c_\infty m
- c_6\frac{2-\alpha}{1-\alpha},
\]
where $c_\infty$ was defined in Lemma \ref{lem:cn}.
We want to show that $\widetilde{m}(\overline{\alpha},\infty)\leq 3/2$.
We have
\begin{align*}
\partial_m h_{\overline{\alpha}} \left( 1, \infty \right) 
& =
\frac{\overline{\alpha}}{1-\overline{\alpha}} c_6  + 2 c_\infty 
- c_6\frac{2-\overline{\alpha}}{1-\overline{\alpha}} 
< -0.0024 < 0,
\\
\partial_m h_{\overline{\alpha}} \left(3/2,\infty \right) 
& =
\frac{\overline{\alpha}}{1-\overline{\alpha}} c_6 \left(\frac{3}{2}\right)^{\overline{\alpha}-1} + 2 c_\infty \frac 32
- c_6\frac{2-\overline{\alpha}}{1-\overline{\alpha}} 
> 0.12 > 0.
\end{align*}
Therefore $\widetilde{m}(\overline{\alpha},\infty) \in (1,3/2)$ by the Intermediate Value Theorem. This proves \eqref{eq:ubm} and completes the proof.
\end{proof}

\begin{remark}
Despite the fact that we used the specific value of ${\overline{\alpha}}$ in several places in the proof of  Lemma \ref{lem:inequality}, we expect Lemma \ref{lem:inequality} to hold for all $\alpha\in(0,1)$.
\end{remark}


\subsection{Lower bound on the area of optimal cells}
\label{sec:lbmi}

Let $\mu=\sum_{i=1}^{N_\mu}m_i\delta_{z_i} \in \mathcal{A}_{\delta,\overline{\alpha}}$ be a minimizer of $\f_{\delta,\overline{\alpha}}$.
We will prove the convexity inequality $h_{\overline{\alpha}}(m_i,n)\ge 0$ for all $i \in \{1,\ldots,N_\mu\}$, $n \in \mathbb{N} \cap [3,\infty)$.
The idea is to prove a lower bound $m_i\geq \overline{m}$ such that 
$h_{\overline{\alpha}}(\overline{m},n)\ge 0$ for all  $n \in \mathbb{N} \cap [3,\infty)$. Then the convexity inequality follows from Lemma 
\ref{lem:inequality}. 

We prove the lower bound $m_i\geq \overline{m}$ following the strategy of the proof of \cite[Lemma 7]{BouPelTheil}, which was developed for the case $\alpha=0.5$.
The main differences are that we have to deal with the more difficult case of $\alpha=\overline{\alpha}>0.5$ and that we optimise some of the estimates.
Our proof can be used to give a lower bound $\overline{m}$ on the areas $m_i$ for all $\alpha\in(0,1)$, \blue{but this lower bound does not satisfy the convexity inequality $h_{\alpha}(\overline{m},n)\ge 0$ if $\alpha > \overline{\alpha}$}.
Our lower bound on the area of the cells holds for cells with arbitrarily many sides, although by Corollary \ref{cor:reduction} we only need the lower bound for cells with $3$, $4$ or $5$ sides. We saw no advantage in the proof of restricting the number of sides.

The following result gives the difference in energy of a 
partition and the one obtained by merging two of its cells. Recall that 
$\mathcal{S}_{\delta,\alpha}$ and
$F$ were defined in Definitions \ref{def:Sdeltaalpha} and \ref{def:optpart}.

\begin{lemma}[Merging]\label{lem:mergsplit}
Let $\mathcal{C}=(C_1,\dots,C_k)\in\mathcal{S}_{\delta,\alpha}$, $k\geq 2$.
For $i\in\{1,\dots,k\}$ let $m_i=|C_i|$ and let $z_i\in C_i$ be the centroid of $C_i$.
Define $\mathcal{D}\in\mathcal{S}_{\delta,\alpha}$ by $\mathcal{D}:=(C_1\cup C_2, C_3,\dots, C_k)$.
For all \blue{$\alpha\in(-\infty,1)$},
\begin{equation}\label{eq:splitting}
F(\mathcal{D}) - F(\mathcal{C}) = \frac{c_6}{1-\alpha}\left(  (m_1+m_2)^\alpha -m_1^\alpha-m_2^\alpha \right)
	+ |z_2-z_{1}|^2\frac{m_1m_2}{m_1+m_2}.
\end{equation}
\end{lemma}

\begin{proof}
By definition
\[
z_1 = \frac{1}{m_1}\int_{C_1} x\,\dd x,
\quad\quad
z_2 = \frac{1}{m_2}\int_{C_2} x\,\dd x.
\]
Let $\overline{z}\in C_1\cup C_2$ be the centroid of $C_1\cup C_2$:
\[
\bar{z} = \frac{m_1}{m_1+m_2}z_{1} + \frac{m_2}{m_1+m_2}z_2.
\]
A direct computation gives
\begin{align*}
\int_{C_1\cup C_2} & |x-\bar{z}|^2 \,\dd x - \int_{C_1} |x-z_{1}|^2 \,\dd x - \int_{C_2} |x-z_2|^2 \,\dd x  \\
&= \int_{C_1} \left( |\bar{z}|^2 - 2 x\cdot\bar{z} - |z_{1}|^2 + 2 x\cdot z_{1}  \right) \,\dd x
	+ \int_{C_2} \left( |\bar{z}|^2 - 2 x\cdot\bar{z} - |z_2|^2 + 2 x\cdot z_2  \right) \, \dd x \\
&= m_1|\bar{z}|^2 - 2m_1 z_{1}\cdot\bar{z} +m_1|z_{1}|^2
	+ m_2|\bar{z}|^2 - 2m_2 z_2\cdot\bar{z} + m_2|z_2|^2   \\
& = m_1 |\bar{z}-z_1|^2 + m_2 |\bar{z}-z_2|^2 \\
& = m_1 \left| \frac{m_2}{m_1+m_2}(z_2-z_1)\right|^2 +
m_2 \left| \frac{m_1}{m_1+m_2}(z_1-z_2)\right|^2
\\
&= \frac{m_1 m_2}{m_1+m_2}|z_2-z_{1}|^2.
\end{align*}
The result now follows immediately from the definition of $F$.
\end{proof}

\begin{figure}
\includegraphics[scale=2.5]{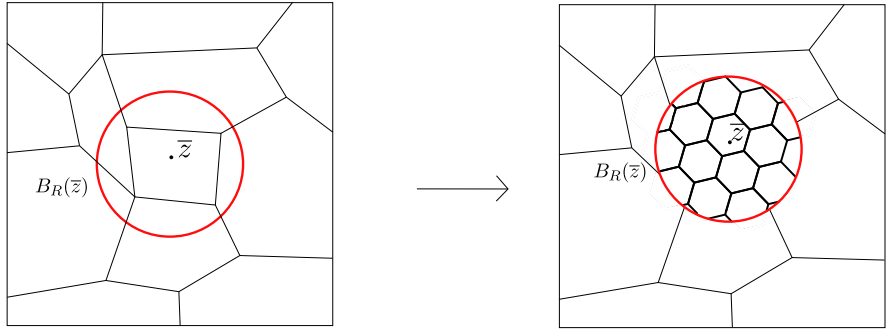}
\caption{The construction inside the ball $B_R(\bar{z})$ (see the proof of Lemma \ref{lem:lowerboundmass}).}
\label{fig:ball}
\end{figure}

We now prove a lower bound on the area of optimal cells, as well as an upper bound on the diameter of the cells and the maximum distance between the centroids.
The latter two estimates will be used later to deal with the fact that the lower bound on the area of cells close to the boundary of $Q_{\delta,\overline{\alpha}}$ is not good enough to ensure the validity of the convexity inequality \eqref{eq:ineq}.

\begin{lemma}[Lower bound on the area of optimal cells]\label{lem:lowerboundmass}
Let $\mu=\sum_{i=1}^{N_\mu}m_i\delta_{z_i} \in \mathcal{A}_{\delta,\overline{\alpha}}$ be a minimizer of $\f_{\delta,\overline{\alpha}}$.
If $\delta >0$ is sufficiently small, then the following hold:
\begin{itemize}
\item[(i)] If $\mathrm{dist}(z_i,\partial Q_{\delta,\overline{\alpha}}) \ge 4$, then 	
\[
m_i > \overline{m} :=  2.0620 \times 10^{-4}.
\]
\item[(ii)] If $\mathrm{dist}(z_i,\partial Q_{\delta,\overline{\alpha}}) < 4$, then 
\[
m_i > m_\mathrm{b} := 1.5212 \times 10^{-5}.
\]
\item[(iii)] Let $B \subset Q_{\delta,\overline{\alpha}}$ be a ball of radius $R$. If $B \cap \mathrm{supp}(\mu) = \emptyset$, then 
$R<R_0 := 3.3644$.
\item[(iv)] Let $T$ be the optimal transport map defining $W_2(\ca_{Q_{\delta,\overline{\alpha}}},\mu)$. For all $i \in \{1,\ldots,N_\mu\}$,
\[
\mathrm{diam}(T^{-1}(z_i)) \le D_0 := 2 \left( 8 R_0^2 + \frac{\overline{\alpha}}{1-\overline{\alpha}} c_6  m_{\mathrm{b}}^{\overline{\alpha}-1}  \right)^{\frac 12}.
\]
\end{itemize}
\end{lemma}

\begin{proof}
\emph{Step 1: Upper bound on the distance between Dirac masses.}
Let 
$\bar{z}\in\supp(\mu)$ satisfy $\mathrm{dist}(\bar{z},\partial Q_{\delta,\overline{\alpha}}) \ge 4$.
Take $R \in (0,4)$ such that
$B_{R}(\bar{z})\cap\supp(\mu)=\{\bar{z}\}$.
We want to get an upper bound on $R$.
Let $S:B_{R}(\bar{z}) \to \supp(\mu)$ be any Borel map. Then
\begin{align}\label{eq:lbW22}
\nonumber
\int_{B_{R}(\bar{z})} |x-S(x)|^2 \, \dd x &= \int_{B_{R/2}(\bar{z})} |x-S(x)|^2 \, \dd x
	+ \int_{B_{R}(\bar{z})\setminus B_{R/2}(\bar{z})} |x-S(x)|^2 \, \dd x  
	\\
\nonumber
& \geq \int_{B_{R/2}(\bar{z})} |x-\bar{z}|^2 \, \dd x
+ \int_{B_{R}(\bar{z})\setminus B_{R/2}(\bar{z})} \Big| x-R\frac{x}{|x|} \Big|^2 \, \dd x \\
&= \frac{\pi}{12}R^4.
\end{align}
It is now convenient to use the partition energy $F$. Let $\mathcal{C}^\mu$ be the partition associated to the minimizer $\mu$ 
(see Definition \ref{def:partmeas}).
Consider the partition $\widetilde{\mathcal{C}} \in \mathcal{S}_{\delta,\alpha}$ obtained by modifying  $\mathcal{C}^\mu$ 
in the ball $B_R(\bar{z})$ as follows. Write $\mathcal{C}^\mu=(C_1,\ldots,C_{N_\mu})$ and define $\widetilde{C}_i =  C_i \setminus B_R(\bar{z})$ for $i \in \{ 1 ,\ldots, N_\mu\}$. 
Let $\{H_j\}$ be a tiling of the plane by regular hexagons of area $A$, where $A>0$ will be defined below, and let $\widetilde{H_j}=H_j \cap B_R(\bar{z})$.  Define $\widetilde{\mathcal{C}}$ to be the partition consisting of the sets  $\widetilde{C}_i$ for all $i \in \{ 1 ,\ldots, N_\mu\}$ and $\widetilde{H}_j$ for all $j$ such that $\widetilde{H_j} \ne \emptyset$;
see Figure \ref{fig:ball}.

Let $d_A:=2^{\frac{3}{2}}3^{-\frac{3}{4}}A^{\frac{1}{2}}$
be the diameter of a regular hexagon of area $A$.
The number of hexagons $N_A$ needed to cover a ball of radius $R$ is bounded above by
\begin{equation}\label{eq:NA}
N_A\leq \frac{\pi(R+d_A)^2}{A}.
\end{equation}
If $H \in \widetilde{\mathcal{C}}$ is a (whole) regular hexagon of area $A$ with centroid $\xi_H$, then
\begin{equation}
\label{eq:Hex_energy}
\frac{c_6}{1-\overline{\alpha}} |H|^{\overline{\alpha}} + \int_H |x-\xi_H|^2 \, \dd x
= \frac{c_6}{1-\overline{\alpha}} A^{\overline{\alpha}} + c_6 A^2.
\end{equation}
Since $\mu$ is a minimizer of $\f_{\delta,\overline{\alpha}}$, then $C^\mu$ is a minimizer of $F$. Therefore 
\begin{align}
\nonumber
0 & \ge F(\mathcal{C}^\mu) - F(\widetilde{\mathcal{C}})
\\
\nonumber
& \ge \sum_{\substack{i=1\\C_i \cap B_R(\bar{z}) \ne \emptyset}}^{N_\mu}
\left( \frac{c_6}{1-\overline{\alpha}} |C_i|^{\overline{\alpha}} + \int_{C_i \setminus B_R(\bar{z})} |x-\xi_{C_i}|^2 \, \dd x + \int_{C_i \cap B_R(\bar{z})} |x-\xi_{C_i}|^2 \, \dd x
\right)
\\
\nonumber
& \quad -  \sum_{\substack{i=1\\C_i \cap B_R(\bar{z}) \ne \emptyset}}^{N_\mu}
\left(
\frac{c_6}{1-\overline{\alpha}} |C_i \setminus B_R(\bar{z})|^{\overline{\alpha}}
+  \int_{C_i \setminus B_R(\bar{z})} |x-\xi_{C_i \setminus B_R(\bar{z})}|^2 \, \dd x
\right)
\\
\nonumber
& \quad - N_A \left( \frac{c_6}{1-\overline{\alpha}} |H|^{\overline{\alpha}} + \int_H |x-\xi_H|^2 \, \dd x\right)
\\
\nonumber
& \stackrel{\eqref{eq:Hex_energy}}{\ge} 
\sum_{i=1}^{N_\mu}
\int_{C_i \cap B_R(\bar{z})} |x-\xi_{C_i}|^2 \, \dd x
 - 
N_A\left( \frac{c_6}{1-\overline{\alpha}}A^{\overline{\alpha}} + c_6 A^2 \right)
\\
\nonumber
& \quad
+ \sum_{\substack{i=1\\C_i \cap B_R(\bar{z}) \ne \emptyset}}^{N_\mu}
\left( \int_{C_i \setminus B_R(\bar{z})} |x-\xi_{C_i}|^2 \, \dd x  
-
\int_{C_i \setminus B_R(\bar{z})} |x-\xi_{C_i \setminus B_R(\bar{z})}|^2 \, \dd x
\right)
\\
\label{eq:44}
& \ge \int_{B_R(\bar{z})} |x-S(x)|^2 \dd x - 
N_A\left( \frac{c_6}{1-\overline{\alpha}}A^{\overline{\alpha}} + c_6 A^2 \right),
\end{align}
where $S:B_R(\bar{z}) \to \mathrm{supp}(\mu)$ is defined by $S(x)=\xi_{C_i}$ if $x \in C_i$, and where in the final inequality we used the property of the centroid that 
\[
\int_{C_i \setminus B_R(\bar{z})} |x-\xi_{C_i \setminus B_R(\bar{z})}|^2 \, \dd x
 = \inf_{y \in \mathbb{R}^2}
 \int_{C_i \setminus B_R(\bar{z})} |x-y|^2 \, \dd x.
\]
Combining estimates \eqref{eq:44} and \eqref{eq:lbW22} gives
\begin{equation}
\label{eq:45}
\frac{\pi}{12}R^4 \le  N_A\left( \frac{c_6}{1-\overline{\alpha}}A^{\overline{\alpha}} + c_6 A^2 \right) 
 \stackrel{\eqref{eq:NA}}{\leq} \frac{\pi(R+d_A)^2}{A}\left( \frac{c_6}{1-\overline{\alpha}}A^{\overline{\alpha}} + c_6 A^2 \right). 
\end{equation}
Define
\[
p(R,A):=R^2 - (R+d_A)\left( \frac{12 c_6}{1-\overline{\alpha}}A^{\overline{\alpha}-1} + 12 c_6 A \right)^{\frac{1}{2}}.
\]
Then \eqref{eq:45} implies that $p(R,A)\le 0$. The quadratic polynomial $R \mapsto p(R,A)$ has one positive root and one negative root. Let $\widetilde{R}(A)$ denote the positive root:
\[
\widetilde{R}(A):= \frac 12 \left[12c_6\left( \frac{A^{\overline{\alpha}-1}}{1-\overline{\alpha}} + A \right)\right]^{\frac{1}{2}}
	+ \frac 12 \sqrt{12c_6\left( \frac{A^{\overline{\alpha}-1}}{1-\overline{\alpha}} + A \right)+
		4\left[12c_6\left( \frac{A^{\overline{\alpha}-1}}{1-\overline{\alpha}} + A \right)\right]^{\frac{1}{2}}d_A}.
\]
Since  $p(R,A)\le 0$, we have $R \in [0,\widetilde{R}(A)]$ for all $A$, and so 
\begin{equation}\label{eq:R}
R\leq \min_{A>0} \widetilde{R}(A)
\leq \widetilde{R}(0.52) < 3.3644,
\end{equation}
where the final inequality was obtained by numerically evaluating $\widetilde{R}(0.52)$. The choice $A=0.52$ was motivated by numerically minimising $\widetilde{R}(A)$.

\emph{Step 2: Proof of (i).}
Let $\bar{z}\in\mathrm{supp}(\mu)$ satisfy $\mathrm{dist}(\bar{z},\partial Q_{\delta,\overline{\alpha}}) \ge 4$. Let  $R_0=3.3644$. By Step 1, there exists at least one point $z\in B_{R_0}(\bar{z}) \cap \mathrm{supp}(\mu)$. In particular, 
	\begin{equation}
	\label{eq:UpperBoundOnR}
	|\bar{z}-z| \le 3.3644.
	\end{equation}
	Let $m=\mu(\{\bar{z}\})$, $M=\mu(\{z\})$. We can assume without loss of generality that $m \le M$ (otherwise simply interchange the roles of $\bar{z}$ and $z$). 
	
Let $\mathcal{C}^\mu$ be the partition associated to the minimizer $\mu$. We can define a new partition $\mathcal{D}$ by replacing the cells $\overline{T^{-1}(\bar{z})}$ and
$\overline{T^{-1}(z)}$ with their union. Then  
Lemma \ref{lem:mergsplit} yields 
\[
0 \leq F(\mathcal{D}) - F(\mathcal{C}^\mu) \leq
\frac{c_6}{1-\overline{\alpha}}\left( (m+M)^{\overline{\alpha}}-m^{\overline{\alpha}}-M^{\overline{\alpha}} \right)
	+ |\bar{z}-z|^2 \frac{mM}{m+M}.
\]
Define $\lambda:=\frac{m}{M} \in (0,1]$.
By dividing the previous inequality by \blue{$M^{\overline{\alpha}}$} and rearranging we obtain
\begin{equation}\label{eq:mb}
m \geq \left[ \frac{1}{|\bar{z}-z|^2}\frac{c_6}{1-\overline{\alpha}}
	\inf_{\lambda\in(0,1]}\left( (1+\lambda)
	\frac{1+\lambda^{\overline{\alpha}}-(1+\lambda)^{\overline{\alpha}}}{\lambda^{\overline{\alpha}}} \right) \right]^{\frac{1}{1-\overline{\alpha}}}.
\end{equation}
For $\alpha\in(0,1)$, let
\[
\Theta_\alpha(\lambda):=(1+\lambda)\frac{1+\lambda^{\alpha}-(1+\lambda)^{\alpha}}{\lambda^{\alpha}}.
\]
\blue{We can restrict out attention to $\alpha >0$ since eventually we will apply this result to $\alpha=\overline{\alpha}>0$.}
We want to bound  $\Theta_{\alpha}$ from below for the case $\alpha=\overline{\alpha}$.
The idea is the following: We first prove in Step 2a that the function $\lambda\mapsto\Theta_{\alpha}(\lambda)$ has one minimum point for all $\alpha\in(0,1)$. In Step 2b we estimate this minimum point for the case $\alpha=\overline{\alpha}$.

\emph{Step 2a.} In this substep we prove that, for all $\alpha\in(0,1)$, the function $\lambda\mapsto\Theta'_\alpha(\lambda)$ vanishes at only one point $\lambda\in[0,1)$. 
Fix $\alpha\in(0,1)$. We have
\[
\Theta'_\alpha(\lambda) = \frac{1}{\lambda^{1+\alpha}}\left[ (1-{\alpha})\lambda + \lambda^{1+\alpha}
	-(\alpha+1)(1+\lambda)^{\alpha}\lambda + \alpha(1+\lambda)^{1+\alpha} -\alpha \right]
	=:\frac{1}{\lambda^{1+\alpha}} \Lambda(\lambda).
\]
The strategy we use to prove that there exists a unique $\lambda\in(0,1)$ such that
$\Lambda(\lambda)=0$ is the following: Using the fact that
\[
\Lambda(0)<0,\quad\quad
\Lambda(1)=(1-\alpha)(2-2^\alpha)>0, \quad \quad
\Lambda'(0) = -\alpha + \alpha^2 < 0,
\]
the desired result is proved once we show that $\Lambda'$ vanishes at only one point $\lambda\in(0,1)$.
A direct computation gives
\begin{align*}
\Lambda'(\lambda) & = \frac{1}{(1+\lambda)^{1-\alpha}}\left[ (1-\alpha)(1+\lambda)^{1-\alpha} +
	(1+\alpha)\lambda^\alpha (1+\lambda)^{1-\alpha} - 1 +\alpha^2 - (1+\alpha)\lambda \right]
	\\
	& =: \frac{1}{(1+\lambda)^{1-\alpha}} \Phi(\lambda)
\end{align*}
where
\[
\Phi(\lambda) = (1-\alpha)(1+\lambda)^{1-\alpha} +
(1+\alpha)\lambda^\alpha (1+\lambda)^{1-\alpha}
-1+\alpha^2 - (1+\alpha)\lambda.
\]
We claim that
\begin{equation}
\label{eq:vp12_1}
\Phi(0)<0, \quad \quad \Phi(1)>0,
\end{equation}
and that
\begin{equation}\label{eq:vp12_2}
\Phi'(\lambda)>0
\end{equation}
for all $\lambda\in[0,1)$. 
This will show that $\Phi$ vanishes at only one point $\lambda\in(0,1)$ and, in turn, that the same holds for $\Lambda'$.

We start by proving \eqref{eq:vp12_1}. Note that
$\Phi(0)=-\alpha + \alpha^2 < 0$ since $\alpha\in(0,1)$. Let
\[
\psi(\alpha):=\Phi(1)=2^{2-\alpha} - 2 - \alpha + \alpha^2.
\]
We want to prove that $\psi(\alpha)>0$. 
Since $\psi(1)=0$, it is sufficient to prove that $\psi'(\alpha)<0$.
Note that $\psi'(\alpha) = -2^{2-\alpha}\ln 2 - 1 +2\alpha$, $\psi''(\alpha)=2^{2-\alpha}(\ln 2)^2 + 2 >0$, and $\psi'(1)=-2\ln 2 + 1 <0$. Therefore $\psi(\alpha)>0$ for all $\alpha\in(0,1)$, which completes the proof of \eqref{eq:vp12_1}.

Finally, we prove \eqref{eq:vp12_2}. We have
\begin{align*}
\Phi'(\lambda)& = (1-\alpha)^2(1+\lambda)^{-\alpha} + \alpha(1+\alpha)\lambda^{\alpha-1}(1+\lambda)^{1-\alpha}
	+(1-\alpha^2)\lambda^\alpha (1+\lambda)^{-\alpha} - (1+\alpha),
	\\
\Phi''(\lambda) & = 
-\alpha(1-\alpha)^2 (1+\lambda)^{-1-\alpha}
- \alpha (1-\alpha^2) \lambda^{\alpha-2} (1+\lambda)^{-\alpha}(1-\lambda)
 - \alpha (1-\alpha^2) \lambda^\alpha (1+\lambda)^{-1-\alpha}.
\end{align*}
Since $\Phi''(\lambda)<0$ for all $\lambda\in[0,1)$, we have that
\[
\Phi'(\lambda)\geq\Phi'(1) = 2^{1-\alpha}(1+\alpha^2) - (1+\alpha)
=: \varphi(\alpha).
\]
We have
\[
\varphi'(\alpha) = -\ln 2 (1+\alpha^2)2^{1-\alpha} + 2\alpha 2^{1-\alpha} -1, \quad
\varphi''(\alpha)  = 2^{1-\alpha}\left[ (\ln 2)^2 \alpha^2 -4\alpha \ln 2 + 2 +(\ln 2)^2  \right].
\]
Therefore $\varphi''(\alpha)=0$ if and only if $\alpha \in \{\alpha_-,\alpha_+\}$, where 
\[
\alpha_{\pm} = \frac{2\pm \sqrt{2-(\ln2)^2}}{\ln2}.
\]
Since $1<\alpha_-<\alpha_+$, $\varphi$ is strictly convex in $(0,\alpha_-)$. Therefore, for all $\alpha\in(0,1)$,
\[
\varphi(\alpha) > \varphi(1) + \varphi'(1)(\alpha - 1) 
= (1-2\ln 2)(\alpha-1) > 0.
\]
Therefore $\Phi'(\lambda) \ge \varphi(\alpha)>0$, which completes the proof of \eqref{eq:vp12_2}.

\emph{Step 2b.} Let $\bar{\lambda} \in (0,1)$ denote the unique root of $\Theta'_{\overline{\alpha}}$. 
We now estimate $\Theta_{\overline{\alpha}}(\bar{\lambda})=\inf_{\lambda \in (0,1]}\Theta_{\overline{\alpha}}(\lambda)$.
Let $\lambda_1 = 0.160764$ and $\lambda_2 = 0.160767$.
Numerically we see that
\[
\Theta'_{\overline{\alpha}}(\lambda_1) < -2\times 10^{-6}<0,
\quad\quad
\Theta'_{\overline{\alpha}}(\lambda_2) > 2\times 10^{-6}>0.
\]
Therefore $\bar{\lambda}\in(\lambda_1,\lambda_2)$ by the Intermediate Value Theorem.
Recall that
\[
\Theta'_{\overline{\alpha}}(\lambda)
= \lambda^{-1-\alpha} [(1-\overline{\alpha})\lambda + \lambda^{1+\overline{\alpha}}
+ (1+\lambda)^{\overline{\alpha}} (\overline{\alpha}-\lambda)
- \overline{\alpha}].
\]
Therefore, 
for all $\lambda\in(\lambda_1,\lambda_2)$,
\begin{align*}
\Theta'_{\overline{\alpha}}(\lambda)
& \le \lambda_1^{-1-\overline{\alpha}}
[ (1-\overline{\alpha})\lambda_2 +\lambda_2^{1+\overline{\alpha}} + (1+\lambda_2)^{\overline{\alpha}} (\overline{\alpha} - \lambda_1) - \overline{\alpha}] < 6.2 \times 10^{-5}
\\
\Theta'_{\overline{\alpha}}(\lambda)
& \ge \lambda_1^{-1-\overline{\alpha}}
[ (1-\overline{\alpha})\lambda_1 +\lambda_1^{1+\overline{\alpha}} + (1+\lambda_1)^{\overline{\alpha}} (\overline{\alpha} - \lambda_2) - \overline{\alpha}] > - 6.3 \times 10^{-5}.
\end{align*}
Therefore $|\Theta'_{\overline{\alpha}}(\lambda)|\le 6.3 \times 10^{-5}$ for all 
 $\lambda\in(\lambda_1,\lambda_2)$. It follows that
\begin{equation}\label{eq:lowerboundf}
\Theta_{\overline{\alpha}}(\bar{\lambda})\geq \Theta_{\overline{\alpha}}(\lambda_1)
	- (\lambda_2-\lambda_1)\| \Theta'_{\overline{\alpha}} \|_{L^\infty((\lambda_1,\lambda_2))}
	\geq 0.85482.
\end{equation}
From \eqref{eq:UpperBoundOnR}, \eqref{eq:mb} and \eqref{eq:lowerboundf} we conclude that
\begin{equation}
\label{eq:m>2e-4}
m > \left( \frac{1}{(3.3644)^2} \cdot \frac{c_6}{1-\overline{\alpha}}\cdot 0.85482 \right)
	^{\frac{1}{1-\overline{\alpha}}} >  2.0620 \times 10^{-4} = \overline{m},
\end{equation}
which proves (i).

\emph{Step 3: Proof of (iii).} This is very similar to Step 1. Let $B=B_R(x_0) \subset Q_{\delta,\overline{\alpha}}$
satisfy $B \cap \mathrm{supp}(\mu) = \emptyset$. Let $S:B \to \supp(\mu)$ be any Borel map. We have
\[
\int_{B} |x-S(x)|^2 \, \dd x 
\ge \int_{B} \mathrm{dist}(x,\partial B)^2 \, \dd x 
\ge \int_{B} \mathrm{dist}(x,\{x_0\} \cup \partial B)^2 \, \dd x \stackrel{\eqref{eq:lbW22}}{=} \frac{\pi}{12}R^4.
\]
The second inequality is clearly suboptimal, but it is sufficient for our purposes. 
Repeating exactly the same argument used in Step 1 (with $B_R(\bar{z})$ replaced by $B$) gives $R<3.3644$, as required. 

\emph{Step 4: Proof of (ii).} Let $z_i \in \mathrm{supp}(\mu)$ satisfy $\mathrm{dist}(z_i,\partial Q_{\delta,\overline{\alpha}}) < 4$. 
Let $U$ be the square $U = \{ x \in Q_{\delta,\overline{\alpha}} : \mathrm{dist}(x,\partial Q_{\delta,\overline{\alpha}}) \ge 4 \}$.
Take $x \in \partial U$ satisfying $|x-z_i| = \mathrm{dist}(z_i,\partial U)$. Let $K$ be a closed square of side-length $2R_0$ such that $x \in \partial K$ and $K \subset U$ (such a square exists if $\delta$ is sufficiently small). By Step 3, there exists $z_j \in \mathrm{supp}(\mu) \cap K$, $z_j \ne z_i$. Therefore
\begin{equation}
\label{eq:Boris}
|z_i - z_j| \le |z_i - x| + |x-z_j| \le \sqrt{32}+\mathrm{diam}(K) = \sqrt{32} + \sqrt{8} R_0.
\end{equation}
Without loss of generality can assume that $m_i < \overline{m}$ (otherwise $m_i \ge \overline{m} > m_{\mathrm{b}}$ and there is nothing to prove). Let $T$ be the optimal transport map defining 
$W_2(\ca_{Q_{\delta,\alpha}},\mu)$. Let $z \in T^{-1}(z_i)$. By
 Lemma \ref{lem:cellspolygons} and Lemma \ref{lem:propoptcell}(i),
\begin{equation}
\label{eq:Trump}
|z-z_i| + \frac{\overline{\alpha}}{1-\overline{\alpha}} c_6 m_i^{\overline{\alpha}-1}
 \le |z-z_j| + \frac{\overline{\alpha}}{1-\overline{\alpha}} c_6 m_j^{\overline{\alpha}-1}.
\end{equation}
By Lemma \ref{lem:propoptcell}(ii), $z_i \in T^{-1}(z_i)$. Taking $z=z_i$ in \eqref{eq:Trump} gives
\[
\frac{\overline{\alpha}}{1-\overline{\alpha}} c_6 m_i^{\overline{\alpha}-1}
\le |z_i-z_j| + \frac{\overline{\alpha}}{1-\overline{\alpha}} c_6 m_j^{\overline{\alpha}-1}
\le \sqrt{32} + \sqrt{8} R_0 + \frac{\overline{\alpha}}{1-\overline{\alpha}} c_6 \overline{m}^{\overline{\alpha}-1}
\]
by \eqref{eq:Boris} and Lemma \ref{lem:lowerboundmass}(i). Therefore
\[
m_i \ge \left[\frac{1-\overline{\alpha}}{c_6 \overline{\alpha}} 
\left( \sqrt{32} + \sqrt{8} R_0 + \frac{\overline{\alpha}}{1-\overline{\alpha}} c_6 \overline{m}^{\overline{\alpha}-1}\right) \right]^{\frac{1}{\overline{\alpha}-1}} > 1.5212 \times 10^{-5}
\]
as required.

\emph{Step 5: Proof of (iv).} Take any $x \in T^{-1}(z_i)$. Let $K$ be a closed square of side-length $2R_0$ such that $x \in K$ and $K \subset Q_{\delta,\overline{\alpha}}$ (such a square exists if $\delta$ is sufficiently small). By Step 3, there exists at least one point $z_j \in \mathrm{supp}(\mu) \cap K$. Therefore $|x-z_j|^2 \le \mathrm{diam}(K)^2 = 8R_0^2$. By  Lemma \ref{lem:cellspolygons} and Lemma \ref{lem:propoptcell}(i),
\[
|x-z_i|^2 \le |x-z_j|^2 + \frac{\overline{\alpha}}{1-\overline{\alpha}} c_6 m_j^{\overline{\alpha}-1} - \frac{\overline{\alpha}}{1-\overline{\alpha}} c_6 m_i^{\overline{\alpha}-1}
\le 8R_0^2 + \frac{\overline{\alpha}}{1-\overline{\alpha}} c_6 m_{\mathrm{b}}^{\overline{\alpha}-1}.
\] 
By Lemma \ref{lem:propoptcell}(ii), $z_i \in T^{-1}(z_i)$. Therefore 
\[
\mathrm{diam}(T^{-1}(z_i)) \le 2 \max_{x \in T^{-1}(z_i)} |x-z_i| 
\le 2 \left( 8R_0^2 + \frac{\overline{\alpha}}{1-\overline{\alpha}} c_6 m_{\mathrm{b}}^{\overline{\alpha}-1} \right)^{\frac 12}
\]
as required.
\end{proof}

The lower bound we obtained in Lemma \ref{lem:lowerboundmass}(i) is good enough to ensure the validity of the convexity inequality \eqref{eq:ineq}:

\begin{corollary}[Convexity inequality]
	\label{Cor:CI}
Let $\mu=\sum_{i=1}^{N_\mu}m_i\delta_{z_i} \in \mathcal{A}_{\delta,\overline{\alpha}}$ be a minimizer of $\f_{\delta,\overline{\alpha}}$.
If $\delta >0$ is sufficiently small and $\mathrm{dist}(z_i,\partial Q_{\delta,\overline{\alpha}}) \ge 4$, then 	
\begin{equation} \label{eq:milb}
h_{\overline{\alpha}}(m_i,n) \ge 0 \quad \forall \; n \in \mathbb{N} \cap [3,\infty).
\end{equation}
\end{corollary}

\begin{proof}
	Recall that $\overline{m} = 2.0620 \times 10^{-4}$.
By a direct computation we see that
\begin{align*}
h_{\overline{\alpha}}(\overline{m}, 3)  > 7\times 10^{-7} >0,
\qquad
h_{\overline{\alpha}}(\overline{m}, 4)  > 8 \times 10^{-4} > 0,
\qquad
h_{\overline{\alpha}}(\overline{m}, 5)  > 1 \times 10^{-3} > 0.
\end{align*}
Therefore \eqref{eq:milb} follows from Lemma \ref{lem:lowerboundmass}(i), Lemma \ref{lem:inequality}
and Corollary \ref{cor:reduction}.
\end{proof}


\subsection{Proof of Theorem \ref{thm:pen}}\label{sec:proof}
We are now in position to prove one of our main results. 
\begin{theorem}[Lower bound on $C_{\mathrm{p}}(\overline{\alpha})$]
	\label{thm:LBCp}
	We have
	\[
	C_{\mathrm{p}}(\overline{\alpha}) \ge \frac{2 - \overline{\alpha}}{1-\overline{\alpha}} \, c_6. 
	\]
\end{theorem}
Combining Theorem \ref{thm:LBCp} with Lemma \ref{lem:UBCp}, Lemma \ref{lem:monoCp} and Corollary \ref{cor:monoCp} 
completes the proof of Theorem \ref{thm:pen}.

\begin{proof}[Proof of Theorem \ref{thm:LBCp}]
By \eqref{eq:suff} it is sufficient to prove that
	\begin{equation}
	\label{eq:suffThm415}
\lim_{\delta\to 0} V^{-1}_{\delta,\overline{\alpha}} \min \left\{
	\f_{\delta,\overline{\alpha}}(\mu) : \mu\in\mathcal{A}_{\delta,\overline{\alpha}} \right\}
\ge
	\frac{2 - \overline{\alpha}}{1-\overline{\alpha}} \, c_6.
	\end{equation}
	Let $\mu=\sum_{i=1}^{N_\mu}m_i\delta_{z_i} \in \mathcal{A}_{\delta,\overline{\alpha}}$ be a minimizer of $\f_{\delta,\overline{\alpha}}$ and let
\begin{equation}
\label{eq:IJ}
\begin{aligned}
\mathcal{I}  & = \{ i \in \{1,\ldots,N_\mu \} : \mathrm{dist}(z_i,\partial Q_{\delta,\overline{\alpha}}) \ge 4 \},
\\
\mathcal{J}  & = \{ j \in \{1,\ldots,N_\mu \} : \mathrm{dist}(z_j,\partial Q_{\delta,\overline{\alpha}}) < 4 \}.
\end{aligned}
\end{equation}
Let $U$ be the tubular neighbourhood of $\partial Q_{\delta,\overline{\alpha}}$ of width $4+D_0$, where $D_0$ was defined in  
Lemma \ref{lem:lowerboundmass}(v):
\[
U = \left\{ x + y : x \in \partial Q_{\delta,\overline{\alpha}}, \; y \in B_{4+D_0}(x) \right\}  \cap Q_{\delta,\overline{\alpha}}.
\]
Let $T$ be the optimal transport map defining $W_2(\ca_{Q_{\delta,\overline{\alpha}}},\mu)$.
By Lemma \ref{lem:lowerboundmass}(v),
\begin{equation}
\label{eq:Biden}
 \bigcup_{j \in \mathcal{J}} T^{-1}(z_j) \subset U.
\end{equation}
Recall that $Q_{\delta,\overline{\alpha}}$ is a square of side-length $V_{\delta,\overline{\alpha}}^{1/2}$. 
By Lemma \ref{lem:lowerboundmass}(iv), $|T^{-1}(z_j)| > m_\mathrm{b}$ for all $j \in \mathcal{J}$. Therefore, \blue{for $\delta$ sufficiently small so that $V_{\delta,\overline{\alpha}}^{1/2} > 2(4+D_0)$},
\begin{equation}
\label{eq:Borat}
\# \mathcal{J} < \frac{|U|}{m_{\mathrm{b}}} =  \frac{V_{\delta,\overline{\alpha}}-(V_{\delta,\overline{\alpha}}^{1/2}-2(4+D_0))^2}{m_{\mathrm{b}}}.
\end{equation}
Hence
\begin{equation}
\label{eq:BCardJ}
\lim_{\delta \to 0} \frac{\# \mathcal{J}}{V_{\delta,\overline{\alpha}}} = 0.
\end{equation}
Since $\sum_{i=1}^{N_{\mu}} m_i=V_{\delta,\blue{\overline{\alpha}}}$, we have
\[
1  \ge V_{\delta,\blue{\overline{\alpha}}}^{-1} \sum_{i \in \mathcal{I}} m_i 
 = V_{\delta,\blue{\overline{\alpha}}}^{-1} \sum_{i=1}^{N_\mu} m_i  - V_{\delta,\blue{\overline{\alpha}}}^{-1} \sum_{j \in \mathcal{J}} m_j 
\stackrel{\eqref{eq:Biden}}{\ge} 1 -  V_{\delta,\blue{\overline{\alpha}}}^{-1} |U|.
\]
Using this and \eqref{eq:Borat} we conclude that
\begin{equation}
\label{eq:Bruno}
\lim_{\delta \to 0} V_{\delta,\blue{\overline{\alpha}}}^{-1} \sum_{i \in \mathcal{I}} m_i =1.
\end{equation}
For $i \in \{ 1 \ldots, N_\mu \}$, let $n_i$ be the number of edges of the convex polygon $T^{-1}(z_i)$.
Recall from Lemma \ref{lemma:PCP} that
\begin{equation}
\label{eq:BNE}
\sum_{i=1}^{N_\mu} (n_i - 6) \le 0.
\end{equation}
Finally, we put everything together to complete the proof: 
\begin{align*}
V_{\delta,\overline{\alpha}}^{-1} \f_{\delta,\overline{\alpha}}(\mu) 
& \ge  V_{\delta,\overline{\alpha}}^{-1} \sum_{i=1}^{N_\mu} g_{\overline{\alpha}}(m_i, n_i)
& (\textrm{equation } \eqref{eq:Putin})
\\ 
& \ge  V_{\delta,\overline{\alpha}}^{-1} \sum_{i \in \mathcal{I}} g_{\overline{\alpha}}(m_i, n_i)
\\
& \ge V_{\delta,\overline{\alpha}}^{-1}\sum_{i \in \mathcal{I}} \left( g_{\overline{\alpha}}(1,6) + \nabla g_{\overline{\alpha}}(1,6)\cdot(m_i-1,n_i-6) \right)
& (\textrm{Corollary }\ref{Cor:CI})
\\
& = V_{\delta,\overline{\alpha}}^{-1} \sum_{i \in \mathcal{I}}   \left( \frac{2-{\overline{\alpha}}}{1-{\overline{\alpha}}} \, c_6 \, m_i+\kappa(n_i-6) \right)
& (\textrm{equation } \eqref{eq:ineq})
\\
& =
  \frac{2-{\overline{\alpha}}}{1-{\overline{\alpha}}} \, c_6 \, V_{\delta,\overline{\alpha}}^{-1} \sum_{i \in \mathcal{I}}  m_i 
 		+  V_{\delta,\overline{\alpha}}^{-1} \kappa \underbrace{\sum_{i=1}^{N_\mu} (n_i-6)}_{\le 0 \textrm{ by } \eqref{eq:BNE}}
 		- V_{\delta,\overline{\alpha}}^{-1} \sum_{j \in \mathcal{J}} \kappa (n_j-6)
\\
& \ge  \frac{2-{\overline{\alpha}}}{1-{\overline{\alpha}}} \, c_6 \, V_{\delta,\overline{\alpha}}^{-1} \sum_{i \in \mathcal{I}}  m_i 
+ 3 \kappa \, V_{\delta,\overline{\alpha}}^{-1}  \# \mathcal{J}.
& (\kappa<0, \, n_j \ge 3)	
\end{align*}
Taking the limit $\delta \to 0$ and using \eqref{eq:BCardJ} and \eqref{eq:Bruno} proves \eqref{eq:suffThm415}, as required.
\end{proof}


\section{The constrained optimal location problem: Proof of Theorem \ref{thm:constr}}\label{sec:proofconstrained}

This section is devoted to the proof of Theorem \ref{thm:constr}.
The idea is to use Theorem \ref{thm:pen} together with the following relation between $C_{\mathrm{c}}(\alpha)$ and $C_{\mathrm{p}}(\alpha)$:

\begin{lemma}[Relation between the constrained and penalized problems] 
	\label{ref:RelCP}
	For all 
	 \blue{$\alpha\in(-\infty,1)$},
	\begin{equation}\label{eq:relation}
	C_{\mathrm{c}}(\alpha)\geq C_{\mathrm{p}}(\alpha)- \frac{c_6}{1-\alpha}.
	\end{equation}
\end{lemma}

\begin{proof}
By Remark \ref{remark:3_5},
\begin{align*}
C_{\mathrm{c}}(\alpha) &= \lim_{L\to\infty} L^{\frac{1}{1-\alpha}}\inf\left\{ W^2_2(\ca_Q,\mu)
	: \mu\in \mathcal{P}_{\mathrm{d}}(Q), \; \sum_{i=1}^{N_\mu} m_i^\alpha \leq L \right\} 
	\\
&=\lim_{L\to\infty} L^{\frac{1}{1-\alpha}}\inf \left\{ W^2_2(\ca_Q,\mu)
	+  \frac{c_6}{1-\alpha}L^{\frac{\alpha-2}{1-\alpha}}\sum_{i=1}^{N_\mu} m_i^\alpha
	-  \frac{c_6}{1-\alpha}L^{\frac{\alpha-2}{1-\alpha}}\sum_{i=1}^{N_\mu} m_i^\alpha \right. \\
&\hspace{9cm} : \left. \mu\in\mathcal{P}_{\mathrm{d}}(Q), \; \sum_{i=1}^{N_\mu} m_i^\alpha \leq L \right\}
\\
&\geq \lim_{L\to\infty} L^{\frac{1}{1-\alpha}}\inf\left\{W^2_2(\ca_Q,\mu)
	+  \frac{c_6}{1-\alpha}L^{\frac{\alpha-2}{1-\alpha}}\sum_{i=1}^{N_\mu} m_i^\alpha
	: \mu\in\mathcal{P}_{\mathrm{d}}(Q) \right\} - \frac{c_6}{1-\alpha} 
	L^{\frac{1}{1-\alpha}} L^{\frac{\alpha-2}{1-\alpha}} L
	\\
&=\lim_{\delta \to 0}  \left(\frac{c_6}{\delta(1-\alpha)}\right)^{\frac{1}{2-\alpha}}
	\inf\left\{  W^2_2(\ca_Q,\mu)+ \delta\sum_{i=1}^{N_\mu} m_i^\alpha
	: \mu\in \mathcal{P}_{\mathrm{d}}(Q) \right\} - \frac{c_6}{1-\alpha}\\
&= C_{\mathrm{p}}(\alpha)- \frac{c_6}{1-\alpha}
\end{align*}
by Remark \ref{rem:Cp}. In the penultimate equality we used the change of variables $\delta=\frac{c_6}{1-\alpha}L^{\frac{\alpha-2}{1-\alpha}}$.
\end{proof}

\begin{proof}[Proof of Theorem \ref{thm:constr}]
Let \blue{$\alpha\in(-\infty,\overline{\alpha}]$}. 
By Corollary \ref{cor:monotonicityCc}, in order to prove Theorem \ref{thm:constr} it is sufficient to prove that 
$ C_{\mathrm{c}}(\alpha) = c_6$. We have
\begin{align*}
c_6 & = C_{\mathrm{c}}(0)
& (\textrm{Remark } \ref{rem:Gdelta1})
\\
& \geq C_{\mathrm{c}}(\alpha) 
& (C_{\mathrm{c}}(\alpha) \textrm{ is non-increasing})
\\
& \geq C_{\mathrm{p}}(\alpha)- \frac{c_6}{1-\alpha}
& (\textrm{Lemma } \ref{ref:RelCP})
\\
& = c_6
& (\textrm{Theorem }\ref{thm:pen})
\end{align*}
as required.	
\end{proof}


\section{Paving the way towards $\alpha=1$}\label{sec:alpha1}

In this section we prove an asymptotic crystallization result for the full range of \blue{$\alpha\in(-\infty,1)$} under the ansatz (which we are not able to prove) that the neighbouring cells of the smallest ones are not too large.

	\begin{theorem}[Asymptotic crystallization for all $\blue{\alpha \in (-\infty,1)}$]
		\label{thm:alpha01}
		Let \blue{$\alpha\in(-\infty,1)$} and $\delta >0$. 
		Let $\mu=\sum_{i=1}^{N_{\mu}} m_i\delta_{z_i} \in \mathcal{A}_{\delta,\alpha}$ be a minimizer of $\f_{\delta,\alpha}$. 
		Let $i_* \in \{1,\ldots,N_{\mu} \}$ be the index of the smallest cell: $m_{i_*}=\min_i m_i$ ($i_*$ need not be unique). 
		Let $T$ be the optimal transport map defining $W_2(\ca_{Q_{\delta,\alpha}},\mu)$.
		Assume for all $\delta$ sufficiently small
		that the area of the `closest neighbour' to the smallest cell is no more than $\tfrac{19}{4}$ times the area of the smallest cell, i.e., $m_j \le \frac{19}{4} m_{i_*}$ for all $j$ such that
		\[
		\mathrm{dist}(z_{i_*},\partial T^{-1}(z_{i_*})) = \mathrm{dist}(z_{i_*},\partial T^{-1}(z_{j})) =: d_{i_* j}.
		\]
		Assume also that, for all $\delta$ sufficiently small, $z_{i_*}$ is not too close to the boundary:
		\[
		B_{d_{i_* j}}(z_{i_*}) \subset Q_{\delta,\alpha}.
		\]
		Then the asymptotic crystallization results \eqref{eq:limitpen} and \eqref{eq:limitcon} hold.
	\end{theorem}

First we prove an analogue of Lemma \ref{lem:inequality} for $\alpha$ close to $1$.

\begin{lemma}[Positivity of $h_{\alpha}$ for $\alpha$ near $1$]\label{lem:inequality_alpha1}
There exist\blue{s} $\varepsilon>0$ such that, for all $\alpha\in(1-\varepsilon,1)$ and all integers $n\geq3$, the following holds: if $h_\alpha(m_1,n) \ge 0$ for some $m_1\geq0$, then $h_\alpha(m,n)\ge 0$ for all $m\geq m_1$.
\end{lemma}

\begin{proof}
Define $h_1:(0,\infty)\times[3,\infty)\to\mathbb{R}$ by
\[
h_1(m,n):= \lim_{\alpha \to 1} h_\alpha(m,n) =  c_n m^2-c_6m -c_6 m \ln m -\kappa(n-6).
\]
\emph{Step 1.} We claim that $h_\alpha$ converges to $h_1$ locally uniformly in $m$ and uniformly in $n$ as $\alpha \to 1^-$.
Let $m \in (0,\infty)$, $n \ge 3$, $\alpha \in (0,1)$. Then
\[
\frac{1}{c_6} (h_\alpha(m,n)-h_1(m,n)) = \frac{m^{\alpha}-m+(1-\alpha)m \ln m}{1-\alpha} =: \frac{\theta({\alpha})}{1-\alpha}.
\]
Observe that $\theta(1)=\theta'(1)=0$ and $\theta''(\alpha)=(\ln m)^2 m^\alpha$. By Taylor's Theorem, there exists $\beta(\alpha) \in (\alpha,1)$ such that
\[
\frac{1}{c_6} |h_\alpha(m,n)-h_1(m,n)| = \frac12 (1-\alpha) |\theta''(\beta(\alpha))| = \frac12 (1-\alpha) (\ln m)^2 m^{\beta(\alpha)}.
\]
Fix $m_1, m_2 \in(0,\infty)$. Then 
\[
\lim_{\alpha \to 1^-}\sup_{\substack{m \in [m_1,m_2]\\n \in [3,\infty)}} |h_\alpha(m,n)-h_1(m,n)| 
\le \lim_{\alpha \to 1^-} \frac{c_6}{2} (1-\alpha) \sup_{m \in [m_1,m_2]} (\ln m)^2 m^{\beta(\alpha)} = 0
\]
as required.

\emph{Step 2.} Next we study the shape of the function  $m\mapsto h_1(m,n)$. Its derivative is
\begin{equation}
\label{eq:h1_m}
\partial_m h_1(m,n) = 2 c_n m-2c_6 -c_6\ln m,
\end{equation}
which is strictly convex in $m$ with $\lim_{m \to 0}\partial_m h_1(m,n)=+\infty$, $\lim_{m \to \infty}\partial_m
h_1(m,n)=+\infty$. Therefore 
$m\mapsto \partial_m h_1(m,n)$ has exactly one critical point, which is $m=\frac{c_6}{2 c_n}$. Since $n \mapsto c_n$ is decreasing, 
\[
\partial_m h_1 \left( \frac{c_6}{2 c_n},n \right) = -c_6 - c_6 \ln \left( \frac{c_6}{2 c_n} \right) \le 
-c_6 - c_6 \ln \left( \frac{c_6}{2 c_3} \right) < -0.019 < 0.
\]
Therefore $m\mapsto h_1(m,n)$ has exactly two critical points, the smallest critical point is a local maximum point, and the largest critical point is a local minimum point. Let $\widetilde{m}(n)$ denote the local minimum point. 
By the calculation above,
\begin{equation}\label{eq:secordcond}
\widetilde{m}(n)>\frac{c_6}{2c_n} > \frac{c_6}{2c_3}.
\end{equation}
Define $\varphi(n)$ to be the value of $h_1(\cdot,n)$ at its local minimum point:
\[
\varphi(n) := h_1(\widetilde{m}(n),n). 
\]
Observe that $\partial_m h_1(1,6)=0$. Therefore, by \eqref{eq:secordcond}, the local minimum point of $h_1$ is $\widetilde{m}(6)=1$ and $\varphi(6)=0$.

\emph{Step 3.} We claim that there exists a constant $c>0$ such that $\varphi(n)>c$ for all integers $n \geq3$, $n\ne 6$.
First we show that it is sufficient to prove this for the finite number of cases $n \in \{3,4,5,7\}$.

\emph{Step 3a: Reduction to the case $n \in \{3,4,5,7\}$.}
By equation \eqref{eq:h1_m}, for all $n \ge 3$,
\[
\partial_m h_1(1.05,n) = 2 c_n \cdot 1.05 -2c_6 -c_6\ln (1.05) \ge 2 c_\infty \cdot 1.05-2c_6 -c_6\ln (1.05) > 0.005 > 0.
\]
Therefore 
\begin{equation}
\label{eq:mtildeLB}
\widetilde{m}(n) < 1.05.
\end{equation}
For all $n \ge 7$,
\begin{align*}
\partial_n \varphi(n) 
& = 
\frac{\partial h_1}{\partial n} (\widetilde{m}(n),n)
& (\textrm{since }\partial_m h_1(\widetilde{m}(n),n)=0)
\\
& =
\widetilde{m}^2(n) \, \partial_n c_n - \kappa
\\
& \geq  \widetilde{m}^2(n) \, \partial_n c_n|_{n=7} - \kappa
& (\textrm{by Lemma } \ref{lem:cn})
\\
& \geq  1.05^2 \, \partial_n c_n|_{n=7} - \kappa
& \textrm{(by }\eqref{eq:mtildeLB})
\\
& >4 \times 10^{-4}.
\end{align*}
Therefore $\varphi$ is increasing for $n \ge 7$. Consequently, for all integers $n \geq3$, $n\ne 6$,
\begin{equation}
\label{eq:3457}
\varphi(n) \ge \min \{ \varphi(3), \varphi(4), \varphi(5), \varphi(7)\}.
\end{equation}

\emph{Step 3b: Proof for the case $n \in \{3,4,5,7\}$.}
We show that the right-hand side of \eqref{eq:3457} is positive.
Let
\[
m_1(n):=
\begin{cases}
0.66 & n=3,
\\
0.92 & n=4,
\\
0.981 & n=5,
\\
1.007 & n=7,
\end{cases}
\quad\quad
m_2(n):=
\begin{cases}
0.661 & n=3,
\\
0.93 & n=4,
\\
0.982 & n=5,
\\
1.0075 & n=7.
\end{cases}
\]
Then evaluating $\partial_m h_1$ gives
\[
\partial_m h_1(m_1(n),n)\leq
\begin{cases}
-7 \times 10^{-5} & n=3,\\
-7 \times 10^{-4} & n=4,\\
-1 \times 10^{-4} & n=5,\\
-5 \times 10^{-5} & n=7,
\end{cases}
\quad\quad
\partial_m h_1(m_2(n),n)\geq
\begin{cases}
6 \times 10^{-5} & n=3,\\
8 \times 10^{-4} & n=4,\\
4 \times 10^{-5} & n=5,\\
2 \times 10^{-5} & n=7.
\end{cases}
\]
Let $n \in \{3,4,5,7\}$.
By the Intermediate Value Theorem, the map $m \mapsto \partial h_1(m,n)$ has a root between $m_1(n)$ and $m_2(n)$. Moreover, since $\partial_m h_1(m_1(n),n) < 0$, we have bracketed the largest root $\widetilde{m}(n)$:  $m_1(n)<\widetilde{m}(n)<m_2(n)$.
Therefore
\begin{align*}
\varphi(n)\geq c_n m_1(n)^2 - c_6 m_2(n) - c_6 m_2(n) \ln m_2(n) - \kappa(n-6)
\geq
\begin{cases}
0.01 & n=3,\\
9\times 10^{-4} & n=4,\\
2\times 10^{-4} & n=5,\\
1\times 10^{-4} & n=7.
\end{cases}
\end{align*}
Combining this with \eqref{eq:3457} proves that $\varphi(n)>c >0$ for all integers $n \geq3$, $n\ne 6$.

\emph{Step 4.} By Step 1 of the proof of Lemma \ref{lem:inequality}, the function  $m\mapsto h_\alpha(m,n)$ has at most two critical points for all $\alpha \in (0,1)$, $n \ge 3$. If it has less than two critical points, then  Lemma \ref{lem:inequality_alpha1} follows immediately, so we just need to consider the case where it has
exactly two critical points. By Step 1 of the proof of Lemma \ref{lem:inequality}, the smallest critical point is a local maximum point and the largest critical point is a local minimum point. 

By Step 1,  $h_\alpha$ converges to $h_1$ uniformly on the interval $[c_6/(2c_3),1.05] \times [3,\infty)$. 
Observe that the local minimum point $\widetilde{m}(n)$ of $m \mapsto h_1(m,n)$ lies in the interval $[c_6/(2c_3),1.05]$ by \eqref{eq:secordcond} and \eqref{eq:mtildeLB}. Let $n$ be an integer, $n \ge 3$, $n \ne 6$.
By Step 3, $h_1(\widetilde{m}(n),n)>c>0$.
Therefore, for $\alpha$ sufficiently close to 1, the value of $m \mapsto h_\alpha(m,n)$ at its local minimum point is also positive by the uniform convergence.  This proves Lemma \ref{lem:inequality_alpha1} for the case $n \ge 3$, $n \ne 6$. 

Finally, the case $n=6$ follows immediately from Step 2  of the proof of Lemma \ref{lem:inequality}, where it was shown that
the value of $h_\alpha(m,6)$ at its local minimum point is $0$ for all $\alpha\in(0,1)$.
\end{proof}

An immediate consequence of Lemma \ref{lem:inequality_alpha1} is the following convexity inequality:
\begin{corollary}[Reduction to $n \in \{3,4,5\}$]
	\label{cor:reduction1}
	Let $n \ge 6$ be an integer. If $\alpha\in(0,1)$ is sufficiently close to $1$, then
	$h_{\alpha}(m,n)\geq0$ for all $m \ge 0$.
\end{corollary}
\begin{proof}
	Observe that $h_{\alpha}(0,n)=-\kappa(n-6) \ge 0$ for all $n \ge 6$. Therefore the result follows immediately from Lemma \ref{lem:inequality_alpha1}.
\end{proof}

Now we are in a position to prove the main theorem of the section.
\begin{proof}[Proof of Theorem \ref{thm:alpha01}]
\blue{By the monotonicity result (Lemma \ref{lem:monoCp}) we can assume that $\alpha \in (0,1)$.}
Let $\mu$ and $z_{i_*}$ be as in the statement of Theorem \ref{thm:alpha01}. To simplify the notation, let $\overline{z} = z_{i_*}$, $m=\mu(\{\overline{z}\})=\min \{\mu(\{z\}) : z \in \mathrm{supp}(\mu) \}$. Let $z \in \mathrm{supp}(\mu)$ satisfy 
\[
\mathrm{dist}(\overline{z},\partial T^{-1}(\overline{z})) = \mathrm{dist}(\overline{z},\partial T^{-1}(z)) =: d.
\]
Let $x$ belong to the edge
$e_{\overline{z}z} := \overline{T^{-1}(\overline{z})} \cap  \overline{T^{-1}(z)}$ and let $x$ satisfy
$|\overline{z}-x|=d$, which means that $x$ is the closet point on the boundary of the Laguerre cell 
$\overline{T^{-1}(\overline{z})}$ to $\overline{z}$.
Define $M=\mu(z)$. By assumption, $M \le  \frac{19}{4} m$.

\emph{Step 1: Upper bound on $|\overline{z}-z|$.} 
The point $x$ lies in the edge $e_{\overline{z}z}$. Therefore
by Lemma \ref{lem:cellspolygons} and Lemma \ref{lem:propoptcell}(i) we have
\[
|x-\overline{z}|^2 + \frac{\alpha}{1-\alpha} c_6 m^{\alpha-1}   = |x-z|^2 + \frac{\alpha}{1-\alpha} c_6 M^{\alpha-1}.
\]
Moreover, $x = \overline{z} + t (z-\overline{z})$ for some $t \in (0,1)$ (since $x$ is the closest point to $\overline{z}$ in the edge
$e_{\overline{z}z}$, which has normal $z-\overline{z}$). Solving for $t$ gives
\[
x  - \overline{z} =  \left[ \frac 12 + \frac{c_6}{2} \frac{\alpha}{1-\alpha} \frac{M^{\alpha-1}-m^{\alpha-1}}{|\overline{z}-z|^2} \right](z-\overline{z}).
\]
The term in square brackets is positive since $\overline{z}$ lies in its Laguerre cell  $\overline{T^{-1}(\overline{z})}$ (Lemma \ref{lem:propoptcell}(ii)). Moreover, 
the ball $B_d(\overline{z})$ is contained in the Laguerre cell $\overline{T^{-1}(\overline{z})}$ (since $e_{\overline{z}z}$ is the closest edge to $\overline{z}$). Therefore
\[
\pi d^2 \le m \quad \Longleftrightarrow \quad \pi  \left[ \frac 12 + \frac{c_6}{2} \frac{\alpha}{1-\alpha} \frac{M^{\alpha-1}-m^{\alpha-1}}{|\overline{z}-z|^2} \right]^2 |z-\overline{z}|^2 \le m\blue{.}
\]
Let $R=|z-\overline{z}|$. Then we can rewrite the inequality above as
\[
R^2 - 2\left( \frac{m}{\pi} \right)^{\frac 12} R + c_6 \frac{\alpha}{1-\alpha} (M^{\alpha-1}-m^{\alpha-1}) \le 0.
\]
Thus we get the upper bound
\begin{equation}
\label{eq:upperclosest}
R \le \left( \frac{m}{\pi} \right)^{\frac 12} + \left( \frac{m}{\pi} -   c_6 \frac{\alpha}{1-\alpha} (M^{\alpha-1}-m^{\alpha-1}) \right)^{\frac 12}.
\end{equation}

\emph{Step 2: Lower bound on $m$.} 
Let $\mathcal{C}^\mu$ be the partition associated to the minimizer $\mu$. Define a new partition $\mathcal{D}$ by replacing the cells $\overline{T^{-1}(\bar{z})}$ and
$\overline{T^{-1}(z)}$ with their union. By 
Lemma \ref{lem:mergsplit}
\begin{align*}
0 & \leq F(\mathcal{D}) - F(\mathcal{C}^\mu) 
\le \frac{c_6}{1-\alpha}\left( (m+M)^{\alpha}-m^{\alpha}-M^{\alpha} \right)
+ \frac{mM}{m+M} |\bar{z}-z|^2.
\end{align*}
By Young's inequality $(a+b)^2 \le 2 (a^2+b^2)$ and \eqref{eq:upperclosest} we obtain
\begin{align*}
0 & \le \frac{c_6}{1-\alpha}\left( (m+M)^{\alpha}-m^{\alpha}-M^{\alpha} \right)
+ 2 \frac{mM}{m+M} \left[
\frac{2}{\pi} m -  c_6 \frac{\alpha}{1-\alpha} (M^{\alpha-1}-m^{\alpha-1})
\right]
\\
& =
\frac{c_6}{1-\alpha} M^\alpha \left( \left( \frac mM + 1\right)^{\alpha}-\left( \frac mM \right)^{\alpha}-1 \right)
+ 2 \frac{1}{\frac mM + 1}  m^\alpha \left[
\frac{2}{\pi} m^{2-\alpha}-  c_6 \frac{\alpha}{1-\alpha} \left( \left( \frac Mm\right)^{\alpha-1}-1 \right)
\right].
\end{align*}
Define  $\lambda:=\frac{m}{M}$. By assumption, $\lambda \in [\frac{4}{19},1]$. Then rearranging the previous inequality gives the lower bound
\begin{equation}
\label{eq:mgeLB}
m ^{2-\alpha} \ge \overline{m}_\alpha(\lambda) :=
\frac{\pi c_6}{2}
\left[ 
\frac{\alpha}{1-\alpha} (\lambda^{1-\alpha}-1) 
-
\frac{(\lambda+1)((\lambda+1)^\alpha-\lambda^\alpha-1)}{2 (1-\alpha)
	\lambda^\alpha} \right].
\end{equation}
Note that \eqref{eq:mgeLB} gives a non-trivial lower bound on $m$ only when $\overline{m}_\alpha(\lambda)>0$.
For each $\lambda\in[\frac{4}{19},1]$ we have that
\[
\lim_{\alpha\to1}\overline{m}_\alpha(\lambda) 
= \overline{m}_1(\lambda)
:=\frac{\pi c_6}{2}  \left(  \frac{1-\lambda}{2} \ln \lambda + \frac{(\lambda+1)^2}{2\lambda} \ln (\lambda + 1) \right).
\]

\emph{Step 3: Lower bound on $\overline{m}_1(\lambda)$.} We claim that $\overline{m}_1$ is increasing. We have
\begin{align*}
\overline{m}_1'(\lambda)
& = \frac{\pi c_6}{2}  \left(
- \frac{\ln \lambda}{2} + \frac{1-\lambda}{2\lambda} + \frac{(\lambda^2-1)}{2 \lambda^2} \ln (\lambda +1) + \frac{\lambda+1}{2\lambda}
\right),
\\
\overline{m}_1''(\lambda) & = \frac{\pi c_6}{2} \lambda^{-2} \left( \frac{\ln(\lambda+1)}{\lambda} - \frac 32 
\right) =: \frac{\pi c_6}{2} \lambda^{-2} \phi(\lambda).
\end{align*}
Then
\begin{align*}
\phi'(\lambda) & = \frac{\lambda - (\lambda+1)\ln(\lambda+1)}{\lambda^2 (\lambda+1)} =: \frac{\psi(\lambda)}{\lambda^2 (\lambda+1)},
\\
\psi'(\lambda) & = - \ln (\lambda+1) < 0.
\end{align*}
Therefore, for all $\lambda \in [\frac{4}{19},1]$, $\psi(\lambda) \le \psi(\frac{4}{19}) < -0.02 <0$. Hence $\phi'(\lambda)<0$ and so 
\[
\phi(\lambda) \le \phi(\tfrac{4}{19}) < -0.59 < 0, \qquad
\overline{m}_1''(\lambda) =  \frac{\pi c_6}{2} \lambda^{-2} \phi(\lambda) < 0.
\]
We conclude that,  for all $\lambda \in [\frac{4}{19},1]$, $\overline{m}_1'(\lambda) \ge \overline{m}_1'(1)=1$, and so $\overline{m}_1$ is increasing, as claimed. 
Therefore 
\begin{equation}
\label{eq:4/19}
\inf_{\lambda \in [\frac{4}{19},1]} \overline{m}_1(\lambda) = \overline{m}_1(\tfrac{4}{19}) > 0.0125.
\end{equation}
This is essentially a lower bound on $m$ for $\alpha$ near $1$ (we will make this statement precise below) and it is much better than the lower bound you obtain using the method from Lemma \ref{lem:lowerboundmass} (but in Lemma \ref{lem:lowerboundmass} we did not make the ansatz that $\lambda \in [\frac{4}{19},1]$).

\emph{Step 4: Convexity inequality for $h_1$.} Numerical evaluation gives 
\begin{align*}
h_1(\overline{m}_1(\tfrac{4}{19}),3) & > 4.2 \times 10^{-3} > 0,
\\
h_1(\overline{m}_1(\tfrac{4}{19}),4) & > 5.0 \times 10^{-3} >0,
\\ 
h_1(\overline{m}_1(\tfrac{4}{19}),5) & > 5.9 \times 10^{-3} >0.
\end{align*}

\emph{Step 5: Uniform convergence of  $\overline{m}_\alpha$ to $\overline{m}_1$.}
Set 
\begin{align*}
\varphi_\alpha(\lambda) & := \frac{\alpha}{1-\alpha} (\lambda^{1-\alpha}-1) 
-
\frac{(\lambda+1)((\lambda+1)^\alpha-\lambda^\alpha-1)}{2 (1-\alpha) \lambda^\alpha},
\\
\varphi_1(\lambda) & := \frac{1-\lambda}{2} \ln \lambda + \frac{(\lambda+1)^2}{2\lambda} \ln (\lambda + 1).
\end{align*}
By Taylor's Theorem, for all $x \in \mathbb{R}$,
\[
|e^x-1-x| =\frac 12 x^2 e^\xi
\]
for some $\xi(x)$ between $0$ and $x$. Since $x \mapsto e^x$ is increasing we conclude that 
\begin{equation}\label{eq:exponential}
|e^x - 1 - x| \le \frac 12 x^2 \max \{ 1,e^x \}
\end{equation}
for all $x \in \mathbb{R}$. 
We estimate
\begin{align}
\nonumber
|\overline{m}_\alpha(\lambda) - \overline{m}_1(\lambda)| & = \frac{\pi c_6}{2} |\varphi_\alpha(\lambda) - \varphi_1(\lambda)| 
\\
\nonumber
& \leq \frac{\pi c_6}{2} \left| \frac{(\lambda+1)}{2 \lambda^\alpha}
	\frac{(\lambda+1)^\alpha-\lambda^\alpha-1}{1-\alpha}
	  + \frac{(\lambda+1)^2}{2\lambda}\ln(\lambda+1) - \frac{1+\lambda}{2}\ln \lambda \right|
	  \\
	  \label{eq:step5_4}
	  & \qquad +
	  \frac{\pi c_6}{2} \left|\frac{\alpha}{1-\alpha} (\lambda^{1-\alpha}-1) - \ln \lambda \right|.
\end{align}
For all $\lambda \in [\frac{4}{19},1]$, we estimate the second term on the right-hand side of \eqref{eq:step5_4} as follows:
\begin{align}\label{eq:step5_5}
\left|\frac{\alpha}{1-\alpha} (\lambda^{1-\alpha}-1) - \ln \lambda \right|
	&\stackrel{\phantom{\eqref{eq:exponential}}}{\leq} \alpha \left|\frac{1}{1-\alpha} (\lambda^{1-\alpha}-1) - \ln \lambda \right| + (1-\alpha)|\ln \lambda| \nonumber \\
&\stackrel{\phantom{\eqref{eq:exponential}}}{=} \alpha \left|\frac{e^{(1-\alpha)\ln\lambda} - 1 - (1-\alpha)\ln \lambda}{1-\alpha} \right| + (1-\alpha)|\ln \lambda| \nonumber \\
& \stackrel{\eqref{eq:exponential}}{\leq} \frac 12 \alpha  (1-\alpha) |\ln \lambda|^2
	+ (1-\alpha)|\ln \lambda| \nonumber \\
&\stackrel{\phantom{\eqref{eq:exponential}}}{\leq} C(1-\alpha),
\end{align}
where $C>0$ is a constant independent of $\alpha$ and $\lambda$. The existence of $C$ follows from the fact that $\lambda \in [\frac{4}{19},1]$, $\alpha\in(0,1)$. We estimate the first term on the right-hand side of \eqref{eq:step5_4} by
\begin{align}\label{eq:step5_6}
& \left| \frac{(\lambda+1)}{2 \lambda^\alpha}
\frac{(\lambda+1)^\alpha-\lambda^\alpha-1}{1-\alpha}
+ \frac{(\lambda+1)^2}{2\lambda}\ln(\lambda+1) - \frac{1+\lambda}{2}\ln \lambda \right| \nonumber \\
&\hspace{4cm}\leq \frac{\lambda+1}{2\lambda} \left|
	\frac{(\lambda+1)^\alpha-\lambda^\alpha-1}{1-\alpha}
	  + (\lambda+1)\ln(\lambda+1) - \lambda \ln \lambda \right| \nonumber \\
&\hspace{6cm}
	  + \frac{\lambda+1}{2} \left| \frac{1}{\lambda^\alpha} - \frac{1}{\lambda}\right| \, \left| \frac{(\lambda+1)^\alpha-\lambda^\alpha-1}{1-\alpha} \right|\blue{.}
\end{align}
Now we bound the first term on the right-hand side of \eqref{eq:step5_6}:
\begin{align}\label{eq:step5_7}
&\left| \frac{(\lambda+1)^\alpha-\lambda^\alpha-1}{1-\alpha}
	  + (\lambda+1)\ln(\lambda+1) - \lambda \ln \lambda \right| \nonumber \\
&\hspace{2cm} \stackrel{\phantom{\eqref{eq:exponential}}}{\leq} (\lambda+1)\left| \frac{(\lambda+1)^{\alpha-1} - 1
	- (\alpha-1)\ln(\lambda+1)}{1-\alpha} \right|
	+ \lambda \left| \frac{\lambda^{\alpha-1}-1-(\alpha-1)\ln\lambda}{1-\alpha} \right| \nonumber 
	\\
&\hspace{2cm}  \stackrel{\eqref{eq:exponential}}{\leq} \frac 12 (\lambda+1)  (1-\alpha)|\ln(\lambda+1)|^2
	+\frac 12 \lambda^\alpha(1-\alpha)|\ln\lambda|^2 \nonumber \\
&\hspace{2cm} \stackrel{\phantom{\eqref{eq:exponential}}}{\leq} C(1-\alpha)
\end{align}
for some constant $C>0$ since $\lambda \in [\frac{4}{19},1]$, $\alpha\in(0,1)$.
Next we estimate the second term on the right-hand side of \eqref{eq:step5_6}. By Taylor's Theorem there exists $\xi_1,\xi_2,\xi_3 \in (\alpha,1)$ such that
\begin{align}
\nonumber
& \left| \frac{1}{\lambda^\alpha} - \frac{1}{\lambda}\right| \, \left| \frac{(\lambda+1)^\alpha-\lambda^\alpha-1}{1-\alpha} \right|
\\
\nonumber
& = \frac{|\ln \lambda| \lambda^{\xi_1} (1-\alpha) }{\lambda^\alpha \lambda} \,
\frac{|\lambda + 1 + \ln(\lambda+1)(\lambda+1)^{\xi_2}(\alpha-1) - (\lambda + \ln(\lambda) \lambda^{\xi_3} (\alpha -1)) - 1|}{1-\alpha}
\\
\label{eq:step5_8}
& \le C(1-\alpha)
\end{align}
for all $\lambda \in  [\frac{4}{19},1]$, $\alpha \in (0,1)$, where $C>0$ is a constant independent of $\alpha$ and $\lambda$.
By using \eqref{eq:step5_4}--\eqref{eq:step5_8} we get the uniform convergence of $\overline{m}_\alpha$ to $\overline{m}_1$ on the interval $[\frac{4}{19},1]$.

\emph{Step 6: Convexity inequality for $h_\alpha$ for $\alpha$ sufficiently close to $1$.}
Recall from \eqref{eq:mgeLB} that 
\begin{equation}
\label{eq:Kamala}
\min \{\mu(\{z\}) : z \in \mathrm{supp}(\mu) \} = m \ge \left( \inf_{[\frac{4}{19},1]}\overline{m}_\alpha\right)^{\frac{1}{2-\alpha}} =: \eta_\alpha.
\end{equation}
Note that the positivity of 
	$\inf_{[\frac{4}{19},1]}\overline{m}_\alpha$ for $\alpha$ sufficiently close to $1$ follows from the positivity of 
	$\inf_{[\frac{4}{19},1]}\overline{m}_1$ and the uniform convergence of  $\overline{m}_\alpha$ to $\overline{m}_1$.
By equation \eqref{eq:4/19} we have
\begin{align*}
h_\alpha(\eta_\alpha,n) - h_1(\overline{m}_1(\tfrac{4}{19}),n)
 = h_\alpha(\eta_\alpha,n)
- h_1(\eta_\alpha,n)
+ h_1(\eta_\alpha,n)
- h_1(\textstyle \inf_{[\tfrac{4}{19},1]}\overline{m}_1,n).
\end{align*}
By the uniform convergence of $h_\alpha$ to $h_1$ (Step 1 of the proof of Lemma \ref{lem:inequality_alpha1}), the uniform convergence of  $\overline{m}_\alpha$ to $\overline{m}_1$ (which implies that $\eta_\alpha$ converges to $\inf_{[\tfrac{4}{19},1]}\overline{m}_1$), and the continuity of $h_1$, we find that
\begin{equation}
\label{eq:KimJong-un}
\lim_{\alpha \to 1} h_\alpha(\eta_\alpha,n) = h_1(\overline{m}_1(\tfrac{4}{19}),n) > 0
\end{equation}
for all $n \in \{3,4,5\}$ by Step 4.
  By  \eqref{eq:Kamala}, \eqref{eq:KimJong-un}, Lemma \ref{lem:inequality_alpha1}, and Corollary \ref{cor:reduction1} we conclude that
  $h_\alpha(\mu(\{z\}),n) \ge 0$ for all $z \in \mathrm{supp}(\mu)$ and all integers $n \ge 3$, provided that
 $\alpha$ is sufficiently close to $1$.
 
\emph{Step 7: Conclusion.} By using the convexity inequality  $h_\alpha(\mu(\{z\}),n) \ge 0$ from Step 6, we can conclude the proof using the same argument that we used to prove Theorem \ref{thm:LBCp}.
\end{proof}


\section{Proof of Theorem \ref{thm:stability}}
\label{sec:stability}
\blue{
The proof is similar in spirit to the analogous result for the case $\alpha=0.5$ from \cite[Theorem 3]{BouPelTheil}, although we have to do some extra work to take care of particles near the boundary.
The main ingredient is the following stability result due to G.~Fejes T\'oth \cite{Fej_Stability}, which roughly states that if $\mu$
is a discrete measure on a convex $n$-gon $\Omega$, with $n \le 6$, such that the rescaled quantization error
$\frac{N_\mu}{|\Omega|^2} W_2^2(\ca_\Omega,\mu)$ is close to the asymptotically optimal value of $c_6$, then the support of $\mu$ is close to a triangular lattice. 

\begin{theorem}[Stability Theorem of G.~Fejes T\'oth]
\label{thm:FT_stability}
Let $\Omega \subset \mathbb{R}^2$ be a convex polygon with at most six sides. Let $\{z_i\}_{i=1}^N$ be a set of $N$ distinct points in $\Omega$ and let $\{ V_i \}_{i=1}^N$ be the Voronoi tessellation of $\Omega$ generated by $\{z_i\}_{i=1}^N$, i.e., 
\[
V_i = \{ z \in \Omega : |z-z_i| \le |z-z_j| \; \forall j \in \{1,\ldots,N\}\}.
\]
Define the defect of 
the configuration $\{z_i\}_{i=1}^N$ by
\[
\hat{\varepsilon}(\{z_i\}_{i=1}^N):= \frac{N}{|\Omega|^2} \sum_{i=1}^{N} \int_{V_i} |z-z_i|^2\, \mathrm{d}z - c_6 .
\]
There exist $\varepsilon_0 >0$ and $c>0$ such that the following hold.
If $\varepsilon\in(0,\varepsilon_0)$ and $\{z_i\}_{i=1}^N$ satisfy
\[
\hat{\varepsilon}(\{z_i\}_{i=1}^N)\leq \varepsilon,
\]
then, with the possible exception of at most $N c  \varepsilon^{1/3}$ indices $i\in\{1,\dots,N\}$, the following hold:
\begin{itemize}
\item[(i)] $V_i$ is a hexagon;
\item[(ii)] the distance between $z_i$ and each vertex of $V_i$ is between $(1 \pm \varepsilon^{1/3}) \sqrt{\frac{|\Omega|}{N}} \sqrt{\frac{\phantom{|}2\phantom{|}}{3 \sqrt{3}}}$;
\item[(iii)] the distance between $z_i$ and each edge of $V_i$ is between $(1 \pm \varepsilon^{1/3}) \sqrt{\frac{|\Omega|}{N}} \sqrt{\frac{\phantom{|}1\phantom{|}}{2 \sqrt{3}}}$. 
\end{itemize}
\end{theorem}

To appreciate the geometric significance of this result, note that for a \emph{regular} hexagon of area $|\Omega|/N$, the distance between the centre of the hexagon and each vertex is $\sqrt{\frac{|\Omega|}{N}}  
\sqrt{\frac{\phantom{|}2\phantom{|}}{3 \sqrt{3}}}$, and the distance between the centre of the hexagon and each edge is $\sqrt{\frac{|\Omega|}{N}} \sqrt{\frac{\phantom{|}1\phantom{|}}{2 \sqrt{3}}}$.

\begin{proof}
This was proved in \cite{Fej_Stability} in a much more general setting. A similar result was proved by Gruber in \cite{Gru_Riem}. We make some quick remarks about how the version stated here can be read off from \cite{Fej_Stability}. In the notation of \cite[p.~123]{Fej_Stability}, we have $f(t)=t^2$ and
\begin{gather*}
r(\mu,N)=(1+\mu) r(H_N), \qquad r(H_N) = \sqrt{\frac{|\Omega|}{N}} r(H_1), \qquad r(H_1)=\sqrt{\frac{1}{2 \sqrt{3}}},
\\
R(\mu,N)=(1-\mu) R(H_N), \qquad R(H_N) = \sqrt{\frac{|\Omega|}{N}} R(H_1),  \qquad R(H_1)=\sqrt{\frac{2}{3 \sqrt{3}}},  
\\
h(\mu,N) = |r(\mu,N)^2 - R(\mu,N)^2| = b(\mu) \frac{|\Omega|}{N}, 
\qquad
 b(\mu)=\frac{1}{\sqrt{3}} \left( \frac 16 - \frac 73 \mu + \frac 16 \mu^2 \right).
\end{gather*}
Since $f$ is strictly increasing, or by a direct computation, it is easy to see that the condition $h(\mu,N) \ne 0$ stated in \cite[equation (3)]{Fej_Stability} holds for all $\mu \in (0,(2-\sqrt{3})^2)$. Fix any $\mu \in (0,(2-\sqrt{3})^2)$. Then by \cite[Theorem p.~213]{Fej_Stability}, there exist $c=c(\mu)>0$ and $\varepsilon(\mu)>0$ such that if $0< \tilde{\varepsilon} < \varepsilon(\mu)$ and 
\[
 \sum_{i=1}^{N} \int_{V_i} |z-z_i|^2\, \mathrm{d}z - N  \left( \frac{|\Omega|}{N} \right)^2 c_6 \le \tilde{\varepsilon} \, b(\mu) \frac{|\Omega|^2}{N},
\]
then statements (i)-(iii) of Theorem \ref{thm:FT_stability} hold. Defining $\varepsilon=\tilde{\varepsilon} \, b(\mu)$ and $\varepsilon_0 = \varepsilon(\mu) b(\mu)$  completes the proof.
\end{proof}
}

\blue{
The other key ingredient is an improved version of the convexity inequality \eqref{eq:ineq} for sufficiently large masses.

\begin{lemma}[Improved convexity inequality]
\label{lem:sharp_convexity}
There exists a constant $\xi>0$ such that
\[
\frac{c_6}{1-\overline{\alpha}}m^{\overline{\alpha}} + c_n m^2
	-c_6\left(\frac{2-\overline{\alpha}}{1-\overline{\alpha}}\right)m-\kappa(n-6) \geq \xi (m-1)^2
\]
for every integer $n\geq 3$ and every $ m\geq \overline{m}$, where $\overline{m}>0$ is given in Lemma \ref{lem:lowerboundmass}.
\end{lemma}

\begin{proof}
The constant $\xi>0$ will be chosen as the minimum of several quantities that we are now going to introduce. Recall that
\[
h_{\overline{\alpha}}(m,n) = \frac{c_6}{1-\overline{\alpha}}m^{\overline{\alpha}} + c_n m^2
	-c_6\left(\frac{2-\overline{\alpha}}{1-\overline{\alpha}}\right)m-\kappa(n-6).
\]
Recall also that $n\mapsto c_n$ is decreasing with $\lim_{n\to\infty}c_n=c_\infty>0$ (see Lemma \ref{lem:cn}) and $\kappa < 0$. Therefore, for all $n \ge 3$,
\begin{align*}
h_{\overline{\alpha}}(m,n) & \ge \frac{c_6}{1-\overline{\alpha}}m^{\overline{\alpha}} + c_\infty m^2
	-c_6\left(\frac{2-\overline{\alpha}}{1-\overline{\alpha}}\right)m + 3 \kappa
\\
& = \frac{c_\infty}{2} (m-1)^2 + \left[ \frac{c_\infty}{2} m^2 
+ \left( c_\infty - c_6\left(\frac{2-\overline{\alpha}}{1-\overline{\alpha}}\right) \right) m +
\frac{c_6}{1-\overline{\alpha}}m^{\overline{\alpha}}
 + 3 \kappa - \frac{c_\infty}{2}
 \right].
\end{align*}
The expression in the square brackets is positive for $m$ sufficiently large. Therefore there exists a constant $M>0$ (independent of $n$) such that
\begin{equation}
\label{eq:Case1}
h_{\overline{\alpha}}(m,n)  \ge \frac{c_\infty}{2} (m-1)^2
\end{equation}
for all $n \ge 3$ and all $m \ge M$. Without loss of generality we can take $M>2$.

Next we treat the case $m \in [\overline{m},M]$, $n \ne 6$.
Recall from the proof of Lemma \ref{lem:inequality}, Step 1, that 
the map $m\mapsto h_{\overline{\alpha}}(m,n)$ has exactly two critical points. The smallest critical point is a local maximum point and the largest critical point is a local minimum point, denoted by $\widetilde{m}(\overline{\alpha},n)$. 
 Note that $\widetilde{m}(\overline{\alpha},n) \le 3/2 < M$ by Steps 3 and 4 of the proof of Lemma \ref{lem:inequality}.
Therefore the global minimum of $m \mapsto 
h_{\overline{\alpha}}(m,n)$ in the interval $[\overline{m},M]$ occurs at either 
 $\widetilde{m}(\overline{\alpha},n)$ or $\overline{m}$. 
 Define
 \[
p_1 := \min_{\substack{n\geq3\\n\neq 6}} \, h_{\overline{\alpha}}(\widetilde{m}(\overline{\alpha},n),n), 
\qquad \qquad
p_2 := \min_{\substack{n\geq3\\n\neq 6}} \, h_{\overline{\alpha}}(\overline{m},n).
 \]
 Observe that $p_1>0$ by Steps 3 and 4 of the proof of Lemma \ref{lem:inequality}, and $p_2>0$  by Corollary \ref{cor:reduction} and the proof of Corollary \ref{Cor:CI}.
 Putting everything together gives the following for all $m \in [\overline{m},M]$, $n \ge 3$, $n \ne 6$:
 \begin{equation}
 \label{eq:Case2}
  h_{\overline{\alpha}}(m,n)
  \ge \min \left\{ p_1, p_2  \right\} \ge \frac{\min \left\{ p_1, p_2  \right\}}{(M-1)^2} (m-1)^2
 \end{equation}
since $m \in [\overline{m},M]$ and $(M-1)^2 > 1 > (\overline{m}-1)^2$. 

Next we treat the case  $n = 6$, $m \in [\overline{m},M]$.
Since $m\mapsto h_{\overline{\alpha}}(m,6)$ is of class $C^2$, $h_{\overline{\alpha}}(1,6)=\partial_m h_{\overline{\alpha}}(1,6)=0$, and $\partial^2_{mm}h_{\overline{\alpha}}(1,6)>0$, there exist $r \in (0,1/2)$ and $l>0$ such that
\begin{equation}\label{eq:Case3}
h_{\overline{\alpha}}(m,6)\geq l(m-1)^2
\end{equation}
for all $m\in[1-r,1+r]$.

Finally, we treat the case  $n = 6$, $m \in [\overline{m},1-r] \cup [1+r,M]$.
The function $m \mapsto h_{\overline{\alpha}}(m,6)$ has no local minima in this interval (by Lemma \ref{lem:inequality}, Step 1). Therefore
\[
h_{\overline{\alpha}}(m,6)\geq
\min \left\{ 
h_{\overline{\alpha}}(\overline{m},6), h_{\overline{\alpha}}(1-r,6),
h_{\overline{\alpha}}(1+r,6), h_{\overline{\alpha}}(M,6)
\right\}
=:p > 0.
\]
Hence, for all  $m \in [\overline{m},1-r] \cup [1+r,M]$,
\begin{equation}
\label{eq:Case4}
h_{\overline{\alpha}}(m,6)\geq  \frac{p}{(M-1)^2} (m-1)^2.
\end{equation}
Define
\begin{align*}
\xi&:=\min\left\{\frac{c_\infty}{2},  \frac{\min \left\{ p_1, p_2  \right\}}{(M-1)^2} , l, \frac{p}{(M-1)^2} \right\}.
\end{align*}
From \eqref{eq:Case1}, \eqref{eq:Case2}, \eqref{eq:Case3}, \eqref{eq:Case4} we conclude that
\[
h_{\overline{\alpha}}(m,n) \geq \xi(m-1)^2
\]
for all $m\geq\overline{m}$ and all $n\geq3$, as desired.
\end{proof}
}

\blue{
We are now in position to prove Theorem \ref{thm:stability}.
The idea is essentially to bound the defect $\hat{\varepsilon}$ from Theorem \ref{thm:FT_stability} by the defect $\mathrm{d}$ from Theorem \ref{thm:stability}.
}

\blue{
\begin{proof}[Proof of Theorem \ref{thm:stability}]
\emph{Step 1.} In this step we rescale the energy.
Let $\mu_\delta =\sum_{i=1}^{N_\delta} \widetilde{m}_i\delta_{\widetilde{z}_i}\in\Pd$ satisfy the hypotheses of Theorem \ref{thm:stability}. Define 
\[
m_i := V_{\delta,\overline{\alpha}}\widetilde{m}_i, \qquad 
z_i := V_{\delta,\overline{\alpha}}^{1/2} \widetilde{z}_i, \qquad
\Omega_{\delta,\overline{\alpha}} := V_{\delta,\overline{\alpha}}^{1/2} \Omega, \qquad
\mu := \sum_{i=1}^{N_\delta} m_i \delta_{z_i}.
\]
In analogy with \eqref{eq:Fda}, define
\[
\f_{\delta,\overline{\alpha}}(\mu):= \frac{c_6}{1-\overline{\alpha}} \sum_{i=1}^{N_\delta} m_i^{\overline{\alpha}} + W_2^2(\ca_{\Omega_{\delta,\overline{\alpha}}},\mu).
\]
As in Remark \ref{rem:rescaled}, it is easy to check that the defect $\mathrm{d}(\mu_\delta)$ can be rewritten as
\begin{equation}
\label{eq:Countryside}
\mathrm{d}(\mu_\delta) = V_{\delta,\overline{\alpha}}^{-1} \f_{\delta,\overline{\alpha}}(\mu) - \frac{2-\overline{\alpha}}{1-\overline{\alpha}} \, c_6.
\end{equation}

\emph{Step 2.} In this step we estimate the number of small particles: $\# \{i \in \{1,\ldots,N_\delta\} : m_i < \overline{m} \}$.  
Despite the fact that Lemma \ref{lem:lowerboundmass} was proved for the case where $\Omega$ is a square, the same argument can be applied to the case where $\Omega$ is any convex polygon with at most six sides (actually, to any Lipschitz domain). The only changes will be to the constant $D_0$ and to the lower bound $m_{\mathrm{b}}$ on the mass of particles close to the boundary (to be precise, for those particles within distance $4$ of the boundary of the rescaled domain). The lower bound $\overline{m}$ on the mass of particles far from the boundary remains the same. This is the only constant whose specific value matters for us. Therefore we will denote the other two constants by$D_0$ and $m_{\mathrm{b}}$ also in this case. Define
\[
\mathcal{K}  := \{ i \in \{1,\ldots,N_\delta \} : m_i < \overline{m} \}, \qquad
\mathcal{K}^c  := \{ i \in \{1,\ldots,N_\delta \} : m_i\geq \overline{m} \}.
\]
By Lemma \ref{lem:lowerboundmass}(i), $\mathcal{K} \subseteq \mathcal{J}$, where $\mathcal{J}$ was defined in equation \eqref{eq:IJ}.
Define
\[
U = \left\{ x + y : x \in \partial \Omega_{\delta,\overline{\alpha}}, \; y \in B_{4+D_0}(x) \right\}  \cap \Omega_{\delta,\overline{\alpha}}.
\]
Similarly to the proof of \eqref{eq:Borat},
\[
\# \mathcal{K} \le \# \mathcal{J} < \frac{|U|}{m_{\mathrm{b}}}.  
\]
Fix any $\eta > 0$. If $\delta$ is sufficiently small, then
\[
|U| \le V_{\delta,\overline{\alpha}}^{1/2} (4+D_0) (\mathcal{H}^1(\partial \Omega) + \eta)
\]
since $\Omega$ is convex polygon and the Minkowski content of $\partial \Omega$ equals $\mathcal{H}^1(\partial \Omega)$ \cite[Theorem 2.106]{AFP}.
Therefore
\begin{equation}\label{eq:asympt_good_cells}
\frac{\# \mathcal{K}}{V_{\delta,\overline{\alpha}}} \le
	V_{\delta,\overline{\alpha}}^{-1/2} \frac{(4+D_0) (\mathcal{H}^1(\partial \Omega) + \eta)}{m_{\mathrm{b}}}.
\end{equation}

\emph{Step 3.} 
Let $\xi>0$ be the constant given by Lemma \ref{lem:sharp_convexity}. Define
\[
\widetilde{\beta}_2 := \frac{2-\overline{\alpha}}{1-\overline{\alpha}} c_6 \overline{m} - 3\kappa + \xi > 0.
\]
In this step we prove the following lower bound on the defect:
\begin{equation}\label{eq:step1}
\mathrm{d}(\mu_\delta)\geq \frac{\xi}{V_{\delta,\overline{\alpha}}} \sum_{i=1}^{N_\delta} (m_i-1)^2
	-\widetilde{\beta}_2 \frac{\#\mathcal{K}}{V_{\delta,\overline{\alpha}}}.
\end{equation}
We estimate
\begin{align*}
F_{\delta,\overline{\alpha}}(\mu) &\geq \sum_{i=1}^{N_\delta} g_{\overline{\alpha}}(m_i,n_i)
 & (\textrm{Lemma }\ref{lem:Fejes})
 \\
 & 
 \ge \sum_{i\in\mathcal{K}^c} g_{\overline{\alpha}}(m_i,n_i) 
& (g_{\overline{\alpha}} \ge 0)
 \\ 
&\geq \sum_{i\in\mathcal{K}^c} \left[c_6\left( \frac{2-\overline{\alpha}}{1-\overline{\alpha}} \right)m_i
	+\kappa(n_i-6) +\xi (m_i-1)^2 \right] 
	& (\textrm{Lemma } \ref{lem:sharp_convexity})
	\\
&= \frac{2-\overline{\alpha}}{1-\overline{\alpha}}c_6 \sum_{i=1}^{N_\delta} m_i
	+ \kappa \sum_{i=1}^{N_\delta}(n_i-6)
	+\xi \sum_{i=1}^{N_\delta} (m_i-1)^2 \\
&\hspace{1cm} -\frac{2-\overline{\alpha}}{1-\overline{\alpha}}c_6 \sum_{i\in\mathcal{K}} m_i
	- \kappa \sum_{i\in\mathcal{K}}(n_i-6)
	-\xi \sum_{i\in\mathcal{K}} (m_i-1)^2 \\
&\geq \frac{2-\overline{\alpha}}{1-\overline{\alpha}}c_6 V_{\delta,\overline{\alpha}}
	+\xi \sum_{i=1}^{N_\delta} (m_i-1)^2 \\
&\hspace{1cm} -\frac{2-\overline{\alpha}}{1-\overline{\alpha}}c_6 \sum_{i\in\mathcal{K}} m_i
	- \kappa \sum_{i\in\mathcal{K}}(n_i-6)
	-\xi \sum_{i\in\mathcal{K}} (m_i-1)^2 ,
\end{align*}
where in the last step we used the fact that $\kappa<0$ together with Lemma \ref{lemma:PCP}.
Combining this estimate and \eqref{eq:Countryside} gives
\begin{align*}
\frac{\xi}{V_{\delta,\overline{\alpha}}} \sum_{i=1}^{N_\delta} (m_i-1)^2
	&\leq 
	\mathrm{d}(\mu_\delta) +\frac{1}{V_{\delta,\overline{\alpha}}} \left[
	\frac{2-\overline{\alpha}}{1-\overline{\alpha}}c_6 \sum_{i\in\mathcal{K}} m_i
	+ \kappa \sum_{i\in\mathcal{K}}(n_i-6)
	+\xi \sum_{i\in\mathcal{K}} (m_i-1)^2 \right] \\
&\leq \mathrm{d}(\mu_\delta) + \frac{\#\mathcal{K}}{V_{\delta,\overline{\alpha}}} \left[
	\frac{2-\overline{\alpha}}{1-\overline{\alpha}}c_6 \overline{m}
	 - 3\kappa + \xi \right],
\end{align*}
where in the last step we used the fact that $m_i \in (0,\overline{m}]$ for each $i\in\mathcal{K}$.
This proves \eqref{eq:step1}.

\emph{Step 4.} In this step we prove Theorem \ref{thm:stability}(a). We have
\begin{align}
\nonumber
\left| \frac{V_{\delta,\overline{\alpha}}}{N_\delta} + \frac{N_\mu}{V_{\delta,\overline{\alpha}}} - 2 \right|
	&=\frac{V_{\delta,\overline{\alpha}}}{N_\delta}\left(
		\frac{V_{\delta,\overline{\alpha}} - N_\delta}{V_{\delta,\overline{\alpha}}} \right)^2
=\frac{V_{\delta,\overline{\alpha}}}{N_\delta}\left( \frac{1}{V_{\delta,\overline{\alpha}}} \sum_{i=1}^{N_\delta} (m_i-1)\right)^2 
\\
\nonumber
&
\leq \frac{1}{V_{\delta,\overline{\alpha}}}\sum_{i=1}^{N_\delta} (m_i-1)^2 
\\
\label{eq:KungFury}
&\leq \frac{1}{\xi} \left( \mathrm{d}(\mu_\delta) + \widetilde{\beta}_2 \frac{\#\mathcal{K}}{V_{\delta,\overline{\alpha}}} \right),
\end{align}
where in the penultimate step we used Jensen's inequality and in the last step we used \eqref{eq:step1}. Therefore $\lim_{\delta \to 0} V_{\delta,\overline{\alpha}} / N_\delta = 1$ by \eqref{eq:asympt_good_cells} and since 
$\lim_{\delta \to 0} \mathrm{d}(\mu_\delta)=0$ by Theorem \ref{thm:pen}. This proves Theorem \ref{thm:stability}(a).
For the future, we record that if $\delta$ is sufficiently small, then we can read off from \eqref{eq:KungFury} that 
\begin{equation}\label{eq:stability_claim2}
\frac{N_\mu}{V_{\delta,\overline{\alpha}}}\leq 3.
\end{equation}

\emph{Step 5.}
In this step we prove the technical estimate
\begin{equation}\label{eq:stability_claim1}
\sum_{i=1}^{N_\delta} m_i^{\overline{\alpha}} \geq N(1-\overline{\alpha}) + \overline{\alpha} V_{\delta,\overline{\alpha}} - (1-\overline{\alpha})\frac{V_{\delta,\overline{\alpha}}}{\xi}\mathrm{d}(\mu_\delta) - (1-\overline{\alpha}) \widetilde{\beta}_2 \frac{\#\mathcal{K}}{\xi}.
\end{equation}
Define $\phi(x)=x^{\overline{\alpha}}$, $x>0$. 
Let $q$ be the unique quadratic polynomial such that $q(0)=\phi(0)$, $q(1)=\phi(1)$, $q'(1)=\phi'(1)$:
\[
q(x):= 1+ \overline{\alpha}(x-1) + (\overline{\alpha}-1)(x-1)^2.
\]
Let $\psi = \phi - q$.
It is easy to check that $\psi'''(x)>0$ for all $x>0$ and hence $\psi'$ is convex with $\psi'(1)=0$, $\psi''(1) = \overline{\alpha}^2 - 3 \alpha + 2 > 0$, $\lim_{x \to 0+} \psi'(x)=+\infty$, $\lim_{x \to + \infty} \psi'(x)=+\infty$. 
Therefore $\psi'(x)=0$ for only two points $x>0$, one of which is $x=1$ and the other is $x^* \in (0,1)$. Moreover, $\psi(0)=\psi(1)=0$, $\lim_{x \to \infty}\psi(x)=+\infty$, and $\psi'(x)<0$ for $x \in (x^*,1)$. Therefore $\psi(x)\ge 0 $ for all $x \ge 0$.
Consequently
\[
\sum_{i=1}^{N_\delta}  m_i^{\overline{\alpha}}
	\geq \sum_{i=1}^{N_\delta} q(m_i)
	= N_\delta + \overline{\alpha}(V_{\delta,\overline{\alpha}} - N_\delta)
		- (1-\overline{\alpha})\sum_{i=1}^{N_\delta} (m_i-1)^2.
\]
The desired inequality \eqref{eq:stability_claim1} follows by \eqref{eq:step1}.

\emph{Step 6.} Finally, we bound the defect $\hat{\varepsilon}(\{ z_i \}_{i=1}^{N_\delta})$ from Theorem \ref{thm:FT_stability}.
Let $T$ be the optimal transport map defining $W_2(\ca_{\Omega_{\delta,\overline{\alpha}}},\mu)$.
For $\delta>0$ sufficiently small we have
\begin{align}
\nonumber
\hat{\varepsilon}(\{ z_i \}_{i=1}^{N_\delta}) 
& \stackrel{\phantom{\eqref{eq:stability_claim1}}}{=} \frac{N_\delta}{V_{\delta,\overline{\alpha}}^2}
	\sum_{i=1}^{N_\delta} \int_{V_i} |z-z_i|^2 \, \mathrm{d} z - c_6
	= \frac{N_\delta}{V_{\delta,\overline{\alpha}}^2} \int_{\Omega_{\delta,\overline{\alpha}}} \min_i |z-z_i|^2 \, \mathrm{d} z - c_6 
	\\
	\nonumber
& \stackrel{\phantom{\eqref{eq:stability_claim1}}}{\le} \frac{N_\delta}{V_{\delta,\overline{\alpha}}^2} \sum_{i=1}^{N_\delta} \int_{T^{-1}(x_i)} |z-z_i|^2 \, \mathrm{d} z - c_6 
= \frac{N_\delta}{V_{\delta,\overline{\alpha}}^2} W_2(\ca_{\Omega_{\delta,\overline{\alpha}}},\mu) - c_6
\\
\nonumber
& \stackrel{\phantom{\eqref{eq:stability_claim1}}}{=} \frac{N_\delta}{V_{\delta,\overline{\alpha}}} \mathrm{d}(\mu_\delta)
	- \frac{N_\delta}{V_{\delta,\overline{\alpha}}^2} \frac{c_6}{1-\overline{\alpha}}\sum_{i=1}^{N_\delta}  m_i^{\overline{\alpha}}
	+\frac{N_\delta}{V_{\delta,\overline{\alpha}}}\frac{2-\overline{\alpha}}{1-\overline{\alpha}}c_6 - c_6 
\\
\nonumber
& \stackrel{\eqref{eq:stability_claim1}}{\leq} \frac{N_\delta}{V_{\delta,\overline{\alpha}}}\left( 1 + \frac{c_6}{\xi} \right)\mathrm{d}(\mu_\delta)
	+ \frac{N_\delta}{V_{\delta,\overline{\alpha}}^2} \frac{c_6 \widetilde{\beta}_2}
	{\xi} \#\mathcal{K} 
	- c_6 \underbrace{\left[
\frac{N_\delta^2}{V_{\delta,\overline{\alpha}}^2} + \frac{N_\delta}{V_{\delta,\overline{\alpha}}}  \frac{\overline{\alpha}}{1-\overline{\alpha}} - 
\frac{N_\delta}{V_{\delta,\overline{\alpha}}} 
\frac{2-\overline{\alpha}}{1-\overline{\alpha}} + 1\right]}_{= \left( \frac{V_{\delta,\overline{\alpha}}-N_\delta}{V_{\delta,\overline{\alpha}}} \right)^2}
	\\
	\nonumber
& \stackrel{\eqref{eq:stability_claim2}}{\leq} 3\left( 1 + \frac{c_6}{\xi} \right)\mathrm{d}(\mu_\delta)
	+ 3\frac{c_6 \widetilde{\beta}_2}{\xi} \frac{\#\mathcal{K}}{V_{\delta,\overline{\alpha}}}
	\\
	\label{eq:TheFountain}
& \stackrel{\eqref{eq:asympt_good_cells}}{\leq} 3\left( 1 + \frac{c_6}{\xi} \right)\mathrm{d}(\mu_\delta)
	+ 3\frac{c_6 \widetilde{\beta}_2}{\xi} 
	V_{\delta,\overline{\alpha}}^{-1/2} \frac{(4+D_0) (\mathcal{H}^1(\partial \Omega) + \eta)}{m_{\mathrm{b}}}.
\end{align}
Let $\varepsilon_0>0$ be the constant given by Theorem \ref{thm:FT_stability}. Define
\[
\beta_1 = 3\left( 1 + \frac{c_6}{\xi} \right), \qquad \beta_2 = 3\frac{c_6 \widetilde{\beta}_2}{\xi} \frac{(4+D_0) (\mathcal{H}^1(\partial \Omega) + \eta)}{m_{\mathrm{b}}}.
\]
By \eqref{eq:TheFountain},
if $\delta>0$ is sufficiently small, and if $\varepsilon\in(0,\varepsilon_0)$ and $\mu_\delta$ satisfy
\begin{equation*}
\beta_1 \mathrm{d}(\mu_\delta) + \beta_2 V_{\delta,\alpha}^{-1/2} \leq \varepsilon,
\end{equation*}
then
\[
\hat{\varepsilon}(\{ z_i \}_{i=1}^{N_\delta})  \le \varepsilon.
\]
Applying Theorem \ref{thm:FT_stability}  completes the proof.
\end{proof}
}


\subsection*{Acknowledgements}
The authors would like to thank Giovanni Leoni, Mark Peletier, Lucia Scardia and Florian Theil for useful discussions,
and Steven Roper for producing Figure \ref{fig:numerics}, Figure \ref{fig:numerics2} and Table \ref{Table1}.
This research was supported by the UK Engineering and Physical Sciences Research Council (EPSRC) via the grant EP/R013527/2 Designer Microstructure via Optimal Transport Theory.

\bibliographystyle{siam}
\bibliography{bib}

\end{document}